\documentclass[12pt,english]{extarticle}

\usepackage[utf8]{inputenc}

\usepackage[a4paper]{geometry}
\usepackage{fancyhdr}
 \usepackage{amsmath,amssymb}

\makeatletter

\newcommand{\grad}{\mathbb{G}}

\renewcommand*\env@cases[1][1.2]{%
  \let\@ifnextchar\new@ifnextchar
  \left\lbrace
  \def\arraystretch{#1}%
  \array{@{\,}c@{\ }l@{}}%
}
\makeatother

\numberwithin{equation}{section}
 
\newcommand{\OmegaT}{{\Omega_T}}
\newcommand{\LB}{\lambda_{\mathrm B}}

\usepackage{authblk}
\usepackage[dvipsnames]{xcolor}
\usepackage{latexsym,amssymb}
\usepackage{amsmath,amsthm,amsxtra,amsbsy,accents}
\usepackage{mathtools}
\usepackage{mathrsfs}
\usepackage{bm}
\usepackage{hyperref}

\DeclareMathAlphabet{\mathup}{OT1}{\familydefault}{m}{n}

\newcommand{\eps}{\varepsilon}

\DeclarePairedDelimiter{\abs}{\lvert}{\rvert}
\DeclarePairedDelimiter{\norm}{\lVert}{\rVert}
\DeclarePairedDelimiter{\bra}{(}{)}
\DeclarePairedDelimiter{\pra}{[}{]}
\DeclarePairedDelimiter{\set}{\{}{\}}

\makeatletter
\newcommand{\customlabel}[2]{%
   \protected@write \@auxout {}{\string \newlabel {#1}{{#2}{\thepage}{#2}{#1}{}} }%
   \hypertarget{#1}{}
}
\makeatother

\numberwithin{figure}{section}

\def\calD{{\mathcal D}} \def\calE{{\mathcal E}} 
  
  \def\calL{{\mathcal L}}
\def\calM{{\mathcal M}}  
\def\calP{{\mathcal P}}  \def\calR{{\mathcal R}}
\def\calS{{\mathcal S}}  
  \def\calX{{\mathcal X}}
 
\def\rmd{{\mathrm d}}

  \def\rmx{{\mathrm x}}
 
  \def\rmC{{\mathrm C}}
\def\rmD{{\mathrm D}}  
 \def\rmH{{\mathrm H}} 
  \def\rmL{{\mathrm L}}
\def\rmM{{\mathrm M}}

  \def\rmX{{\mathrm X}}

  \def\sfC{{\mathsf C}}

  \def\scrL{{\mathscr  L}}
  
  \def\scrR{{\mathscr  R}}

\def\bbG{{\mathbb G}}  
  
 \def\bbN{{\mathbb N}} 
  \def\bbR{{\mathbb R}}
 \def\bbT{{\mathbb T}} 
  
 \def\bbZ{{\mathbb Z}}

\newtheorem{theorem}{Theorem}[section]
\newtheorem{lemma}[theorem]{Lemma}
\newtheorem{corollary}[theorem]{Corollary}
\newtheorem{definition}[theorem]{Definition}

\newtheorem{proposition}[theorem]{Proposition}
\newtheorem{remark}[theorem]{Remark}
\newtheorem{example}[theorem]{Example}
\newtheorem{assumption}[theorem]{Assumption}

\makeatletter
\newcommand{\oset}[3][0ex]{%
  \mathrel{\mathop{#3}\limits^{
    \vbox to#1{\kern-2\ex@
    \hbox{$\scriptstyle#2$}\vss}}}}
\makeatother

\DeclareMathOperator*{\esssup}{ess\,sup}

\DeclareMathOperator{\diag}{diag}

\DeclareMathOperator{\shift}{S}

\DeclareMathOperator{\nablaDDiff}{\overline{\nabla}}
\DeclareMathOperator{\nablaDReact}{\Gamma}
\DeclareMathOperator{\nablaDisc}{\overline{\bbG}}

\DeclareMathOperator{\nablaCont}{\bbG}
\DeclareMathOperator{\nablaCReact}{\Gamma}
\DeclareMathOperator{\nablaCDiff}{\nabla}

\DeclareMathOperator{\divDDiff}{\overline{\div}}

\DeclareMathOperator{\divCDiff}{\div}

\definecolor{rot}{rgb}{1.000,0.000,0.000}

\newcommand{\N}{\mathbb{N}}
\newcommand{\R}{\mathbb{R}}

\newcommand{\cE}{\mathcal{E}}
\newcommand{\cR}{\mathcal{R}}

\newcommand{\Eu}{\mathrm{E}}

\renewcommand{\L}{\mathrm{L}}

\newcommand{\pcrit}{p_\mathrm{crit}}
\newcommand{\diff}{\mathrm{diff}}
\newcommand{\react}{\mathrm{react}}

\newcommand{\C}{\mathsf{C}}
\newcommand{\vertrule}[1][1.6ex]{\rule{.5pt}{#1}}
\newcommand{\CC}{\mathsf{C}\!\!\vertrule\;}
\newcommand{\fff}{\mathsf f}
\newcommand{\hhh}{\mathsf h}
\newcommand{\AC}{\mathrm{AC}}

\newcommand{\dom}{\mathrm{dom}}
\newcommand{\tar}{\mathrm{tar}}

\newcommand{\discref}{w}

\newcommand{\dd}{\,\mathrm{d}}
\renewcommand{\div}{\mathop{\mathrm{div}}}

\definecolor{rot}{rgb}{1.000,0.000,0.000}

\newcommand{\one}{1\!\!1}

\newcommand{\CE}{\mathrm{CE}}
\newcommand{\lCE}{\overline{\mathrm{CE}}}

\usepackage{babel}

\begin{document}

\title{Discrete-to-continuum limit for\\
nonlinear reaction-diffusion systems via \\
EDP convergence for gradient systems}

\author[1]{Georg Heinze}
\author[2]{Alexander Mielke}
\author[3]{Artur Stephan}
\affil[1]{{
\small
Weierstrass Institute, Berlin, Germany, email:
\texttt{georg.heinze@wias-berlin.de}
}}

\affil[2]{{
\small
Weierstrass Institute, Berlin, Germany, email:
\texttt{alexander.mielke@wias-berlin.de}
}}
\affil[3]{{
\small
Technische Universität Wien, Vienna, Austria, email:
\texttt{artur.stephan@tuwien.ac.at}
}}

\date{9 April 2025}

\maketitle

\lhead{EDP convergence for nonlinear RDS }
\rhead{G. Heinze, A. Mielke, A. Stephan}
\chead{\today}

\begin{abstract}
  We investigate the convergence of spatial discretizations for
  reaction-diffusion systems with mass-action law satisfying a detailed balance
  condition.
  Considering systems on the d-dimensional torus, we construct appropriate
  space-discrete processes and show convergence not only on the level of solutions, but
  also on the level of the gradient systems governing the evolutions.
  As an important step, we prove chain rule inequalities for the
  reaction-diffusion systems as well as their discretizations, featuring a non-convex dissipation functional.
  The convergence is then obtained with variational methods by building on the
  recently introduced notion of gradient systems in continuity equation format.

\end{abstract}

\section{Introduction}

The aim of this work is to show convergence of spatial discretizations of a
class of reaction-diffusion systems satisfying mass-action law. Considering finitely many species $\rmX_i$ with $i\in I=\{1,\ldots,i_*\}$ undergoing
finitely many reactions chemical reactions labeled by
$r\in R\coloneqq\{1,\ldots,r_*\}$ and diffusing in a medium, the
reaction-diffusion systems we are considering can in general be written as
\begin{align}
  \label{eq:PDE1Intro}
  \partial_t \rho_i =\div( \delta_i\rho_i \nabla(\log \rho_i {+} V_i)) + \sum_{r\in R}\bra[\bigg]{k^{\text{fw}}_r\prod_{\bar\imath\in I}\rho_{\bar\imath}^{\alpha^r_{\bar\imath}}-k^{\text{bw}}_r \prod_{\bar\imath\in I}\rho_{\bar\imath}^{\beta^r_{\bar\imath}}}(\beta^r_i-\alpha^r_i),\quad \rho_i(0)\coloneqq\rho_i^0.
\end{align}
where $\rho_i=\rho_i(t,x)$ describes the concentration of species $\rmX_i$. The parameters characterizing the diffusion are the diffusion coefficients
$\delta_i>0$ and additional  continuous 
drift potentials $V_i$. The chemical
reactions are prescribed by forward and backward reaction rates
$k^{\text{fw}}_r, k^{\text{bw}}_r >0$, $r\in R$, and stoichiometric
coefficients $\alpha^r_i,\beta^r_i\in(0,\infty)$ that describe the change of
particles of different type, say $\rmX_i$, by the chemical reactions
\[
\forall\, r\in R: \quad \alpha_1^r \rmX_1 + { \cdots +\alpha_{i_*}^r \rmX_{i_*}
\rightleftharpoons \beta_1^r \rmX_1 + \cdots +\beta_{i_*}^r \rmX_{i_*}. } 
\]
For simplicity, we consider the system on the $d$-dimensional torus $\bbT^d$,
though an extension to bounded domains with homogeneous Neumann boundary
conditions should not pose a significant difficulty.

The well-posedness, the long-time behavior, and approximations of systems of
the form~\eqref{eq:PDE1Intro} have been studied with different methods for a
long time. We refer to 
\cite{Alik79LBSR, Rothe_1984, Morg89GESP, desvillettes2007global} 
and references therein for the 
study of classical solutions, i.e., solutions that are globally uniformly
bounded in $\rmL^\infty([0,\infty)\times\bbT^d;\R^I)$. In this context, 
two often used strategies are to obtain bounds for the full 
reaction-diffusion system  by exploiting global a priori bounds for a 
lower order functional such as mass conservation or entropy bounds 
(cf.\ \cite{Alik79LBSR, Morg89GESP}) or by studying the 
space-independent reaction ODE system, which is 
applicable when all the species $\rho_i$, $i\in I$ diffuse with the same 
speed $\delta=\delta_i$. 

However, to the surprise of many, in \cite{Pierre_Schmitt_1997} it was shown
that in the case where species diffuse with different diffusion constants
$\delta_i$,
there may exist no classical solutions to the reaction-diffusion system even
when the involved reactions behave nicely. This furthered the interest in
weaker notions of solutions, like renormalized solutions studied e.g.\ in
\cite{Fischer}. 

In recent years, entropy methods became an important tool for the study 
of reaction-diffusion systems \cite{desvillettes2006exponential, 
desvillettes2007global, desvillettes2015duality, MiHaMa15UDER, DeFeTa17TERD}.  
Here, the idea is to use the non-increasing relative Boltzmann entropy 
(also called free energy) as an a priori bound 
to control solutions and study their long-time behavior. Moreover, 
entropy methods are a useful tool for deriving convergence results for the 
spatial discretization of linear reaction-diffusion systems, 
see  \cite{heinze2024gradient}; and it is the purpose of this work to generalize these results to nonlinear reaction kinetics. 

We stress that our works main focus is not on the regularity of the spatial
discretization nor of the model (like data or coefficients), but instead the
variational nature of our approach, which is based on the theory of generalized
gradient flows (see e.g.\ \cite{Mie24GF} for an introduction).

To be more precise, we are interested in reaction-diffusion systems, where not
only the free energy is decaying, but where the system is a \textit{gradient flow}
of the free energy. Starting with the pioneering work of Otto
\cite{JKO98FPE,Ott01PME}, it is known that many diffusion-type problems can be
understood as gradient flows driven by a suitable free energy. Later, this was
extended to reaction-diffusion systems satisfying \textit{detailed balance} in
\cite{Mielke11GSRD} for quadratic dissipations and in \cite{MPPR17NETCR} for
cosh-type dissipations, which we also use here.  The fundamental assumption
here, is that the system satisfies the detailed balance (or reversibility)
condition. This means that there exist reference concentrations
$\omega = (\omega_i)_{i\in I}$ such that $\omega_i \coloneqq \exp(-V_i)$ and such that
for all $r\in R$ it holds
\begin{align*}
	k^{\text{fw}}_r\prod_{\bar\imath\in I}\omega_{\bar\imath}^{\alpha^r_{\bar\imath}} 
	&= k^{\text{bw}}_r \prod_{\bar\imath\in I}\omega_{\bar\imath}^{\beta^r_{\bar\imath}}
	\eqqcolon \kappa_r\prod_{\bar\imath\in I}\omega_{\bar\imath}^{\frac{\alpha^r_{\bar\imath}+\beta^r_{\bar\imath}}{2}}.
\end{align*}
We observe that the detailed balance assumption ensures that the reaction-diffusion system \eqref{eq:PDE1Intro} can now be written in the symmetric form
\begin{align}
  \label{eq:RDSysCont}
    \partial_t \rho_i = \delta_i\div\bra[\Big]{\rho_i \nabla\log\bra[\Big]{\frac{\rho_i}{\omega_i}}} 
    + \sum_{r\in R}
    \kappa_r\omega^{\frac{\alpha^r+\beta^r}{2}}
    \bra[\bigg]{
    \bra[\Big]{\frac{\rho}{\omega}}^{\alpha^r}
    -
    \bra[\Big]{\frac{\rho}{\omega}}^{\beta^r}}
    (\beta^r_i-\alpha^r_i),
\end{align}
where we introduced the notation
\[
\rho^{\alpha^r}\coloneqq\prod_{i\in I}\rho_i^{\alpha^r_i},
\]
which will be used throughout. We further note that \eqref{eq:RDSysCont} contains the
\emph{tilt-invariant} form  of the reactions derived in \cite{MieSte20CGED}, where the dual dissipation potential, defined below, will not depend on $\omega$. 

What is more important to us, the system \eqref{eq:RDSysCont} has a gradient
structure and can now be investigated with variational methods.  More
precisely, we will introduce continuous and discrete gradient systems in
continuity equation format (cf.\ \cite{Peletier_Schlichting_2023}), link them to
\eqref{eq:RDSysCont} and appropriate jump processes, respectively, and obtain a
convergence result for these gradient systems in the spirit of 
$\Gamma$-convergence for gradient flows, see \cite{SandierSerfaty}, more
precisely EDP-convergence in the sense of
\cite{MieSte20CGED,Step21EDPLRDS,MiMoPe21EFED,Mie24GF}. Upon rigorously
linking these gradient systems to their corresponding equations, the finite
approximation of solutions will then be a direct consequence.

Next, we discuss our finite approximation of the system
\eqref{eq:RDSysCont}. For simplicity, we discretize the torus using uniform
grids $\bbZ^d_N$, $N\in\N$, noting that our model can be generalized to other
domains and discretizations by following ideas of, e.g., \cite{HT23}.
Furthermore, we assume for simplicity that the diffusion coefficients
$\delta_i$ as well as the reaction coefficients $\kappa_r$ are spatially
independent, although our analysis would not be harmed when considering
sufficiently smooth coefficients that are uniformly bounded above and away from
zero.

For fixed $N\in\N$, the discretized evolution equation is a coupled ODE of the concentration $c_{i,k}$ of each species $i\in I$ in each discrete position $k\in\bbZ^d_N$. Denoting by $\Eu \coloneqq \set[\big]{e=(e_1,\ldots,e_d)^\top\in \set{0,1}^d, \sum_{l=1}^d e_l =1}$ the set of $d$-dimensional unit vectors, the evolution is given by 
\begin{equation}\label{eq:RDSysDisc}
\begin{aligned}
    \dot c_{i,k} &= \sum_{e\in\Eu}
    \pra[\Big]{N^2\rmd_{i,k,e}\bra[\Big]{\frac{c_{i,k+e}}{w^N_{i,k+e}} 
    -\frac{c_{i,k}} {w^N_{i,k}}} + N^2 \rmd_{i,k,-e} 
    \bra[\Big]{\frac{c_{i,k-e}}{w^N_{i,k-e}}-\frac{c_{i,k}}{w^N_{i,k}}}}
\\
    &+ \sum_{r\in R}
    \kappa_r\omega^{\frac{\alpha^r+\beta^r}{2}}
     \bra[\bigg]{ \bra[\Big]{\frac{c_k}{w^N_k}}^{\alpha^r}
    - \bra[\Big]{\frac{c_k}{w^N_k}}^{\beta^r}} (\beta^r_i-\alpha^r_i),
\end{aligned}
\end{equation}
for the discrete reference concentrations $w^N_{i,k} \coloneqq \int_{Q^N_k}
\omega_i\rmd x$, where $Q^N_k \coloneqq \set{x\in\bbT^d: x_l 
\in [k_l/N,(k_l{+}1)/N), l=1,\ldots,d}$ are $d$-dimensional cubes of 
side length $1/N$. This system is related to the reaction-diffusion
master equation (RDME) treated in \cite{MoScWi23RHLR}, where also the diffusion
is replaced by jumps between nearest neighbors on the lattice. The 
intensity of the jumps is characterized  by rates $\rmd_{i,k,e} \coloneqq 
\delta_i\sqrt{w_{i,k+e} w_{i,k}}$, which are scaled by $N^2$. In our 
case, the reactions are modeled pointwise nonlinearities analogously to the 
space-continuous system, whereas in the RDME the reactions are modeled 
as linear jump processes on the number of particles. Our 
systems are complemented with suitable initial data satisfying a suitable 
well-preparedness condition specified later.

This work contains three main analytical results, which we summarize here. We 
refer to Section~\ref{sec:MainResults} for more details. 
The first main result, Theorem~\ref{thm:DiscEDB.ODE}, is the rigorous link of each 
prelimit system \eqref{eq:RDSysDisc} to a corresponding gradient structure via a 
so-called \emph{energy-dissipation principle} (EDP). Here, the main step is proving a 
chain rule, which is obtained by exploiting the discrete nature of the underlying 
base space.

The second main result, Theorem~\ref{thm:CRContinuous}, is the energy
dissipation principle for the limit system \eqref{eq:RDSysCont}. Since this
model is defined over a continuous base space, multiple regularity issues have
to be overcome in order to control in particular the nonlinear reaction terms,
which generate non-convexities in the variational formulation. 
Introducing the length 
$|\gamma|_1= \sum_{i\in I} \gamma_i $ of stoichiometric vectors 
$\gamma \in [0,\infty)^I$, our main assumptions are the following:
\begin{equation}
    \label{eq:Intro.Ass}
    \forall\, r\in R: \quad |\alpha^r|_1, |\beta^r|_1\leq \pcrit\coloneqq1{+}2/d 
    \ \text{ and } \ \frac12|\alpha^r{+}\beta^r|_1 < \pcrit.
\end{equation}
Note that the same critical growth exponent $\pcrit$ appears already 
in \cite[Thm.\,2.3]{Morg89GESP} (one has to choose $a=1$ there due to our 
$\rmL^1$ bound obtained from the relative enntropy) for showing global 
existence of smooth solutions. 

The growth power $\pcrit$ can be achieved not only for solutions, 
but for all curves satisfying natural a priori bounds on the energy and 
dissipation by exploiting the regularity the diffusion provides. 
If the conditions \eqref{eq:Intro.Ass} are not met, our analysis can still 
be carried out if the system admits natural $\rmL^\infty$-bounds, 
see Remark \ref{rem:BoundBox}. 
Such bounds are known to be satisfied by solutions to several 
classes of reaction-diffusion systems, namely if there exists a so-called 
bounding box, see \cite{Smol94SWRD} and our Remark \ref{rem:BoundBox}. 

The final main result, Theorem~\ref{thm:Convergence}, is the convergence of gradient
systems. For this, we require the slightly weaker assumption than \eqref{eq:Intro.Ass} that $\frac12|\alpha^r{+}\beta^r|_1 \leq \pcrit$ for all $r\in R$. The convergence proof consists of two parts, a compactness result and a lower limit. To establish the compactness, we construct a suitable family of embeddings into a unified space that crucially keep the gradient structure in tact. 

With these three main results, the convergence of solutions of \eqref{eq:RDSysDisc} 
to solutions of \eqref{eq:RDSysCont} follows in Corollary~\ref{cor:conv_sol}.

A major difficulty in the analysis of the space-continuous reaction-diffusion 
system comes from the reaction-induced nonlinearities and the resulting 
non-convexity of the dissipation. Here, we can use the some of the 
surprising properties of the cosh gradient structure based on the function 
$\C^*(\zeta)= 4 \cosh(\zeta/2)-4$, that is relevant for linear and nonlinear
reactions, see \cite{MPPR17NETCR, LMPR17MOGG, Peletier_Schlichting_2023, MiPeSt21EDPC, PRST22}. These surprising properties are encoded in nontrivial estimates for the Legendre dual $\C$ and its perspective function $\CC:(s,w) \mapsto w\C(s/w)$, see \eqref{eq:CC.prop}.  In particular, we can exploit the 
 \emph{magical estimate}
\begin{equation} 
  \label{eq:I.MagEst}
\forall\, q>1\ \forall\, s\in \R\ \forall\, w>0: \quad \C(s) \leq \frac{q}{q{-}1}
\, \CC\big(s \big|w\big) + \frac{4\,w^q}{q{-}1} 
\end{equation}
(see \ref{eq:CC.prop.d} and Proposition \ref{prop:MagEst} for the proof). Note 
that such an estimate does not hold for dissipation potentials $\psi$ of power-law type: for $\psi(s)= |s|^p$ with $p>1$ we obtain $\Psi(s|w)= |s|^p/w^{p-1}$ such that  the right-hand side in \eqref{eq:I.MagEst} only bounds the weaker power law $|s|^r$ with $r=pq/(p{+}q{-}1)\lneqq p$.  

As usual, the chain rule is proved by a smoothing argument. In contrast to 
linear reaction systems like for Markov processes and Fokker-Planck type 
equations (cf. e.g., \cite{Step21EDPLRDS, Peletier_Schlichting_2023, PRST22}), 
it is not possible to rely solely on convexity arguments. Instead, our 
convergence proof combines the magical estimate with the Hardy-Littlewood 
maximal function from harmonic analysis and the easy but non-trivial estimate 
$\partial_w \CC(s|w)| \leq 2|s/w|$ (see \eqref{eq:CC.prop.b}) to 
obtain an integrable majorant on the reactive flux. We believe that this 
flexible approach could also be used for handling non-convexities in 
other cases where the cosh gradient structure is relevant. 

For the discrete approximation the challenge lies in deriving weak-$\rmL^1$ 
compactness for the reactive fluxes. Here, we require $\pcrit$-uniform 
integrability of the embedded concentrations. To achieve this, we exploit 
the flexibility of the embedding method by introducing a second family of 
more regular embeddings. For this family higher integrability can be obtained, 
while we rely on the first family of embeddings to obtain the liminf inequality.

The paper is structured as follows: In Section~\ref{sec:abstract_strategy}, 
we present the abstract strategy of the paper. Section~\ref{s:GSRDS} introduces 
first the gradient structures for the discrete and the continuous reaction-diffusion 
systems. Then, we connect both models with an embedding such that we can state 
the main results of the paper in Section~\ref{sec:MainResults}. There, we also 
list and discuss in detail 
the assumptions on our reaction coefficients. The proof of our convergence result 
is carried out in Section~\ref{s:ProofConvergence}. Here, we first derive the
compactness, before showing the claimed liminf-estimate. Finally, 
Section~\ref{s:CRs} contains the detailed proofs of the chain rules inequality first 
for discrete and then for the continuous reaction-diffusion systems.

\section{Abstract strategy}\label{sec:abstract_strategy}

To improve clarity, before challenging the reader with the notation of our concrete problem, we first present on a formal and abstract level the strategy of our work.

\subsection{Gradient systems with explicit abstract gradient mappings}
\label{subsec:GradSyst}

We begin by introducing a quintuple $(X,Y,\grad,\calE,\calR^\ast)$, called a \emph{gradient system}. This notion is a small modification of gradient systems in continuity equation format introduced in \cite{Peletier_Schlichting_2023}.

The elements of the gradient system are two pairs of \emph{base spaces} 
$X = (X^\dom,X^\tar)$, $Y = (Y^\dom,Y^\tar)$, where $X^\dom,Y^\dom$ are Borel 
subsets of a Euclidean space, and $X^\tar,Y^\tar$ are Euclidean  spaces. 
Test functions over these spaces are linked by  an abstract linear 
\emph{gradient map} $\grad : \rmC^\infty_c(X^\dom,X^\tar) \to \rmC^\infty_c(Y^\dom,Y^\tar)$,
with dual $\grad^\ast : (\rmC^\infty_c(Y^\dom,Y^\tar))^\ast \to \rmC^\infty_c((X^\dom,X^\tar))^\ast$,  which is sometimes called 
\emph{process-space to tangent map}. Here,  for $Z\in\{X,Y\}$ the 
dual pairing is defined as usual by
\[
\forall\, \phi\in \rmC^\infty_c(Z^\dom,Z^\tar)\quad\forall\, \mu\in (\rmC^\infty_c(Z^\dom,Z^\tar))^\ast: \langle\phi,\mu\rangle_Z \coloneqq\langle\phi,\mu\rangle \coloneqq \int_{Z^\dom} \phi \cdot\dd\mu, 
\]
with $\cdot$ denoting the canonical inner product in the Euclidean space $Z^\tar$.

The fourth element of the quintuple is a lower semicontinuous (lsc) 
\emph{energy functional} $\calE:\calM_+(X^\dom,X^\tar)\to\R\cup\{\infty\}$, 
where $\calM_+(X^\dom,X^\tar)$ denotes the set of $X^\tar$-valued, 
component-wise non-negative  Radon measures. The final element is a 
\emph{dual dissipation potential} $\calR^\ast:\calM_+(X^\dom,X^\tar)\times 
C(Y^\dom,Y^\tar)\to [0,\infty]$, which, by definition, is lsc and non-negative 
 with $\calR^*(\rho,0)=0$,  and satisfies for all 
$\rho\in\calM_+(X^\dom,X^\tar)$ that $\xi\mapsto \calR^*(\rho,\xi)$ is convex.

Fixing an arbitrary time horizon $T>0$ and an initial datum 
$ \rho_0\in\calM_+(X^\dom,X^\tar)$, the gradient system $(X,Y,\grad,\calE,\calR)$
induces on $[0,T]$ an evolution equation, the \emph{gradient flow} equation 
\begin{subequations}
 \label{eq:abstract_GF}
\begin{align}
    \label{eq:abstract_GF_equation}
    &\partial_t \rho = \grad^\ast \partial_\xi\calR^\ast(\rho,-\grad \rmD\calE(\rho)), \\
    \label{eq:abstract_initial}
    &\rho(0) = \rho_0,
\end{align}    
\end{subequations}
where $\rmD\calE$ denotes the variational derivative of $\calE$ and $\partial_\xi\calR^\ast$ denotes the convex subdifferential of $\calR^\ast(\rho,\cdot)$. 

By specifying $\calR^\ast$ and $\calE$, the gradient system contains more information than the gradient-flow equation. Indeed, it is well-known that the same gradient-flow equation can be derived from different gradient systems, each corresponding to a different physical setting, see \cite{Mie24GF}. 

Before further discussing the link between gradient flow and gradient system, we comment on the relation of the presented notion of gradient flow with other notions.

\begin{remark}[Link to other notions of gradient flow]
\label{rem:link_to_other_GF-notions}
The presented notion is heavily influenced by the gradient systems in 
continuity equation format introduced in \cite{Peletier_Schlichting_2023}, 
the only difference being the split of $X$ and $Y$ in a domain and a target 
space, which allows us to directly incorporate well-known objects like the 
classical gradient $\nabla: \rmC^1(\bbT^d)\to \rmC(\bbT^d;\R^d)$ into our framework.
Furthermore, we observe that by setting $\grad=\operatorname{id}$, $Y = X$, 
$\calX = \calM_+(X)$ and $\scrR^\ast(\rho,\varphi) \coloneqq 
\calR^\ast(\rho,\grad\varphi)$ for all $\rho,\varphi$, we recover the 
well-established notion of a gradient system $(\calX,\calE,\scrR^\ast)$ 
as introduced in \cite{Mie24GF}. Choosing $\calR^\ast$ as a quadratic functional, we can 
also recover metric gradient systems and metric gradient flows in the 
spirit of \cite{AGS}.
\end{remark}

To establish the link between the gradient system $(X,Y,\grad,\calE,\calR^\ast)$ and 
the gradient flow equation \eqref{eq:abstract_GF}, we split the latter into two parts:
First, we introduce the \emph{continuity equation}, which links a weak-$^\ast$ 
measurable curve $\rho:[0,T]\to \calM_+(X^\dom,X^\tar)$ 
with a  weak-$^\ast$ measurable \emph{curve of fluxes} $j:[0,T] \to \calM(Y^\dom,Y^\tar)$ 
by the relation (understood in the sense of distributions on $[0,T]\times X^\dom$)
\begin{align}\label{eq:abstract_CE}
    \partial_t \rho = \grad^\ast j.
\end{align}
The set of curves  $(\rho,j)$  satisfying \eqref{eq:abstract_CE} is denoted 
by $\CE$.

Secondly, given a pair $(\rho,j)\in\CE$, we recover \eqref{eq:abstract_GF} if the 
initial condition \eqref{eq:abstract_initial} holds and $j$ satisfies (in the 
sense of measures) the \emph{constitutive relation} 
\begin{align}\label{eq:abstract_KR}
    j = \partial_\xi\calR^\ast(\rho,-\grad \rmD\calE(\rho)).
\end{align}
One important  link between a gradient system and its induced gradient 
flow is called the \emph{energy-dissipation principle}. It is formally 
established as follows. We introduce the \emph{(primal) dissipation potential} 
$\calR:\calM_+(X^\dom,X^\tar) \times \calM(Y^\dom,Y^\tar)\to [0,\infty]$ as the 
convex dual of $\calR^\ast$ with respect to the second variable. Together 
$\CE$, $\calE$, $\calR$ and $\calR^\ast$  give rise to the \emph{dissipation functional}
\begin{align*}
    \calD(\rho,j)\coloneqq 
    \begin{cases}
        \int_0^T \calR(\rho,j) {+} \calR^\ast(\rho,-\grad\rmD\calE(\rho))\dd t & \quad\text{for } (\rho,j)\in\CE\\
        +\infty & \quad\text{for } (\rho,j)\notin\CE,
    \end{cases}
\end{align*}
and the \emph{energy-dissipation functional}
\begin{align*}
    \calL(\rho,j) \coloneqq \calE(\rho(T))-\calE(\rho(0))+\calD(\rho,j).
\end{align*}
 We say that  the gradient system satisfies the \emph{energy-dissipation
principle} if $\rho$ solving \eqref{eq:abstract_GF}  (in a suitable weak sense) 
is equivalent to $(\rho,j)$ solving $\CE$
and $\calL(\rho,j) = 0$. 

A crucial role in making this principle rigorous is played by the 
\emph{chain rule inequality} for the gradient system,  which means that 
$\calL(\rho,j)\ge 0$ holds for all $(\rho,j)\in\CE$. This name is  
motivated by the following formal calculation:
\begin{align*}
    \frac{\rmd}{\rmd t}\calE(\rho) &= \langle\rmD\calE(\rho),\partial_t\rho\rangle_X = -\langle -\grad\rmD\calE(\rho),j\rangle_Y \ge -\calR^\ast(\rho,-\grad\rmD\calE(\rho))-\calR(\rho,j),
\end{align*}
where the first equality is the classical chain rule, the second equality holds for $(\rho,j)\in\CE$, and the inequality follows from the duality of $\calR^\ast$ and $\calR$ (Young-Fenchel estimate). 
Integrating in time from 0 to $T$,  we obtain $\calL(\rho,j)\geq 0$. 

However, if $\calL(\rho,j)\le 0$ is imposed additionally, then we must have equality in the Young-Fenchel inequality for a.a.\ $t\in [0,T]$: 
\[
\langle -\grad\rmD\calE(\rho),j\rangle_Y  = \calR^\ast(\rho,-\grad\rmD\calE(\rho)) +\calR(\rho,j). 
\]
By the Fenchel equivalence, this implies that \eqref{eq:abstract_KR} holds a.e.\ on [0,T]. Plugging this into CE \eqref{eq:abstract_CE} shows that \eqref{eq:abstract_GF} holds. The opposite direction from \eqref{eq:abstract_GF} to $\calL(\rho,j)=0$ with $j$ from \eqref{eq:abstract_GF} is obvious.  

Of course, we will make these arguments rigorous for the reaction-diffusion systems under consideration.

\subsection{Convergence of gradient systems}
\label{ss:GeneralConvergenceStrategy}

Having introduced abstract gradient systems and briefly discussed the 
energy-dissipation principle, we now want to discuss, on an abstract level, 
our strategy for obtaining the convergence of gradient flows. 

To this end, consider a family of approximating gradient systems $(X_N,Y_N, 
\grad_N,\calE_N,\calR^\ast_N)_{N\in\bbN}$ inducing $\CE_N$, $\calD_N$, and 
$\calL_N$ as before. As a first step, one proves that each 
$(X_N,Y_N,\grad_N,\calE_N,\calR^\ast_N)$ satisfies a chain rule inequality 
and an energy-dissipation principle. Next, one shows that for each $N\in\bbN$ 
and each initial datum $\rho_N^0$ with $\calE_N(\rho_N^0)<\infty$ there exists 
a solution pair $(\tilde\rho_N,\tilde\jmath_N)\in\CE_N$ with 
$\tilde\rho_N(0)=\rho_N^0$ and $\calL_N(\tilde\rho_N,\tilde\jmath_N)=0$.

Our aim is to connect the approximating gradient systems with a limit gradient 
system $(X,Y,\grad,\calE,\calR^\ast)$ inducing $\CE$, $\calD$, and $\calL$. 
For this, one has to show that $(X,Y,\grad,\calE,\calR^\ast)$ also satisfies 
chain rule inequality and energy-dissipation principle.

To establish the link, a candidate curve that might be a solution for the 
limit system needs to be obtained by applying a compactness argument to the family of prelimit solutions $(\tilde\rho_N,\tilde\jmath_N)_{N\in\bbN}$. However, the different gradient systems are defined 
over different base spaces, hence a unified space is needed in which compactness 
can be realized. To this end, one constructs an \emph{embedding operator} $\iota_N:\calM_+(X_N^\dom,X_N^\tar)\to \calM_+(X^\dom,X^\tar)$ and a \emph{discretization operator} $\iota_N^\ast:\rmC^\infty_c(X^\dom,X^\tar)\to \rmC^\infty_c(X_N^\dom,X_N^\tar)$
such that for all $\varphi\in \rmC^\infty_c(X^\dom,X^\tar)$ it holds
\begin{align*}
    \langle \iota_N\rho_N,\varphi\rangle_X = \langle \rho_N,\iota_N^\ast\varphi\rangle_{X_N}.
\end{align*}
For the fluxes one constructs $\iota_{N,\grad}:\calM_+(Y_N^\dom,Y_N^\tar)\to \calM_+(Y^\dom,Y^\tar)$ such that for all $\varphi\in \rmC^\infty_c(X^\dom,X^\tar)$ it holds
\begin{align*}
    \langle \iota_{N,\grad}j_N,\grad\varphi\rangle_Y = \langle j_N,\grad_N\iota_N^\ast\varphi\rangle_{Y_N}.
\end{align*}
Since the continuity equation is understood in the sense of distributions, this implies that $(\rho_N,j_N)\in\CE_N$ if and only if $(\iota_N\rho_N,\iota_{N,\grad}j_N)\in\CE$. 

Now it is possible to prove that for each family $(\rho_N,j_N)_{N\in\bbN}$ with $(\rho_N,j_N)\in\CE_N$ and $\sup_{N\in\bbN}\sup_{t\in[0,T]}\calE_N(\rho_N(t))<\infty$ as well as $\sup_{N\in\bbN}\calD_N(\rho_N,j_N)<\infty$, there exists $(\rho,j)\in\CE$ with $\calE(\rho(0))<\infty$ and $\calD(\rho,j)<\infty$ such that (along a subsequence) $(\iota_N\rho_N,\iota_{N,\grad}j_N) \rightharpoonup^\ast (\rho,j)$. 
In particular, such a limit $(\tilde\rho,\tilde\jmath)$ exists for the family $(\tilde\rho_N,\tilde\jmath_N)_{N\in\bbN}$.

Next, one shows that for each family $(\rho_N,j_N)_{N\in\bbN}$ satisfying the a priori 
bounds as before and each limit $(\rho,j)$ of the embedded family, we have the liminf 
estimates 
\begin{align}\label{eq:abstract_lower_limit}
    \liminf_{N\to\infty}\calD_N(\rho_N,j_N) \ge \calD(\rho,j) 
    \qquad \text{and}\qquad \liminf_{N\to\infty}\calE_N(\rho_N(t)) 
    \ge \calE(\rho(t))\quad \forall\, t\in[0,T].
\end{align}
Notice that this inequality relates the dissipation functionals of the non-embedded
curves with the limiting dissipation functional of the limiting curve.
In particular, \eqref{eq:abstract_lower_limit} holds for the previously obtained 
family of solutions $(\tilde\rho_N,\tilde\jmath_N)_{N\in\bbN}$ 
and each of its limits $(\tilde\rho,\tilde\jmath)$.

To conclude that the limits are indeed solution, we now assume the well-preparedness of initial data
\begin{align*}
    \iota_N\rho_N^0 \rightharpoonup^\ast \rho^0\qquad \text{and}\qquad\lim_{N\to\infty}\calE_N(\rho_N^0) = \calE(\rho^0) <\infty.
\end{align*}
The energy identity combined with \eqref{eq:abstract_lower_limit} and the limit 
chain rule inequality yield
\begin{align*}
    0 
    = \liminf_{N\to\infty}\calL_N(\tilde\rho_N,\tilde\jmath_N) 
    \ge \calL(\tilde\rho,\tilde\jmath) 
    \ge 0,
\end{align*}
from which the energy-dissipation principle of the limit gradient system implies that $\tilde\rho$ is solution starting at $\tilde\rho(0) = \rho^0$ and that $\tilde\jmath$ is given by the kinetic relation \eqref{eq:abstract_KR}.

\section{Gradient system for the reaction-diffusion system}\label{s:GSRDS}
\label{sec:Grad4RDS}

We want to describe the evolution of $i_*$ chemical species $ \rmX_i$ with $i\in \{1,\dots, {i_*}\} \eqqcolon I $ undergoing diffusion in a subdomain $\Omega\subset\R^d$ and interacting according to $r_*$ chemical reactions: 
\[
\sum_{i\in I} \alpha^r_i \rmX_i \rightleftharpoons \sum_{i\in I} \beta^r_i \rmX_i, \quad r\in\{1,\dots,r_*\}\eqqcolon R.
\]
Throughout the paper we assume that  the physical domain is given by  
$\Omega=\bbT^d$ (the $d$-dimensional torus), and that we have finitely many 
species and reactions,  i.e., $i_*,r_*\in \N $.  In the following, we 
 will also use the effective stoichiometric vectors  
$\gamma^r\eqqcolon\alpha^r {-} \beta^r \in \R^I$.  
Moreover, we fix reaction coefficients $\kappa_r>0$ (describing the reaction speed) 
for each reaction and diffusion coefficients $\delta_i>0$ for each species.

\subsection{Discrete reaction-diffusion gradient systems}

We present the gradient structure for the spatially discrete reaction-diffusion 
system with fixed $N\in\bbN$. Denoting $\bbZ^d_N = (\bbZ/N\bbZ)^d$ the set of 
discrete positions (with periodic boundary conditions), 
and $\Eu \coloneqq \set[\big]{e=(e_1,\ldots,e_d)^\top\in \set{0,1}^d, \sum_{l=1}^d e_l =1}$ the set of discrete directions, we introduce the spaces 
\begin{align*}
    X_N&&&\coloneqq (X_N^\dom,X_N^\tar) &&\coloneqq (I\times\bbZ^d_N,\R),\\
    Y_{N,\diff}&&&\coloneqq (Y^\dom_{N,\diff},X^\tar_{N,\diff})&&\coloneqq (I\times\bbZ^d_N\times\Eu,\R),\\
    Y_{N,\react}&&&\coloneqq (Y^\dom_{N,\react},X^\tar_{N,\react})&&\coloneqq (R\times\bbZ^d_N,\R),\\
    Y_N&&&\coloneqq (Y^\dom_N,Y^\tar_N)&&\coloneqq (Y^\dom_{N,\diff}\times Y^\dom_{N,\react},Y^\tar_{N,\diff}\times Y^\tar_{N,\react}).
\end{align*}
We introduce the short notation $\rmC(X_N) \coloneqq \rmC(X_N^\dom;X_N^\tar)$ and analogously for all other spaces of functions/measures over these spaces. Furthermore, given a time interval $[0,T]$ we write $\rmC([0,T]\times X_N)\coloneqq \rmC([0,T]\times X_N^\dom;X_N^\tar)$ and analogously for all other spaces of functions or measures over these spaces.

Abusing notation, we denote by $\langle\cdot,\cdot\rangle_N$ the dual products for vectors as well as components, e.g., for $(\zeta,\xi)\in \rmC(Y_N)$ and $(u,v)\in \calM(Y_N)$ we write
\begin{align*}
    \langle(\xi,\zeta),(u,v)\rangle_N &= \langle \xi,u\rangle_N + \langle \zeta,v \rangle_N 
    = \sum_{i\in I} \langle \xi_i,u_i \rangle_N + \sum_{r\in R}\langle \zeta_r, v_r \rangle_N \\
    &= \frac{1}{N^d}\sum_{k\in\bbZ^d_N}\bigg(\sum_{i\in I}\sum_{e\in\Eu}  \xi_{i,k,e} u_{i,k,e} + \sum_{r\in R} \zeta_{r,k} v_{r,k}\bigg)
\end{align*}
and similarly for other functions/measures defined over $X_N$ or $Y_N$.

Again abusing notation, but highlighting that no spatial component is involved, we introduce for the inner products on $\R^I$ and $\R^R$ the notation 
\[
    \gamma\bullet\xi =\sum_{i\in I}\gamma_i\xi_i, \quad f\bullet \psi = \sum_{r\in R} f_r\psi_r. 
\]
Given stoichiometric vectors $\alpha^r,\beta^r\in [0,\infty)^I$ and  
$\gamma^r =\alpha^r{-}\beta^r$ for $r\in R$, we define the \emph{discrete gradient} 
$\nablaDDiff$, the \emph{stoichiometric matrix} (or \emph{Wegscheider matrix}) 
$\Gamma$, and the  \emph{abstract linear gradient map}  $\nablaDisc$ by
\begin{align*}
    \nablaDisc&:\rmC^\infty_c(X_N) \to \rmC^\infty_c(Y_N),\quad 
    &&\nablaDisc\varphi \coloneqq (\nablaDDiff\varphi, \nablaDReact\varphi), 
    \quad \text{with}
\\[0.3em]
    \nablaDDiff&:\rmC^\infty_c(X_N) \to \rmC^\infty_c(Y_{N,\diff}),\quad 
    &&\nablaDDiff\varphi_{i,k,e} \coloneqq \varphi_{i,k+e}-\varphi_{i,k},
    \quad \text{and} 
\\
    \nablaDReact&:\rmC^\infty_c(X_N) \to \rmC^\infty_c(Y_{N,\react}),\quad 
    &&\nablaDReact\varphi_{r,k} \coloneqq \sum\nolimits_{i\in I}\gamma^r_i\varphi_{i,k} = \gamma^r\bullet\varphi_k.
\end{align*}
Their dual operators are given by
\begin{align*}
    \nablaDisc^\ast&:(\rmC^\infty_c(Y_N))^* \to (\rmC^\infty_c(X_N))^*,\quad 
    &&\nablaDisc^\ast(\xi,\zeta) \coloneqq  -\divDDiff\xi + \nablaDReact^\ast\zeta,
      \quad \text{with }    
\\[0.3em]
    -\divDDiff&:(\rmC^\infty_c(Y_{N,\diff}))^* \to (\rmC^\infty_c(X_N))^*,\quad 
    &&-\divDDiff\xi_{i,k} \coloneqq \sum\nolimits_{e\in\Eu}(\xi_{i,k-e,e}-\xi_{i,k,e}),
    \quad \text{and}    
\\
    \nablaDReact^\ast&:(\rmC^\infty_c(Y_{N,\react}))^* \to (\rmC^\infty_c(X_N))^*,\quad 
    &&\nablaDReact^\ast\zeta_{i,k} \coloneqq \sum\nolimits_{r\in R}\gamma^r_i\zeta_{r,k} = \gamma_i\bullet\zeta_k.
\end{align*}
Elements of the state space $\calM_+(X_N)$ are denoted by $c=(c_{i,k})_{i\in I, k\in\bbZ^d_N}$ and will be called \textit{chemical concentrations}.
We consider the \emph{relative entropy} with respect to a  positive 
\emph{reference concentration} $\discref\in\calM_+(X_N)$
\begin{equation}\label{eq:discrete_energy}
    E_N(c) \coloneqq \frac{1}{N^d}\sum_{i\in I}\sum_{k\in\bbZ_N^d}\LB \bra[\bigg]{\frac{c_{i,k}}{\discref_{i,k}}}\discref_{i,k},
\end{equation}
where the Boltzmann function is defined by $\LB (r)=r\log r-r+1$.

The \emph{discrete dual dissipation potential} $R^*_N: \calM_+(X_N)\times \rmC(Y_N) 
\to[0, \infty)$ consists of two parts, which correspond to the discrete diffusion 
(i.e. jumps) and reactions, respectively. It is defined for $c \in\calM_+(X_N)$, 
$\xi \in \rmC(Y_{N,\diff})$, and $\zeta \in \rmC(Y_{N,\react})$ by
\begin{align*}
    R^*_N (c,(\xi,\zeta)) &\coloneqq R^*_{N, \diff} (c,\xi) + R^*_{N, \react} (c,\zeta)
   \quad \text{with}
\\[0.4em]
    R^*_{N, \diff} (c,\xi) &\coloneqq \frac{1}{N^d} \sum_{i\in I}\sum_{k\in\bbZ_N^d} \sum_{e\in \Eu} N^2 \delta_i \bra[\big]{c_{i,k}c_{i,k+e}}^{1/2}    \C^*(\xi_{i,k,e}),
\\
    R^*_{N, \react} (c,\zeta) &\coloneqq  \frac{1}{N^d} \sum_{r\in R}\sum_{k\in\bbZ_N^d}  \kappa_r \bra[\big]{c_k^{\alpha^r}c_k^{\beta^r}}^{1/2} \C^*(\zeta_{r,k}),
\end{align*}
where $\C^*(r) \coloneqq 4\cosh(r/2)-4$.
In the sequel we will write $R^*_N (c,\xi,\zeta)$ instead of $R^*_N (c,(\xi,\zeta))$, and analogously for similar objects depending on a configuration, a diffusive component, and a reactive component.
Note that the diffusive part of the dissipation contains a factor $N^2$ that will provide  the continuous diffusion in the limit $N\to \infty$. Note that $E_N$ depends on $w \in \calM_+(X_N)$, whereas $R^*_N$ is independent of $w$, which is called tilt-invariance in \cite{MieSte20CGED}. 

The previously defined objects form the discrete gradient system 
$(X_N,Y_N,\nablaDisc,E_N,R^\ast_N)$. The corresponding gradient flow equation is 
the discrete reaction-diffusion system \eqref{eq:RDSysDisc}.

Throughout we will make use of  various  properties of the function $\C^*(r)$
characterizing $R^\ast_N$ and its Legendre transform $\C$. We gather these 
properties in the following lemma:

\begin{lemma}\label{lem:PropertiesC}
    The convex function $\C^*:\R\to [0,\infty)$ defined by
    \begin{align*}
        \C^*(\sigma)=4 \cosh(\sigma/2)- 4
    \end{align*}
    and its convex conjugate 
    \begin{align*}
        \C(s) \coloneqq \sup_{\sigma \in \R} \{\sigma s - \C^*(\sigma)\} = 2s\mathop{\mathrm{Arsinh}} ( s/2) - 2\sqrt{s^2{+}4}+4
    \end{align*}
have the following properties:
\begin{subequations}
\begin{align} 
 \label{eq:CStarandLog}
&\begin{aligned} 
  \forall\,a,b>0: \quad &
         \sqrt{ab}\cdot\C^*(\log a- \log b) 
            = 2\abs[\big]{\sqrt{a} - \sqrt{b}}^2, \\
            &\sqrt{ab}\cdot(\C^*)'(\log a - \log b) = a - b;
    \end{aligned}
\\
 & \label{eq:Cprime.bound}
  \forall \, s \in \R : \quad   
    	\C(s) \leq s \C'(s) \leq 2 \C(s)  \ \text{ and } \ 
       \frac{|s|}2 \log\big(1{+}|s|\big) \leq \C(s) \leq |s|\log\big(1{+}|s|\big).
\end{align}
\end{subequations} 
\end{lemma}
\begin{proof}
These results are obtained by elementary calculations, see e.g.\ \cite[Lem.\,3.4]{HT23} for more details.
\end{proof}

 In addition to $\C$ we also need its so-called \emph{perspective function} 
$\CC:\times [0,\infty) \to [0,\infty]$, which is given by
\begin{equation}
    \label{eq:Def.CC.perspect}
\CC(s|w)\coloneqq\sup_{\zeta\in\R}\{s\zeta- w \C^*(\zeta) \} = 
\begin{cases} w\C(s/w) &\text{for } w>0 , 
       \\ \chi_0(s) &\text{for }w=0. \end{cases}
\end{equation}

In the sequel we will need the following properties of $\CC$. The last result 
is the magical estimate that will be crucially used in 
Proposition~\ref{prop:React.Commutator.Estim}. For a similar estimate for the 
relative Boltzmann entropy, we refer to \cite[Eqn.\,(2.7)]{FHKM22GEAE}. 

\begin{lemma}[Properties of the perspective function $\CC$] 
\label{lem:Props.CC}
\begin{subequations}
  \label{eq:CC.prop}
\begin{align}
\label{eq:CC.prop.a}
   &\text{The mapping } \R\times[0,\infty)\ni (s,w)\mapsto \CC(s|w) \ \text{ is strictly convex}.
   \\
\label{eq:CC.prop.b}
   & \left. \begin{aligned} 
    &\forall\, s\in \R: \quad w\mapsto \CC(s|w) \ \text{ is non-increasing with} 
       \\
   &\partial_w \CC(s|w)= \C(r){-}r \C'(r)\big|_{r=s/w}= 4-2\sqrt{(s/w)^2{+}4}\leq 0.
   \end{aligned} \ \ \right\} 
\\
\label{eq:CC.prop.f}
&\forall \, s\in \R\ \forall \, w>0: \quad (0,\infty)\ni
   \lambda \mapsto \CC(\lambda s|\lambda^2 w) 
 \text{ is increasing}.  
    \\
\label{eq:CC.prop.d}
    &\forall\, q>1: \quad \C(s) \leq \frac{q}{q{-}1} \CC(s|w) + \frac{4\,w^q}{q{-}1}.
\end{align}    
\end{subequations}
\end{lemma}

\begin{proof}
Property \eqref{eq:CC.prop.a} follows from the fact that $\CC(\cdot| w)$ is the Legrendre-Fenchel transform of $(\zeta,w)\mapsto w\C^*(\zeta)$ which is convex in $\zeta$ and concave in $w$. 

The relation in \eqref{eq:CC.prop.b} follow by a direct computation using the 
lower estimate for $s \C'(s)$ in \eqref{eq:Cprime.bound}. Assertion \eqref{eq:CC.prop.f} follows by using the upper bound for $s\C'(s)$ in \eqref{eq:Cprime.bound}. 
For the proof of the magical estimate 
\eqref{eq:CC.prop.d} we refer to Appendix \ref{app:CC.estim}. 
\end{proof}

We call \eqref{eq:CC.prop.d} the \emph{magical estimate for $\C$ and its perspective 
function $\CC$}, since such an estimate cannot be expected from general dissipation
functions. For instance, for $\Phi: s \mapsto c_1|s|^q$ with $q>1$ the infimum of 
$w\Phi(s/w) + c_2w^p$ only provides an upper bound for $s \mapsto c_3 |s|^r$ with 
$r=qp/(p{+}q{-} 1) \lneqq q$. 
The magical estimate \eqref{eq:CC.prop.d} will be important for
proving the chain rule, see the proof of Theorem \ref{thm:CRContinuous} in Section \ref{sec:ChainRuleCont}.

As a weaker replacement of \eqref{eq:CC.prop.d} we will need the following result
that is proved in Appendix \ref{app:Superlinear}.

\begin{lemma}[Superliner estimates]
\label{lem:Superlinear}
Consider an even, differentiable, and superlinear function $\phi:\R\to [0,\infty)$ 
such that $ s \phi'(s) \geq \phi(s)$ and 
another non-decreasing superlinear function $\psi:[0,\infty) \to [0,\infty)$.
Then, the function $\Xi_{\phi,\psi}:\R\to [0,\infty)$ defined via 
\[
\Xi_{\phi,\psi}(s) \coloneqq \inf_{w>0} \{w \phi(s/w) + \psi(w)\},
\]
is even, non-decreasing and superlinear. 
\end{lemma}

For $c>0$, the dual dissipation potential induces a slope term by the relation 
$S_N(c)=R^*_N(c,-\nablaDisc\rmD E_N(c))$. This definition can then be extended to all 
$c\in\calM_+(X_N)$ (cf.\ \cite[Remark 3.7]{lam2024variational}) by exploiting the identity
\eqref{eq:CStarandLog} for $\C^*$ and $\log$. This yields the so-called 
\textit{relaxed slope}.

\begin{definition}[Relaxed slope]
\label{def:slope_disc}
The \textit{relaxed slope} $S_N:\calM_+(X_N)\to[0,\infty)$ is defined by
\begin{align*}
        S_N(c) &\coloneqq S_{N,\diff}(c)+S_{N,\react}(c),
\\
        S_{N,\diff}(c) &\coloneqq  \frac{1}{N^d} \sum_{k\in\bbZ_N^d} 
        \sum_{i\in I}  \sum_{e\in \Eu} 2\delta_i N^2 
        \sqrt{\discref_{i,k+e}\discref_{i,k}}\bigg(\sqrt{\frac{c_{i,k+e}}
        {\discref_{i,k+e}}} - \sqrt{\frac{c_{i,k}}{\discref_{i,k}}}\bigg)^2,
\\
        S_{N,\react}(c) &\coloneqq  \frac{1}{N^d} \sum_{k\in\bbZ_N^d} 
        \sum_{r\in R} 2\kappa_r  \sqrt{\discref_k^{\alpha^r}\discref_k^{\beta^r}}
        \Bigg(\bigg(\frac{c_k} {\discref_k}\bigg)^{\alpha^r/2} 
         - \bigg(\frac{c_k}{\discref_k}\bigg)^{\beta^r/2}\Bigg)^2.
    \end{align*}
\end{definition}

Next, we introduce the primal dissipation potential $R_N$.

\begin{definition}[Primal dissipation potential]
\label{def:R_disc}
We define thr primal dissipation potential $R_N: \calM_+(X_N)\times \calM(Y_{N,\diff}) \times\calM(Y_{N,\react})\to[0,\infty)$ for $c \in\calM_+(X_N)$, 
$F \in \calM(Y_{N,\diff})$, and $J \in \calM(Y_{N,\react})$ by
\begin{align*}
        R_N(c,F,J)&\coloneqq 
        R_{N,\diff}(c,F) + R_{N,\react}(c,J),
\\
        R_{N,\diff}(c,F) &\coloneqq \frac{1}{N^d}\sum_{k\in\bbZ_N^d} 
        \sum_{i\in I} \sum_{e\in \Eu}  \CC\left(F^N_{i,k,e} \middle| 
                  N^2\delta_{i}(c_{i,k}^Nc_{i,k+e}^N)^{1/2}\right), 
\\
        R_{N,\react}(c,J) &\coloneqq \frac{1}{N^d}\sum_{k\in\bbZ_N^d} 
        \sum_{r\in R}\CC\left(J_{r,k}^N \middle| 
              \kappa_r(c_k^{N})^{(\alpha^r+ \beta^r)/2}\right).
    \end{align*}
\end{definition}
Note that $R_N$ is the convex conjugate of $R_N^\ast$ with respect to the second and 
third arguments. Indeed, the dualities $\big(R_{N,\diff}^*\big)^*=R_{N,\diff}$ and 
$\big(R_{N,\react}^*\big)^*=R_{N,\react}$ follow from the duality of 
$\C$ and $\C^\ast$. The duality of the sums then follows from the fact that the 
summands are independent of each other.

\begin{definition}[Energy-dissipation functional]\label{def:EDF_disc}
    We introduce the dissipation functional $D_N:\rmL^1(0,T;\calM_+(X_N)\times \calM(Y_{N,\diff})\times\calM(Y_{N,\react})) \to [0,\infty]$ by
    \begin{align}\label{eq:DiscreteD}
        D_N(c,F,J) \coloneqq \int_0^T\{R_N(c(t),F(t),J(t)) + S_N(c(t))\} \dd t.
    \end{align}
    Furthermore, we introduce the energy-dissipation functional
    \begin{align*}
        L_N(c,F,J) \coloneqq E_N(c(T))- E_N(c(0)) + D_N(c,F,J).
    \end{align*}
\end{definition}

\begin{definition}[Continuity equation]\label{def:CE_disc}
    The operator $\nablaDisc$ gives rise to the continuity equation
    \begin{align*}
        \dot c =   \nablaDisc^*(F,J) = - \divDDiff F +  
          \Gamma^* J.
    \end{align*} 
    We denote by $\lCE_N$ the set of triples $(c,F,J)\in\AC([0,T];\calM(X_N))\times \rmL^1(0,T;\calM(Y_{N,\diff})) \times \rmL^1(0,T;\calM(Y_{N,\react}))$ satisfying the above equation. 
\end{definition}

 In this definition the domains of the sets $X_N$, $Y_{N,\diff}$, and $ Y_{N,\react}$ are finite sets, such that the topology for the measure spaces is irrelevant. Only in the continuous case, it will be important to use the the narrow topology, see Definition \ref{def:CE_cont}. Moreover, following the proof of \cite[Lem.\,3.1]{Erbar2014}, we observe that this definition is indeed well-posed for $F$ and $J$ satisfying $\rmL^1$-bounds in time.

For further reference, we note that $(c,F,J)\in\lCE_N$ if and only if for all $\varphi\in \rmC(X_N)$ and all $k\in\bbZ^d_N$ and a.e. $t\in[0,T]$ it holds
\begin{align*}
    \frac{\rmd}{\rmd t}\pra[\bigg]{\sum_{i\in I}\varphi_{i,k} c_{i,k}} = \sum_{i\in I}\sum_{e\in\Eu}(\nablaDDiff\varphi)_{i,k,e} F_{i,k,e} + \sum_{r\in R}(\nablaDReact\varphi)_{r,k}J_{r,k}.
\end{align*}

We conclude this section by specifying our notion of solution for 
\eqref{eq:RDSysDisc}.  For the subsequent analysis, it will be crucial to use
already the notion of \emph{energy-dissipation balance solutions} (in short EDB solutions) that are based on the energy-dissipation functional $L_N$. Theorem \ref{thm:DiscEDB.ODE} will provide a rigorous connection between this notion and the ODE system \eqref{eq:RDSysDisc} based on the corresponding chain rule as explained in Section \ref{subsec:GradSyst}. Even in this finite-dimensional case, this equivalence is
non-trivial.

\begin{definition}[ Discrete EDB solutions]
\label{def:sol_disc}
We say $c\in \AC([0,T];\calM_+(X_N))$ is a \emph{discrete EDB solution} of 
\eqref{eq:RDSysDisc} with initial datum $c_0\in \calM_+(X_N)$, if $c(0) = c_0$ and  
if there exists $(F,J)$ such that $(c,F,J) \in \lCE_N$ and for $0\leq s < t\leq T$ 
we have 
\begin{equation}
  \label{eq:DiscEDB.st}
L_N^{[s,t]}(c,F,J) \coloneqq E_N(c(t))- E_N(c(s)) + \int_s^t \big( R_N(c,F,J) 
 + S_N(c) \big) \dd r =0.  
\end{equation}
\end{definition}

\subsection{Continuous-space reaction-diffusion gradient system}

We present the gradient structure for the reaction-diffusion system for the case of the  `continuous space' given by the torus $\bbT^d$; as a short-hand we will use the name ``\emph{continuum system}''.   The base spaces we consider are  
\begin{align*}
    X&&&\coloneqq (X^\dom,X^\tar) &&\coloneqq (\bbT^d,\R^I),\\
    Y_\diff&&&\coloneqq (Y^\dom_\diff,X^\tar_\diff)&&\coloneqq (\bbT^d,(\R^d)^I),\\
    Y_\react&&&\coloneqq (Y^\dom_\react,X^\tar_\react)&&\coloneqq (\bbT^d,\R^R),\\
    Y&&&\coloneqq (Y^\dom,Y^\tar)&&\coloneqq (Y^\dom_\diff\times Y^\dom_\react,Y^\tar_\diff\times Y^\tar_\react).
\end{align*}
We recall the short notations $\rmC(X) \coloneqq \rmC(X^\dom;X^\tar)$ and $\rmC([0,T]\times X)\coloneqq \rmC([0,T]\times X^\dom;X^\tar)$ as well as analogous notations for all other spaces of functions/measures over these spaces.

Similar to before, we abuse notation, denoting by $\langle\cdot,\cdot\rangle$ the dual products for vectors as well as components, e.g., for $(\zeta,\xi)\in \rmC(Y)$ and $(u,v)\in \calM(Y)$ we write
\begin{align*}
    \langle(\xi,\zeta),(u,v)\rangle &= \langle \xi,u\rangle_N + \langle \zeta,v\rangle 
    = \sum_{i\in I} \langle \xi_i,u_i\rangle + \sum_{r\in R}\langle \zeta_r,v_r\rangle \\
    &= \int_{\bbT^d}\bigg(\sum_{i\in I}\sum_{l=1}^d \xi_{i,l}(x)u_{i,l}(x) + \sum_{r\in R}\zeta_r(x)v_r(x)\bigg)\dd x,
\end{align*}
and similarly for other functions/measures defined over $X$ or $Y$.

The notation for sums over $I$ and $R$ is used also in the continuous context:
\[
    \gamma\bullet\xi =\sum_{i\in I}\gamma_i\xi_i, \quad f\bullet \psi = \sum_{r \in R} f_r\psi_r. 
\]

Given stoichiometric vectors $\alpha^r,\beta^r\in [0,\infty)^I$, 
$\gamma^r =\alpha^r-\beta^r$, $r\in R$, we consider the (classical) gradient 
$\nabla$ and the stoichiometric matrix $\Gamma$ as well as the 
 linear gradient map  $\nablaCont$ given by
\begin{align*}
    \nablaCont&:\rmC^\infty_c(X) \to \rmC^\infty_c(Y),\quad 
    &&\nablaCont \varphi \coloneqq  (\nablaCDiff\varphi ,  
     \nablaDReact\varphi) \quad \text{with} 
     \\
    \nablaCDiff&:\rmC^\infty_c(X) \to \rmC^\infty_c(Y_\diff),\quad 
    &&\nablaCDiff\varphi_i(x) \coloneqq (\partial_{x_l}\varphi_i(x))_{l=1,\ldots,d},\\
    \nablaCReact&:\rmC^\infty_c(X) \to \rmC^\infty_c(Y_\react),\quad 
    &&\nablaCReact\varphi_r(x) \coloneqq \sum_{i\in I}\gamma^r_i\varphi_i(x) = \gamma^r\bullet\varphi(x). 
\end{align*}
Their duals are given by
\begin{align*}
    \nablaCont^\ast&:(\rmC^\infty_c(Y))^* \to (\rmC^\infty_c(X))^*,\quad 
    &&\nablaCont^\ast(\xi,\zeta) \coloneqq  -\divCDiff\xi + \Gamma^\ast\zeta
      \quad \text{with}    \\
    -\divCDiff&:(\rmC^\infty_c(Y_\diff))^* \to (\rmC^\infty_c(X))^*,\quad 
    &&-\divCDiff\xi_i(x) \coloneqq -\sum_{l=1}^d\partial_{x_l}\xi_{i,l}(x),\\
    \Gamma^\ast&:(\rmC^\infty_c(Y_\react))^* \to (\rmC^\infty_c(X))^*,\quad 
    &&\Gamma^\ast\zeta_r(x) \coloneqq \sum_{i\in I}\gamma^r_i\zeta_i(x), = \gamma^r\bullet\zeta(x). 
\end{align*}
    
Elements of the state space $\calM_+(X)$ are denoted by $\rho$ and will be called \emph{(continuous-space) chemical concentrations}. In our situation the measures will always have a densitiy with respect to the Lebesgue measure on $\bbT^d$, which (slightly abusing notation) will also be denoted by $\rho$.

    We consider the \emph{relative entropy} with respect to a \emph{reference measure} $\omega\in\calM_+(X)$ 
    \begin{equation}\label{eq:torus_energy}
        \cE(\rho) =\sum_{i\in I}\int_{\bbT^d}\LB \bra[\bigg]{\frac{\rho_i}{\omega_i}}\omega_i\dd x,
    \end{equation}
    where we recall the Boltzmann function $\LB (r)=r\log r-r+1$.

    The \emph{continuous dual dissipation potential} $\calR^*: \calM_+(X)\times \rmC(Y)\to[0,\infty)$ is given for $\rho \in\calM_+(X)$, $\xi \in \rmC(Y_\diff)$, and $\zeta \in \rmC(Y_\react)$ by
    \begin{align*}
        \calR^*(\rho,(\xi,\zeta)) &\coloneqq \calR^*_\diff (\rho,\xi) + \calR^*_\react (\rho,\zeta),\\
        \calR^*_\diff (\rho,\xi) &\coloneqq \sum_{i\in I}\frac{\delta_i}{2} \int_{\bbT^d}\abs{\xi_i}^2\rho_i\dd x,\\
        \calR^*_\react (\rho,\zeta) &\coloneqq  \sum_{r\in R} \kappa_r \int_{\bbT^d}\bra[\big]{\rho^{\alpha^r}\rho^{\beta^r}}^{1/2} \C^*(\zeta_r)\dd x.
    \end{align*}
    As in the discrete setting, from now on we will write $\calR^*(\rho,\xi,\zeta)$ instead of $\calR^*\big(\rho,(\xi,\zeta)\big)$ and do the same for similar objects. 
    
    These objects form the continuous gradient system in continuity format $(X,Y,\nablaCont,\calE,\calR^\ast)$.

For $\rho>0$ smooth, the dual dissipation potential induces a slope term by the relation $\calS(\rho)=\calR^*(\rho,-\nablaCont\rmD \calE(\rho))$. This definition can then be extended to all $\rho\in\calM_+(X)$ (cf.\ \cite[(3.24)]{heinze2024gradient}) yielding the relaxed slope.

\begin{definition}[Relaxed slope]\label{def:slope_cont}
    The relaxed slope $\calS:\calM_+(X)\to[0,\infty)$ is defined by
    \begin{align*}
        \calS(\rho) &\coloneqq \calS_\diff(\rho)+\calS_\react(\rho),\\
        \calS_\diff(\rho) &\coloneqq  \sum_{i\in I} 2\delta_i \int_{\bbT^d}\abs[\bigg]{\nabla\sqrt{\frac{\rho_i}{\omega_i}}}^2\omega_i\dd x,\\
        \calS_\react(\rho) &\coloneqq \sum_{r\in R} 2\kappa_r \int_{\bbT^d} \sqrt{\omega^{\alpha^r}\omega^{\beta^r}} \Bigg(\bra[\bigg]{\frac{\rho}{\omega}}^{\alpha^r/2} - \bra[\bigg]{\frac{\rho}{\omega}}^{\beta^r/2}\Bigg)^2\dd x.
    \end{align*}
\end{definition}

Next, we introduce the primal dissipation potential $\calR: \calM_+(X)\times \calM(Y)\to[0,\infty)$, which, as before, is given as the convex conjugate of $\calR^*$ with respect to the second argument.
\begin{definition}[Primal dissipation potential]\label{def:R_cont}
    The primal dissipation potential $\calR: \calM_+(X)\times \calM(Y_\diff)\times \calM(Y_\react)\to[0,\infty)$ is defined for $\rho \in\calM_+(X)$, $f \in \calM(Y_\diff)$, and $j \in \calM(Y_\react)$ by
    \begin{align*}
        \calR(\rho,f,j)&\coloneqq 
        \calR_\diff(\rho,f) + \calR_\react(c,j),\\
        \calR_\diff(\rho,f) &\coloneqq 
        \begin{cases}
            \sum_{i\in I}\frac{1}{2\delta_i }\int_{\bbT^d}\frac{|f_i|^2}{\rho_i}
            \dd x  &\text{for } f_i = f_i\dd x\ll \scrL^d,\\
            \infty & \text{otherwise},
        \end{cases}\\
        \calR_\react(\rho,j) &\coloneqq 
        \begin{cases}
            \sum_{r\in R} \int_{\bbT^d}\CC\left(j_r \middle|\kappa_r(\rho^{\alpha^r}\rho^{\beta^r})^{1/2}\right)\dd x 
             & \text{for }j_r = j_r\rmd x\ll \scrL^d,\\
            \infty & \text{otherwise},
        \end{cases}
    \end{align*}
    where we again made a slight abuse of notation.
\end{definition}

\begin{definition}[Energy-dissipation functional]\label{def:EDF_cont}
    We introduce the dissipation functional $\calD:\rmL^1(0,T;\calM_+(X)\times \calM(Y_\diff)\times \calM(Y_\react)) \to \R$ by
    \begin{align}\label{eq:ContinuousD}
        \calD(\rho,f,j) \coloneqq \int_0^T\{\calR(\rho(t),f(t),j(t)) + \calS(\rho(t))\} \dd t.
    \end{align}
    Furthermore, we introduce the energy-dissipation functional
    \begin{align*}
        \calL(\rho,f,j) \coloneqq \calE(\rho(T))- \calE(\rho(0)) + \calD(\rho,f,j).
    \end{align*}
\end{definition}

\begin{definition}[Continuity equation]\label{def:CE_cont}
    The operator $\nablaCont$ gives rise to the continuity equation
    \begin{align*}
        \partial_t \rho =  \nablaCont^\ast(f,j)= -\divCDiff f + 
         \nablaCReact^* j. 
    \end{align*}
    We denote by $\CE$ the set of triples $(\rho,f,j)\in\AC([0,T];\calM_+(X)) \times \rmL^1(0,T;\calM(Y_\diff)) \times\\ \rmL^1(0,T;\calM(Y_\react))$ 
    satisfying for all $\varphi\in \rmC^1(X)$
    \begin{align*}
        \frac{\rmd}{\rmd t}\pra[\bigg]{\int_{\bbT^d}\sum_{i\in I}\varphi_i \rho_i \dd x} 
        = \int_{\bbT^d}\sum_{i\in I}\sum_{e\in\Eu} (\nablaCDiff\varphi)_{i,e} f_{i,e} \dd x 
        + \int_{\bbT^d}\sum_{r\in R}(\nablaCReact\varphi)_r j_r\dd x.
    \end{align*}
\end{definition}

 It is important here to recall that the absolute continuity in 
$\AC([0,T];\calM_+(X))$ has to be understood with respect to a metric generating 
the the narrow topology in $\calM_+(X)$. Moreover, $ \rmL^1(0,T;\calM(Y_\diff))$ is 
meant to contain weakly measurable functions with 
$t\mapsto \|f(t)\|_{\calM(Y_\diff)}$ 
lies in $\rmL^1([0,T])$. 

Combining the proofs of \cite[Lem.\,8.1.2]{AGS} and \cite[Lem.\,3.1]{Erbar2014}, we observe that this definition is well-posed for $f$ and $j$ satisfying $\rmL^1$-bounds.

We conclude this section by specifying our notion of solutions for the 
continuous-space reaction-diffusion gradient system \eqref{eq:RDSysCont}. 

\begin{definition}[Continuum EDB solutions for \eqref{eq:RDSysCont}]
\label{def:sol_cont}
We say $\rho\in \AC([0,T];\calM_+(X))$ is a \emph{continuum EDB solution} of \eqref{eq:RDSysCont} if $\sup_{t\in [0,T]} \calE(\rho(t)) < \infty$ and if there exists $(f,j)$ such that $(\rho,f,j) \in\CE$,
$\calD(\rho,f,j)<\infty$, and for $0\leq s < t\leq T $ we have 
\begin{align*}
\calL^{[s,t]}(\rho,f,j)\coloneqq \calE(\rho(t)) - \calE(\rho(s)) + \int_s^t \!\! 
 \big( \calR(\rho,f,j)+\calS(\rho)\big) \dd r =0. 
\end{align*}
\end{definition}

In the present paper, we will not show that all continuum EDB solutions $\rho$ are 
weak solutions (in a suitable sense). However, under the additional assumption of positivity and boundedness for all $\rho_i$, Proposition \ref{prop:CEDB.vs.CRDS} provides a result in this direction. Instead, we focus on the convergence of
discrete EDB solutions $c^N$ in the sense of Definition \ref{def:sol_disc} to 
continuum EDB solutions. In fact, we establish the stronger EDP-convergence which also asks convergence of $E_N$ to $\calE$ and $D_N$ to $\calD$. 

The strategy is as explained in Section \ref{ss:GeneralConvergenceStrategy}. 
By a limit passage we obtain $\calL^{[0,T]}(\rho,f,j)\leq 0$, see the 
lower-limit estimates in Section \ref{s:ProofConvergence}.  In Section \ref{s:CRs} 
we establish  the chain rule estimate $\calL^{[s,t]}(\rho,f,j)\geq 0$ which then 
implies that $\rho$ is a continuum EDB solution. In Section  \ref{sec:MainResults} 
we state the precise assumptions and results. 

We close this subsection with stating a conditioned Energy-Dissipation Principle. 
If we have lower and upper bounds of the densities, then it follows that
functions are continuum EDB solutions if and only if they are weak solutions. 

\begin{proposition}[Continuum EDB and weak solutions for \eqref{eq:RDSysCont}]
\label{prop:CEDB.vs.CRDS}  
Consider concentrations $\rho \in \rmH^1([0,T];\rmH^{-1}(\bbT^d))$ $\cap$ $ \rmL^2([0,T];\rmH^1(\bbT^d))$ 
and $\sigma \in (0,1)$ such that  $\rho$ 
satisfies $\rho_i(t,x)\in [\sigma , 1/\sigma ]$ for all $i\in I$ and a.a.\ 
$(t,x)\in [0,T]\times \bbT^d$. Then, $\rho$ is a weak solution of 
\eqref{eq:RDSysCont} if and only if $(\rho,f,j)\in \CE$ with 
\begin{equation} 
 \label{eq:CR_cont}
f_i = - \delta_i \nabla \rho_i \quad \text{and} \quad
j_r = \kappa_r \omega^{(\alpha^r+\beta^r)/2}  \Big( \frac{\rho^{\alpha^r}}
{\omega^{\alpha^r} }  - \frac{\rho^{\beta^r}}{\omega^{\beta^r}} \Big), 
\end{equation}
 is a continuum EDB solution in the sense of Definition \ref{def:sol_cont}.
\end{proposition}

The proof is given in Section \ref{sec:ChainRuleCont}.

\subsection{Embedding}
\label{sec:Embedding} 
As a crucial step for obtaining the convergence of gradient systems, we highlighted in Section~\ref{ss:GeneralConvergenceStrategy} the construction of suitable embeddings connecting the prelimit spaces to the limit space.

\begin{definition}[Embedding and discretization operators]\label{def:const_embedding}
    We introduce for each $k\in\bbZ_N^d$ the cube 
    \begin{align*}
        Q^N_k \coloneqq \set{x\in\bbT^d: x_l \in [k_l/N,(k_l+1)/N), l=1,\ldots,d}.
    \end{align*}
    Next, we define the embedding operator $\iota_N:\calM(X_N)\to \calM(X)$ by
    \begin{align*}
        \iota_N (c_i)(x) &\coloneqq \rho_i(x) \coloneqq \sum_{k\in\bbZ_N^d} c_{i,k}\one_{Q^N_k}(x).
    \end{align*} 
Dual to it, we introduce the discretization operator $\iota^*_N:C(X)\to \rmC(X_N)$ 
by setting
\begin{align*}    \iota^*_N\xi_{i,k}\coloneqq(\iota^*_N\xi_i)_k&\coloneqq N^d\int_{Q^N_k} \xi_i\dd x.
\end{align*}
For the diffusive fluxes we introduce $\iota_{N,\diff}:\calM_+(Y_{N,\diff}) 
\to \calM(Y_\diff)$ defined by
\begin{subequations}
    \label{eq:def.jotaFlux}
\begin{equation}
    \label{eq:def.iotaFlux.diff}
         \iota_{N,\diff} F_i(x)\coloneqq f_i(x) = (f_{i,e_1}(x),\ldots,f_{i,e_d}(x))^\top
\end{equation}
where for $ e \in E$ we set
\begin{align}
    \label{eq:def.iotaFlux.diff.1D}
        f_{i,e}(x)= \frac{1}{N}\sum_{k\in\bbZ_N^d}\bigg(\int_0^1\one_{Q^N_{k+\theta e}}(x)\dd \theta\bigg) F_{i,k,e}.
\end{align}
Finally, for the reactive fluxes we define $\iota_{N,\react}:\calM_+(Y_{N,\react})\to \calM(Y_\react)$ by
\begin{equation}
    \label{eq:def.iotaFlux.react}
\iota_{N,\react} J_r(x)\coloneqq j_r(x)\coloneqq\sum_{k\in\bbZ_N^d}\one_{Q^N_k}(x) J_{r,k}.
\end{equation}        
\end{subequations}
\end{definition}
Using the above embedding operator is is clear that $E_N(c) = \mathcal{E}(\iota_N(c))$ and analogous identities hold for other integral functionals as well. 
Moreover, by construction, $ s \mapsto f_{i,e}(x{+}se)$ is piecewise affine, whereas
$s \mapsto f_{i,e}(x {+} s \hat e)$ is piecewise constant for $\hat e \neq e$.  
Moreover, by definition we have for all $i\in I$ and all $e\in\Eu$ the estimate
\begin{align}\label{eq:embedded-flux-bound}
    \int_{\bbT^d}\abs{f_{i,e}}(x)\dd x\le \frac{1}
    {N}\int_0^1\sum_{k\in\bbZ_N^d}\int_{\bbT^d}\one_{Q^N_{j+e\theta}}(x)
    \dd x\dd \theta \abs{F_{i,k,e}} \leq 
    \frac{1}{N^d}\sum_{k\in\bbZ_N^d}\frac{\abs{F_{i,k,e}}}{N},
\end{align}
and for all $r\in R$ the estimate
\begin{align*}
    \int_{\bbT^d}\abs{j_r}(x)\dd x\le \frac{1}{N^d}\sum_{k\in\bbZ_N^d}\abs{J_{r,k}}.
\end{align*}
Denoting the discrete $\rmL^1$-norms on $\rmL^1(\bbZ^d_N)$ by 
$\|G\|_{\rmL^1_N}\coloneqq\frac{1}{N^d}\sum_{k\in\bbZ^d_N}|G_k|$, 
the estimates can be written equivalently as
\[
\forall\, i\in I,e\in \Eu:\quad \|f_{i,e}\|_{\rmL^1}\leq \frac 1 N 
\|F_{i,e}\|_{\rmL^1_N},\quad \forall\, r\in R:\quad \|j_r\|_1\leq \|J_r\|_{\rmL^1_N}.
\]
To simplify notation, for all $N\in\bbN$ and $y\in\R^d$ we introduce the 
\emph{shift-operator} $\shift^N_y$ defined by
\begin{align}\label{eq:shift}
    \shift^N_y:\rmL^1(\bbT^d)\to \rmL^1(\bbT^d),\quad \shift^N_y\phi(x)\coloneqq\phi\Big(\Big(x+\frac y N\Big)\!\!\!\!\!\mod 1\Big),
\end{align}
where the modulus is applied componentwise.

The following lemma uses the embeddings defined above to connect the discrete and continuous continuity equations  (recall $\lCE_N$ from Definition \ref{def:CE_disc}).  

\begin{lemma}
\label{lem:CE.DiscCont}
    For each $c\in\calM_+(X_N)$, $\varphi\in \rmC(X)$, it holds \begin{align*}
        \langle\iota_N c,\varphi\rangle =\langle c,\iota_N^* \varphi\rangle_N.
    \end{align*}
    Furthermore, we have for $\varphi\in \rmC(X)$, $(F,J)\in\calM(Y_N)$ that
    \begin{align*}
        \langle\iota_{N,\diff} F,\nabla\varphi\rangle =\langle F,\nablaDDiff\iota_N^* \varphi\rangle_N \quad \text{and} \quad 
        \langle\iota_{N,\react} J,\nablaCReact\varphi\rangle =\langle J,\nablaDReact\iota_N^* \varphi\rangle_N.
    \end{align*}
    In particular, it holds $(c,F,J)\in\lCE_N$ if and only if $(\iota_N c,\iota_{N,\diff}F,\iota_{N,\react}J) \in\CE$.
\end{lemma}
\begin{proof}
For the first equality, we calculate
\begin{align*}
    \langle\xi_i,\iota_N c_i\rangle = \int_{\bbT^d} \xi_i \sum_{k\in\bbZ_N^d} c_{i,k}\one_{Q^N_k}\dd x = \frac{1}{N^d}\sum_{k\in\bbZ_N^d} \iota^*_N \xi_{i,k}  c_{i,k} = \langle\iota^*_N(\xi_i), c_i\rangle_N.
\end{align*}
Furthermore, we have for all test functions $\varphi\in \rmC^\infty_0(X)$ that
\begin{align*}
    \langle F,\overline\nabla\iota_N^\ast \varphi \rangle _N
    &= \frac{1}{N^d}\sum_{k\in\bbZ_N^d}\sum_{e\in \Eu} F_{k,e} \bullet[(\iota_N^\ast \varphi)_{k+e}-(\iota_N^\ast \varphi)_k]
    =\sum_{k\in\bbZ_N^d}\sum_{e\in \Eu} F_{k,e}\bullet \int_{Q^N_k} [\shift^N_e\varphi-\varphi]\dd x.
\end{align*}
Rewriting for each $i\in I$ the integral by
\begin{align}\label{eq:HSDI}
    \shift^N_e\varphi_i(x)-\varphi_i(x) &= \int_0^1 \nabla\varphi_i\bigg(x+\frac{e}{N}\theta\bigg) \cdot \frac{e}{N}\dd \theta= \frac1N\int_0^1 \shift^N_{e\theta}\partial_{x_e}\varphi_i(x) \dd \theta,
\end{align}
 and using Fubini for the integrable integrand, we get
\begin{align*}
    &\langle F,\overline\nabla\iota_N^\ast \varphi \rangle_N 
    = \sum_{k\in\bbZ_N^d}\sum_{e\in \Eu} \frac{F_{k,e}}{N} \bullet\int_{Q^N_k}\int_0^1 \shift^N_{e\theta}\partial_{x_e}\varphi(x) \dd \theta\dd x
\\
    & =  \sum_{k\in\bbZ_N^d}\sum_{e\in \Eu} \frac{F_{k,e}}{N} \bullet\int_{\bbT^d}\one_{Q^N_k}(x)\int_0^1 \shift^N_{e\theta}\partial_{x_e}\varphi(x) \dd \theta\dd x
\\   
  &  =  \sum_{k\in\bbZ_N^d}\sum_{e\in \Eu} \frac{F_{k,e}}{N} 
  \bullet\int_{\bbT^d}\int_0^1 \one_{Q^N_{j+e\theta}}(x)
   \partial_{x_e}\varphi(x) \dd \theta\dd x
= \int_{\bbT^d}\sum_{e\in \Eu} f_e(x) \bullet\partial_{x_e}\varphi(x) \dd x =\langle f^N,\nabla \varphi\rangle .
\end{align*}
For the reactive flux, we simply observe that
\begin{align*}
\langle J, \Gamma \iota^*_N \varphi \rangle_N    
    &= \frac{1}{N^d} \sum_{k\in\bbZ_N^d} (\Gamma^TJ)\bullet N^d\int_{Q^N_k} 
    \varphi\dd x 
    = \sum_{k\in\bbZ_N^d} J_k\bullet \int_{Q^N_k} \Gamma \varphi\dd x 
    =\langle j, \Gamma \varphi \rangle .
\end{align*}
In particular, for all $\varphi\in \rmC^\infty(\bbT^d)$ we obtain 
$\langle \dot c, \iota_N^*\varphi\rangle_N = -\langle\divDDiff F, \iota_N^*\varphi \rangle_N + \langle\nablaDReact^\ast J, \iota_N^*\varphi \rangle_N$ 
if and only if
$\langle \partial_t \iota_N c, \varphi\rangle = -\langle \div f, \varphi \rangle + \langle\nablaCReact^\ast j,\varphi \rangle$.

 This finishes the proof of Lemma \ref{lem:CE.DiscCont}.   
\end{proof}

\section{Main results}
\label{sec:MainResults}

Before we state our main results, we fix the assumptions on our problem.

\begin{assumption}[General assumptions]\label{ass:General}
    The continuous reference measure has a density  $\omega\in \rmC(\bbT^d,\R^I)$ and there exists $\omega_*,\ \omega^*$ such that for all $x\in\bbT^d$, $i\in I$ it holds
    \begin{align}
    \label{as_general:cont_reference_bound}
    \tag{3.G1}
        0<   \omega_* \leq \omega_i(x)\leq \omega^* <\infty.
    \end{align}
    The diffusion and reaction coefficients satisfy
    \begin{align}
    \label{as_general:pos_coeff}
    \tag{3.G2}
        \forall\, r\in R:\quad\kappa_r>0
        \qquad\text{and}\qquad
        \forall\, i\in I:\quad\delta_i \geq \delta_*>0.
    \end{align}
\end{assumption}
We emphasize that our analysis carries over without difficulty to diffusion and reaction coefficients that are non-constant in space, but are continuous and uniformly bounded above and away from zero.

For each $N\in\bbN$ we define the discrete reference measure $\discref^N = (\discref^N_i)_{i\in I}\in\calP(X_N)$ by
\begin{align*}
    \discref^N_{i,k} \coloneqq \iota_N^\ast\omega_i = N^d\int_{Q^N_k} \omega_i\dd x.
\end{align*}
We immediately observe that the bounds on the continuous reference measure translate uniformly to all discrete reference measures, i.e., for all $N\in\bbN$, $i\in I$, and $k\in\bbZ^d_N$ it holds
\begin{align}\label{as:reference_bound}
  0<\omega_*\leq   \discref_{i,k}^N\leq\omega^*<\infty.
\end{align}
Moreover, we easily obtain the following convergences
\begin{subequations}   
\begin{align}
    \iota_N \discref^N &\to \omega \text{ strongly in } \rmL^\infty(X)\label{eq:ConvergenceRefMeas},\\
    \forall\, e\in E\  \forall\, i\in I: \quad 
    \shift_e^N(\iota_N w^N_i)&\to \omega_i \text{ strongly in } \rmL^\infty(\bbT^d).
\end{align}
\end{subequations}

The above natural Assumptions \ref{ass:General} as well as their simple consequence are used throughout the paper without always referring to them. 

With this notation and under these general assumptions, we can formulate the energy-dissipation principle for the discrete system.  Here we follow an idea in \cite[Thm.\,4.16]{PRST22} and consider the function 
\[
B(c,v)= \sum_{i\in I}\sum_{k\in \bbZ^d_N} b\big(\tfrac{c_{i,k}}{w^N_{i,k}},
v_{i,k}\big) \quad  \text{with } b(a,s)= 
    \begin{cases} s \log  a &\text{for } a>0, \\
      0 & \text{for } a=0.     \end{cases}
\]
The special treatment of the singularity of $\log c_{i,k}$ at $c_{i,k}$ leads
to nontrivial implications that can only be handled due to the property
that the underlying (discrete) reaction-diffusion system preserves
non-negativity or, even more, positivity. 

\begin{lemma}[Chain rule for the discrete setting]
\label{lem:ChainRule.Discr} Let $c\in \AC([0,T];\rmL^1(X_N))$ be such that 
\[
t \mapsto B(c(t),\dot c(t)) \ \text{ lies in } \rmL^1(0,T).
\]
Then, $t \mapsto E_N(c(t))$ is absolutely continuous and we have the chain rule
formula
\begin{equation}
  \label{eq:CR.Discr.2}
  \frac{\rmd}{\rmd t} E_N(c(t)) = B(c(t),\dot c(t)) \ \text{ for a.a. } t \in
  [0,T]. 
\end{equation}
In particular, every curve $ (c,F,J) \in \lCE$ with $ D_N(c,F,J)< \infty$ satisfies 
\[
E_N(c(t)) - E_N(c(s)) = \int_s^t B(c(r)), \dot c(r)) \dd r \geq 
- \int_s^t \big( R_N(c(r),F(r),J(r)) + S_N(c(r)) \big) \dd r. 
\]
\end{lemma} 

This result will be a consequence of the more detailed Proposition \ref{prop:ChainRule.Discr}. With this chain rule it is then possible to show that 
discrete EDB solutions are equivalent to ODE solutions, i.e., in the discrete setting
the Energy-Dissipoation Principle holds.

\begin{theorem}[Discrete EDB and \eqref{eq:RDSysDisc}]
\label{thm:DiscEDB.ODE}
A function $ c \in \AC([0,T];\calM_+(X_N))$ is a solution to the discrete
reaction-diffusion system \eqref{eq:RDSysDisc} if and only if the triple 
$(c,F,J)$ with $F$ and $J$ given by \eqref{eq:DiffReactFlux} 
is a discrete EDB solution in the sense of Definition \ref{def:sol_disc}. 
\end{theorem}

We refer to the end of Section \ref{ss:Proof.Disc.EDP} for the proof. 
\bigskip

We now turn to the continuum system, where we need to restrict the stoichiometric 
vectors $\alpha^r$ and $\beta^r$, which was not the case in the discrete setting. 
At the end of this section we will shortly address the case where we have  a priori 
bounds in $\rmL^\infty$, which is again a case, where arbitrary stoichiometric 
vectors are allowed. 

In our analysis we will use two levels of assumptions: the first is needed for
deriving the lower-limit estimates and the second, which  is slightly stronger, will be used to derive the abstract chain rule.

\begin{assumption}[for lower-limit estimates]
The reaction coefficients satisfy
\begin{align}\tag{4.A1}
  \forall\, r\in R: \quad \frac12\big| \alpha^r+ \beta^r\big|_1  
     \leq {  \pcrit  \coloneqq 1+ 2/d. } 
\label{ass:Reactions1}  
\end{align}
\end{assumption}

\begin{assumption}[for chain rule inequality]
The reaction coefficients satisfy  
\begin{align}\tag{4.A2}
    \forall\, r\in R:\quad &|\alpha^r|_1\leq \pcrit, \quad |\beta^r|_1\leq \pcrit, \quad \frac12|\alpha^r{+}\beta^r|_1 \lneqq \pcrit. 
\label{ass:Reactions2*}
\end{align}
\end{assumption}

\begin{example}
 In all space dimensions we have $\pcrit>1$. Hence, our analysis covers 
linear exchange reactions
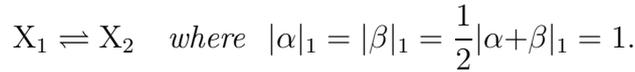
\[
\rmX_1 \rightleftharpoons \rmX_2 \quad \text{where } \ |\alpha|_1=|\beta|_1=\frac12|\alpha{+}\beta|_1=1.
\]
In space dimensions $d\leq 2$ we have $\pcrit=3$ or $\pcrit=2$, which allows to 
handle binary reactions with $j=c_1c_2-c_3$, i.e.,
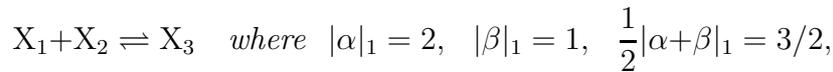
\[
\rmX_1{+}\rmX_2 \rightleftharpoons \rmX_3 \quad \text{where } \ |\alpha|_1=2, \ \ |\beta|_1=1,\ \ 
\frac12|\alpha{+}\beta|_1=3/2,
\]
or the semi-conductor reaction with $j = c_\mathrm{neq}c_\mathrm{pos}-1$, 
i.e.,
\[
\rmX_\mathrm{neg} + \rmX_\mathrm{pos} \rightleftharpoons \emptyset\quad 
\text{where }|\alpha|_1=2, \ \ |\beta|_1=0,\ \ \frac12|\alpha{+}\beta|_1=1.
\]
\end{example}

Having fixed the assumptions, we now state the convergence of the discrete gradient systems to the continuum gradient system.

\begin{theorem}[Convergence and lower limit of energy-dissipation functionals]
\label{thm:Convergence} 
Consider \ \ \ \linebreak[3] 
$(c^N,F^N,J^N)\in\overline{\CE}_N$ such that the uniform bounds
$\sup_{N\in\N}\esssup_{t\in[0,T]}E_N(c^N(t)) 
<\infty$ and $\sup_{N\in\N}D_N(c^N,F^N,J^N) <\infty$ hold true.
Moreover, assume that the reactions satisfy \eqref{ass:Reactions1}.

Then, there exists $(\rho,f,j)\in\CE$ with $f\in\rmL^1([0,T]\times Y_\diff)$ and
$j\in\rmL^1([0,T]\times Y_\react)$ such that (up to a subsequence) we have 
$\iota_Nc^N\to \rho$ strongly in $\rmL^1([0,T]\times X)$, 
 $\iota_{N,\diff} F^N\rightharpoonup f$ weakly in 
$\rmL([0,T]\times Y_\diff)$, and $\iota_{N,\react}J^N 
\rightharpoonup j$ weakly in $\rmL^1([0,T]\times Y_\react)$.

Moreover, we have the lower limit inequalities
\[
    \liminf_{N\to\infty} D_N(c^N,F^N,J^N)\geq \calD(\rho,f,j) \quad 
    \text{and }\quad \liminf_{N\to\infty} E_N(c^N(t))\geq \calE(\rho(t))
     \text{ for all } t\in[0,T],
\]
for the functionals defined in \eqref{eq:discrete_energy}, \eqref{eq:DiscreteD},
\eqref{eq:torus_energy}, and \eqref{eq:ContinuousD}, respectively.

In particular, for well-prepared initial data, i.e.,
$\iota_N c^N(0)\to \rho(0)$ with 
$E_N(c^N(0)\to \calE(\rho(0))$, it holds
\begin{align*}
    \liminf_{N\to\infty} L_N(c^N,F^N,J^N)\geq \calL(\rho,f,j).
\end{align*} 
\end{theorem}

To conclude that the limit $(\rho,f,j)$ solves the limit gradient-flow equation, we need in addition a chain rule inequality for the continuous reaction-diffusion 
system. (Recall $\calD$ and $\calL$ from Definition \ref{def:EDF_cont}.)

\begin{theorem}[Chain rule inequality for continuum system]
\label{thm:CRContinuous}
Consider a curve $(\rho,f,j)\in\CE$ with
$\esssup_{t\in[0,T]}\calE(\rho(t)) <\infty$ and $\calD(\rho,f,j) <\infty$. 
In addition, assume that the reaction coefficients satisfy \eqref{ass:Reactions2*}. 

Then, for every $0\le s < t\le T$ it holds
\begin{align*}
    \calE(\rho(t)) - \calE(\rho(s)) + \int_s^t \calR(\rho(\tau),f(\tau),j(\tau)) + \calS(\rho(\tau)) \dd \tau \ge 0.
\end{align*}
Furthermore, it holds  $\calL(\rho,f,j) = 0$ if and only if $\rho$ is 
a continuum EDB solution of \eqref{eq:RDSysCont} in the sense of Definition~\ref{def:sol_cont}.  
\end{theorem}

The three theorems together imply that solutions of the discrete problems \eqref{eq:RDSysDisc} on $\bbZ^d_N$ starting from well-prepared initial data converge 
(after choosing a suitable subsequence) to solutions of the continuous reaction-diffusion system \eqref{eq:RDSysCont}. This is summarized in our final main result.

\begin{corollary}[Convergence of solutions]\label{cor:conv_sol}
    Assume that the reactions satisfy \eqref{ass:Reactions2*}.
    Let $\rho_0\in \rmL^1(X)$ satisfy $\calE(\rho_0)<\infty$. Let $(c_0^N)_{N\in\bbN}$ with $c^N_0\in\rmL^1(X_N)$ be well-prepared, i.e., let $\iota_Nc^N_0\to \rho_0$ in $\rmL^1(X)$ and $E_N(c^N_0)\to \calE(\rho_0)$ as $N\to\infty$.
    
    Then, for each $N\in\bbN$ there exists a solution $c^N\in\rmL^1(0,T;\calM_+(X_N))$ of \eqref{eq:RDSysDisc} on $\bbZ^d_N$ in the sense of Definition~\ref{def:sol_disc} with initial datum $c_0^N$. 
    
    Furthermore, (up to a subsequence) we have $\iota_N c^N\to\rho$ strongly in $\rmL^1([0,T]\times X)$, where $\rho$ is a solution of the gradient flow equation \eqref{eq:RDSysCont} on $\bbT^d$ in the sense of Definition~\ref{def:sol_cont} with initial datum $\rho_0$. 
\end{corollary}

The proofs of the main results are given in the next two sections: 
In Section \ref{s:ProofConvergence}, we show the necessary compactness and the lower limit of dissipation functionals leading to Theorem~\ref{thm:Convergence}.
In Section \ref{s:CRs} we show that the chain rules and energy-dissipation principles for both, the discrete and the continuous, reaction-diffusion systems hold. \bigskip

Finally, we comment on the restrictions on the stoichiometric vectors. In fact, they are needed for deriving suitable a priori bounds. If however, these bounds can be obtained by other means, then the conditions can be dropped completely. 

\begin{remark}[$\rmL^\infty$ bounds via bounding boxes]\label{rem:BoundBox} 
It can be easily checked that the proofs given below hold for general stoichiometric vectors $\alpha^r$ and $\beta^r$, if we know that the discrete solutions $c^N$ are bounded uniformly in $\rmL^\infty$. Indeed, in this case the limit solution $\rho$  
is also bounded in $\rmL^\infty$ and we can set $\pcrit = \infty$
and check that all proofs work similarly. \medskip

We highlight this fact since for several classes of reaction-diffusion systems there exist so-called positively 
invariant regions in the sense of \cite[Cha.\,14\S B]{Smol94SWRD}.  In the simplest 
case such a region is a rectangular set, also called bounding box:
\[
\mathsf B\coloneqq \prod_{i\in I} \big[0,b_i\big] \coloneqq\big\{\, c \in [0,\infty)^I\,
\big| \, 0\leq c_i \leq b_i \text{ for all }i\in I\,\big\}. 
\]
Positive invariance means that solutions starting inside a region
(i.e.\ $c(t,x)\in \mathsf B$) remain inside the region for all $t>0$. In the 
case of a box the invariance follows, if for $c\in \partial \mathsf B$ the 
reaction vector $R(c)$ points inwards, i.e., $c_i=0$ implies $R_i(c)\geq 0$ and
$c_i=b_i$ implies $R_i(c)\leq 0$. 

Consider a reaction systems where all reactions are of the type $\alpha_i X_i
\rightleftharpoons \beta_\imath X_\imath$ which is additionally in detailed
balance for $w=(w_i)_{i\in I}$. Then, it can be shown that $\mathsf B=\prod_I
[0,w_i]$ is indeed a bounding box. Often there is a family of 
detailed-balance equilibria $w$, which then allow for arbitrary large bounding 
boxes.
\end{remark}

\section{Proof of convergence}
\label{s:ProofConvergence}

The aim of this section is to prove Theorem~\ref{thm:Convergence}, the convergence of the discrete gradient systems to the continuum gradient system. 
We split the section in two parts, first focusing on the compactness in Section~\ref{sub:EstimComp}, before establishing the lower limit in Section \ref{sub:LowLimDiss}.

To show compactness, we rely on 
the $N$-uniform $\rmL^\infty$-bound for the energies and the $N$-uniform bound of the dissipation functionals to obtain suitable a priori estimates for the embedded discrete concentrations $\rho^N= \iota_N c^N$.
We introduce a new and efficient method to show equi-integrability of the fluxes $F^N$ and $J^N$ in Proposition \ref{prop:FluxBound}.
Finally, an argument based on the Aubin-Lions-Simon lemma allows us to derive 
strong compactness of $\rho^N$ in Proposition~\ref{prop:strong_compactness}.
One of the biggest advantages of our approach is its ability of handling non-convex
dependencies on $\rho^N$ of the dissipation functionals.

The lower limit inequalities are then obtained for each rate and each slope term, independently, relying either on Ioffe's liminf theorem or, for the diffusive rate, on a dualization argument.

Throughout this section, we fix a time horizon $T>0$ and denote by $\Omega_T\coloneqq[0,T]\times\bbT^d$ the parabolic cylinder.

\subsection{Compactness}
\label{sub:EstimComp}

We start our considerations from the $N$-uniform $\rmL^\infty$-bound on the energies and the $N$-uniform bound on the dissipations. We introduce the explicit constants $K^A_\rmx$ that will 
make it easier to see the influence of the different bounds throughout the section. 

We start with the a priori bounds
\begin{subequations}
    \label{eq:AprioriEstimBasic}
\begin{align}
K^E	& \coloneqq \sup_{N\in\bbN}\sup_{t\in[0,T]} E_N(c^N(t)) < \infty,
	\label{eq:AprioriEstimBEnergy}\\
K^D	&\coloneqq\sup_{N\in\bbN}D_N(c^N,F^N,J^N) < \infty.
	\label{eq:AprioriEstimBDiss}
\end{align}
\end{subequations}
In particular, these imply
\begin{subequations}	
\label{eq:AprioriEstimatesDissipation_parts}
\begin{align}
	\label{eq:AprioriEstimatesL1Bound}
     K^E_{\rmL^1} & \coloneqq \sup_{N\in \bbN} \sup_{t\in [0,T]} \| c^N\|_{\rmL^1_N} < \infty,
 \\   
 K^R_{\diff}    &\coloneqq\sup_{N\in\bbN}\int_0^T R_{N,\diff}(c^N(t),F^N(t))\dd t < \infty,
	\label{eq:AprioriEstimatesRateDiff}
\\
 K^R_{\react}	& \coloneqq\sup_{N\in\bbN}\int_0^T R_{N,\react}(c^N(t),J^N(t))
      \dd t < \infty,	\label{eq:AprioriEstimatesRateReact}
 \\
 K^S_\diff	& \coloneqq \sup_{N\in\bbN}\int_0^T S_{N,\diff}(c^N(t))\dd t < \infty,
	\label{eq:AprioriEstimatesSlopeDiff}
\\
 K^S_\react	&\coloneqq\sup_{N\in\bbN}\int_0^T S_{N,\react}(c^N(t))\dd t < \infty.
	\label{eq:AprioriEstimatesSlopeReact}
\end{align}
\end{subequations} 
Using the embeddings from Section \ref{sec:Embedding}, we define the curves 
\[
\rho^N\coloneqq\iota_Nc^N, \quad f^N\coloneqq\iota_{N,\diff}F^N, \quad j^N\coloneqq\iota_{N,\react}J^N.
\]
To derive strong relative compactness of $(\rho^N)_{N\in\bbN}$ in $\rmL^1([0,T]\times X)$, we rely on an Aubin-Lions-type result. 
Since these piecewise constant functions are not weakly differentiable  and we will later rely on Sobolev embeddings to obtain higher integrability,
we introduce a second interpolant $\tilde\rho^N$ via 
\begin{equation}
    \label{eq:Def.tilde.rho.N}
    \tilde \rho^N_i = \omega_i \big( \tilde\iota_N U_i^N\big)^2 \quad \text{ where } \quad U^N_{i,k}= 
    \Big( \frac{c_{i,k}}{w^N_{i,k}}\Big)^{1/2} ,
\end{equation}
where the linear interpolator $\tilde \iota_N$ generates continuous and piecewise polynomial functions $\tilde u^N=\tilde\iota_N U^N$, the derivatives of which can be 
controlled uniformly in $N$ by $K^S_\diff$. 

Employing an Aubin-Lions-Simons-type argument, we show relative compactness 
of $(\tilde\rho^N)_{N\in\bbN}$. We then conclude by showing that 
$\norm{\tilde\rho^N-\rho^N}_{\rmL^1([0,T]\times X)}\to 0$ as $N\to \infty$. 

We highlight that, we will be able to show that $\rho^N$ is bounded in an Orlicz space slightly better than $\rmL^{\pcrit }([0,T]\times \bbT^d)$ with $\pcrit =1+2/d$.

\begin{remark}
    Note that our particular choice for the auxiliary embedding $\tilde\iota_N$ is the 
    $d$-linear interpolation, though we stress that other interpolations are possible as long as Lemma~\ref{lem:spatial_reg_NEW} is provable. In particular, we believe it is possible to employ a similar argument for more general geometries when replacing the uniform grids $\bbZ^d_N$.

    Furthermore, we point out that strong $\rmL^1$ compactness of $\rho^N$ could also be obtained directly by applying \cite[Theorem 4.2]{rossi2003tightness} as is done, e.g. in \cite[Theorem 4.8]{HT23}. However, our method additionally allows us to obtain higher integrability as we demonstrate in Proposition~\ref{prop:integrability}. 
\end{remark}

\begin{definition}[Continuous embedding]
Let $M \coloneqq \{0,1\}^d$. For $m\in M$ we define the functions $\fff^N_m:\bbT^d\to [0,1]$ via 
    \begin{align*}
        \fff^N_m(x) = \prod_{k=1}^d 
        \begin{cases}
        N x_k & \text{for } m_k = 1\\
        1{-}N x_k & \text{ for } m_k = 0
        \end{cases} \ \text{ for } x \in Q^N_0 \quad \text{and} \quad 
        \fff^N_m(x)=0 \ \text{ otherwise}. 
    \end{align*}
Recalling the shift operator $\shift^N_y$ from \eqref{eq:shift}, 
we define the continuous embedding operator
\begin{align}\label{eq:ContinuousEmbedding}
       \tilde\iota_N (U^N_i)(x) \coloneqq \tilde u^N_i(x) 
         &\coloneqq \sum_{k\in\bbZ_N^d} \sum_{m\in M} U^N_{i,k+m} \fff^N_m(x{-}k/N)
\end{align}
and its dual  discretization operator
    \begin{align*}    
        \tilde\iota^*_N(\varphi_i)_k &= N^d\int_{Q^N_k} \sum_{m\in M} \shift^N_{-m}\varphi_i(x)  \shift^N_{-k} \fff^N_m(x) \dd x.
    \end{align*}
\end{definition}
The duality of $\tilde\iota_N$ and $\tilde\iota^*_N$ follows by a direct calculation: 
\begin{align*}
 &  \int_{\bbT^d} \varphi_i \tilde\iota_N U^N_i\dd x 
    = \sum_{k\in\bbZ_N^d} \sum_{m\in M} U^N_{i,k+m} \int_{Q^N_k} 
    \varphi_i(x)  \shift^N_{-k} \fff^N_m(x) \dd x\\
& = \sum_{k\in\bbZ_N^d} \sum_{m\in M} U^N_{i,k} \int_{Q^N_{k-m}} 
    \varphi_i(x)  \shift^N_{m-k} \fff^N_m(x) \dd x 
  = \sum_{k\in\bbZ_N^d} \sum_{m\in M} U^N_{i,k} \int_{Q^N_k} \shift^N_{-m} 
      \varphi_i(x)  \shift^N_{-k} \fff^N_m(x) \dd x\\
& = \frac{1}{N^d}\sum_{k\in\bbZ_N^d} U^N_{i,k} \: N^d \int_{Q^N_k} \sum_{m\in M} 
   \shift^N_{-m}\varphi_i(x)  \shift^N_{-k} \fff^N_m(x) \dd x
 = \frac{1}{N^d}\sum_{k\in\bbZ_N^d} \tilde\iota^*_N(\varphi_i)_k U^N_{i,k} .
 \end{align*}
To understand the usage of the functions $\fff^N_m$ better it is useful to define the functions 
\[
 \hhh^N_0(x) = \sum_{m\in M} S^N_m\fff^N_m(x)   \quad \text{and} \quad \hhh^N_k =\shift^N_{-k} \hhh^N_0.
 \]
 Then, all $\hhh^N_k$ are piecewise polynomial and \emph{continuous}, and  
 the simple interpolation formula 
 \[
 \tilde u^N_i = \tilde\iota_N U^N_i = \sum_{k\in \bbZ^d_N} U^N_{i,k} \hhh^N_k
 \]
 holds. The following properties of $\fff^N_m$ and $\hhh^N_k$ will be used 
 in the sequel without further specification:
 \begin{subequations}
     \label{eq:Props.fff.hhh}
 \begin{align}
  \label{eq:Props.fff.hhh.a}
     &\fff^N_m(x)\in [0,1], \quad \int_{\bbT^d}\fff^N_m (x) \dd x = \frac1{(2N)^d}, \quad \sum_{m\in M} \fff^N_m (x) = \one_{Q^N_0}(x), 
     \\
 \label{eq:Props.fff.hhh.b}
     & \hhh^N_m(x)\in [0,1], \quad \int_{\bbT^d}\hhh^N_k(x) \dd x=\frac1{N^d}, \quad 
     \sum_{k\in \bbZ^d_N} \hhh^N_k(x) = 1 \text{ on }\bbT^d,
\\
 \label{eq:Props.fff.hhh.c}
     &
     \|\nabla \hhh^N_m\|_{\rmL^\infty} \leq N, \quad
     \|\nabla \hhh^N_m\|_{\rmL^1} \leq C_d, \quad \hhh^N_k(x) \geq \frac1{2^d} 
     \shift^N_{\frac12\mathbf{1}_d}\one_{Q^N_k}(x),  
 \end{align}
 where we denoted $\mathbf{1}_d = (1,\ldots,1)\in \R^d$.
\end{subequations}
The next results shows that the concentrations $c^N$ enjoy a higher 
integrability as the one obtained from the uniform bound $K^E$ for $E_N$. 
For this we use the bound $K^S_\diff$ in \eqref{eq:AprioriEstimatesSlopeDiff} 
and a suitable Galiardo-Nirenberg interpolation applied to $\tilde u^N$. 
We first show that $\nabla \tilde u^N$ is uniformly bounded in $\rmL^2([0,T] 
\times \bbT^d)$, which is a consequence of the fact that $S_{N,\diff}(c^N)$ 
is in fact a quadratic form in $U^N$.  

\begin{lemma}[Spatial regularity]\label{lem:spatial_reg_NEW}
Let $c^N$ satisfy the a priori estimates \eqref{eq:AprioriEstimatesSlopeDiff}. 
Then, we have 
\[
\iint_\OmegaT  |\nabla \tilde u^N|^2 \dd x \dd t \leq \bra[\Big]{\frac43}^{d-1} \frac{K^S_\diff}{\delta_*\omega_*} 
\]
\end{lemma}
\begin{proof}
We work for fixed $N$ and $i$ and hence drop these indices throughout this proof. 

We recall that each $\hhh_k$ is nontrivial only on $2^d$ cubes $Q_{k-m}$. Moreover, 
fixing $l\in \{1,\ldots, d\}$ the derivative $\partial_{x_l} h_k$ has positive values in those $2^{d-1}$ cubes with $m_l=1$ and negative values in those with $m_l=0$:
\[
\partial_{x_l} h_k = \underbrace{\partial_{x_l} h_k \one_{V_{l,k}}}_{\geq 0}  -
\underbrace{\big(-\partial_{x_l}h_k \one_{V_{l,k+e_l}}\big)}_{\geq 0} \quad
\text{with } V_{l,k}= \bigcup\nolimits_{\genfrac{}{}{0pt}{2}{m\in M}{m_l=1}} Q^N_{k-m}.
\]
Using $\partial_{x_l} h_k =- \partial_{x_l} h_{k-e_l}$ on $V_{l,k}$ we find 
\[
\partial_{x_l} \tilde u = \sum_{k\in\bbZ^d_N} U_k \partial_{x_l} h_k = \sum_{k\in\bbZ^d_N} \big( U_k{-}U_{k-e_l}\big) \partial_{x_l} h_k \one_{V_{l,k}}. 
\]
At each $x \in \bbT^d$ there are at most $2^{d-1}$ terms, since each $V_{l,k}$ consists 
of $2^{d-1}$ small cubes. Hence, we obtain 
\begin{align*}
\int_{\bbT^d} \big| \partial_{x_l} \tilde u\big|^2 \dd x 
& \leq 2^{d-1}  \int_{\bbT^d} \sum_{k\in\bbZ^d_N} |U_k{-}U_{k-e_l}|^2 
     \big|\partial_{x_l}h_k\big|^2 \one_{V_{l,k}} \dd x 
= \bra[\Big]{\frac43}^{d-1} \frac1{N^d} \sum_{k\in \bbZ_N^d} N^2 
    \big| U_k{-} U_{k-e_l}|^2 ,
\end{align*}
where we used $\int_{\bbT^d} |\partial_{x_l}h_k|^2 \one_{V_{l,k}} \dd x = N^2(2/3)^{d-1}$.
This concludes the proof.
\end{proof}

To obtain uniform higher integrability of the densities $\rho^N$ we combine the 
spatial regularity with the uniform energy bound \eqref{eq:AprioriEstimBEnergy}. The former provides $\rmL^2$ 
integrability in time in the good space $\rmH^1(\bbT^d)$ while the latter provides
boundedness of $\calE(\tilde\rho^N)$ which is slightly better than $\esssup \tilde\rho^N 
(t) \leq K^E_{\rmL^1}$. We will exploit the following interpolation estimate that follows 
by applying a suitable Gagliardo-Nierenberg interpolation, see Appendix \ref{app:GagliaNirenb} for the proof of a more general version. 
Setting $\alpha\geq 2$,  $ \alpha d \in [4,4{+}2d]$ and $ q= 2d/(4{-}(\alpha{-}2)d)\in [1,\infty]$,  
it holds the bilinear interpolation estimate  
\begin{align}
\label{eq:GaglNireInterpol}
\iint_\OmegaT \! u^\alpha v \dd x \dd t &  
\leq C \| v\|_{\rmL^\infty(0,T;\rmL^q (\bbT^d))} 
      \| u\|^{\alpha -2}_{\rmL^\infty(0,T;\rmL^2(\bbT^d))} 
      \int_0^T\! \|u(t)\|^2 _{\rmH^1(\bbT^d)} \dd t,
\end{align}
for a suitable constant $C$ depending on $d$ and $\alpha$. To estimate 
$\tilde\rho^N_i = \omega_i (\tilde u^N)^2$ we will apply this estimate for 
$u=\tilde u^N_i$ and 
either $v\equiv 1$ or $v= \LB((\tilde u^N_i)^2)^\beta$. 

\begin{proposition}[Improved integrability]
\label{prop:integrability}
Let $c^N$ satisfy the a priori estimates \eqref{eq:AprioriEstimBEnergy}, \eqref{eq:AprioriEstimatesL1Bound}, and \eqref{eq:AprioriEstimatesSlopeDiff}. 
Then, with $\pcrit =1+2/d$ from Assumption
\eqref{ass:Reactions1} we have 

\begin{subequations}
\label{eq:HighInte.pcrit}
\begin{align}
    \label{eq:HighInte.pcrit.a}
   \sup_{N\in\bbN}\norm{c^N}_{\rmL^{\pcrit}([0,T]\times X_N)} \leq 
   C_{(1)},
\end{align}
where $C_{(1)}$ only depends on $d$, $\omega_*$, $\omega^*$ and $\delta_*$.
Moreover, with $\eta_d = 2/d$ for $d\geq 3$, $\eta_2\in (0,1)$, and $\eta_1=1$, 
there exists $C_{(2)}>0$ depending on $d,\,\omega_*,\omega^*,\, \delta_*,\, K^E,\, K^E_{\rmL^1}$ 
such that 
\begin{align}
    \label{eq:HighInte.pcrit.b}
  \sup_{N\in\bbN}\frac{1}{N^d}\sum_{k\in\bbZ^d_N}\sum_{i\in I}\int_0^T\bra{c_{i,k}^N}^{\pcrit} 
   \big( \log(1+c_{i,k}^N)\,\big)^{\eta_d} 
   \dd t \leq C_{(2)}.
\end{align}
In particular, analogous $N$-uniform estimates to \eqref{eq:HighInte.pcrit.a} and \eqref{eq:HighInte.pcrit.b} also hold for $\rho^N$.
\end{subequations} 
\end{proposition}
\begin{proof}
We consider only one species $i$ and drop its index throughout this proof.
 In light of \eqref{eq:Props.fff.hhh.c}, it is sufficient to prove the spatial regularity for $\tilde\rho^N$. 
From the definition of $\tilde\rho^N =\omega (\tilde u^N)^2$ we immediately obtain
$\| \tilde u^N(t)\|_{\rmL^2}^2 \leq K^E_{\rmL^1} /\omega_*$. Applying 
\eqref{eq:GaglNireInterpol} with $u=\tilde u^N_i$, $v\equiv 1$ and $\alpha = 2\pcrit$ (which implies $q=\infty$), we find 
\[
\iint_\OmegaT \big(\tilde u^N\big)^{2\pcrit} \dd x \dd t 
 \leq C \big\|\tilde u^N \big\|^{4/d}_{\rmL^\infty(0,T;\rmL^2)} 
  \int_0^T \big\| \tilde u^N \big\|_{\rmH^1(X)}^2 \dd t. 
\]
Using $\tilde\rho^N =\omega (\tilde u^N)^2$ and Lemma \ref{lem:spatial_reg_NEW}, we obtain \eqref{eq:HighInte.pcrit.a}. 

For the second part we choose $u=\tilde u^N_i$ and $v= \LB\big( ( \tilde u^N)^2\big)^{\eta_d}$.
For $d\geq 3$ let $\alpha=2$ and $q=d/2$ to find
\[
\iint_\OmegaT \! u^2 \LB(u^2)^{2/d} \dd x \dd t
\leq \int_0^T\|u\|_{\rmL^{2d/d-2}}^2 \|\LB(u^2)\|_{\rmL^1} \dd t  
\leq  C K_E^{2/d} K_\mathrm{diff}.
\]
For $d=1$ we choose $\alpha=4$ and $q=1$ giving 
\[
\iint_\OmegaT \! u^4 \LB(u^2) \dd x \dd t 
\leq \int_0^T\! \|u\|_{\rmL^\infty}^4 \|\LB(u^2)\|_{\rmL^1}  \dd t  \leq 
C  K_E^2 K^S_\mathrm{diff}.
\]
For $d=2$ we choose $\eta_2\in (0,1) $ arbitrary and set $q=1/\eta_2$ and
$\alpha =4{-}2\eta_2$. This leads to the estimate  
\begin{align*}
&\iint_\OmegaT \! u^{4-2\eta_2} \LB(u^2)^{\eta_2} \dd x \dd t
 \leq  \int_0^T\! \|u\|_{\rmL^{(4-2\eta_2)/(1-\eta_2)}}^{4-2\eta_2} 
 \|\LB(u^2)^\eta_2\|_{\rmL^{1/\eta_2}}^{1/\eta_2}  \dd t \\
 &= \int_0^T\! \|u\|_{\rmL^{(4-2\eta_2)/(1-\eta_2)}}^{4-2\eta_2} 
 \|\LB(u^2)\|_{\rmL^1}  \dd t \leq  \rmC(\eta_2)  K_E^2 K_\mathrm{diff},
\end{align*}
where $C(\eta_2)\to \infty$ for $\eta_2 \nearrow 1$. 

Using $u^q \log(e{+}u)^\eta \leq C u^{q-2\eta} \big( 1+ \LB(u^2)\big)^\eta $ and $\tilde\rho^N \leq \omega^* (\tilde u^N)^2$, the estimate \eqref{eq:HighInte.pcrit.b} follows.
\end{proof}

The higher integrability derived in \eqref{eq:HighInte.pcrit.b} will allow 
us to show that the diffusion fluxes $ f^N=\iota_{N,\diff} F^N$ and the 
reaction fluxes $ j^N=  \iota_{N,\react} J^N$ are uniformly equi-integrable, 
and hence one may choose a subsequence converging weakly in $\rmL^1(\OmegaT)$. 
The estimate for $f^N$ will rely on the magical estimate \eqref{eq:CC.prop.d}, whereas the estimate for $j^N$ has to be based on the weaker result of Lemma \ref{lem:Superlinear}.

\begin{proposition}[Boundedness of fluxes]
\label{prop:FluxBound} 
Assume  \eqref{ass:Reactions1} and let $(c^N,F^N,J^N)$ satisfy the a priori estimates \eqref{eq:AprioriEstimBasic}. 
Then, there exist constants $C^\diff_{\mathrm{flux}}>0$ and $C^\react_\mathrm{flux}>0$ 
and a convex superlinear function $\Phi_d:\R\to [0,\infty)$
depending only on $d$,
such that for all $N\in\bbN$, $e\in\Eu$, $i\in I $, and $ r \in R $ s we have
\begin{align*}
\frac{1}{N^d}\sum_{k\in\bbZ^d_N}\int_0^T \C\bra[\bigg]{\frac{F^N_{i,k,e}}{N}} \dd t \leq C^\diff_{\mathrm{flux}}
\quad \text{and} \quad 
\frac{1}{N^d}\sum_{k\in\bbZ^d_N}\int_0^T \Phi_d ( J^N_{r,k}) \dd t \leq C^\react_{\mathrm{flux}},
\end{align*}
where $C^\diff_{\mathrm{flux}}$ ($ C^\react_{\mathrm{flux}}$)  depends only on the constants $C_{(1)}$ and $C_{(2)}$ from \eqref{eq:HighInte.pcrit} and $K^R_\diff$ ($K^R_\react$). 

Moreover, there exist curves of fluxes $f$ and $j$ with $ f_{i,e}\in \rmL^1(\OmegaT)$ and $j_r\in \rmL^1(\OmegaT)$ such that along a (not renamed) subsequence, we have
\begin{align*}
   &\iint_\OmegaT \C(f_{i,e}) \dd x \dd t \leq C^\diff_{\mathrm{flux}}
    \quad \text{and} \quad 
    \iint_\OmegaT \Phi_d( j_r) \dd x \dd t \leq C^\react_{\mathrm{flux}},
\\
  & f^N_{i,e} \rightharpoonup f_{i,e} \quad \text{and} \quad 
     j^N_r \rightharpoonup j_r  \quad\text{weakly in } \rmL^1(\OmegaT),
\end{align*}
where we recall $f^N=\iota_{N,\diff} F^N$ 
and $j^N= \iota_{N,\react} J^N$ defined in \eqref{eq:def.iotaFlux.diff} and \eqref{eq:def.iotaFlux.react}, respectively.
\end{proposition}
\begin{proof} 
We consider the diffusive flux $F^N_{i,e}$ first, where we fix and then omit the indices 
$i,e$. 
We apply the magical estimate
\eqref{eq:CC.prop.d} with $q=\pcrit>1 $ to obtain
\begin{align*}
  \frac{1}{N^d}\sum_{k\in\bbZ^d_N}\int_0^T \C\Big(\frac1N F^N_k\Big) \dd t
&\leq
  \frac{1}{N^d}\sum_{k\in\bbZ^d_N}\int_0^T \Big( \frac{\pcrit}{ \pcrit{-}1}  \CC\Big(\frac1N F^N_k \Big| \sigma^N_k \Big) 
    + \frac{4}{\pcrit {-}1 } (\sigma^N_k)^2\Big) \dd t
\\
& \leq  C_{\pcrit} \frac{1}{N^d}\sum_{k\in\bbZ^d_N}\int_0^T \Big( \CC\big( F^N_k 
         \big| N^2\sigma^N_k \big) 
        + (\sigma^N_k)^{\pcrit}\Big) \dd t \\
&\leq   C_{\pcrit} K^R_\diff 
        +  C_{\pcrit}  \big\| \sigma^N\big\|^{\pcrit}_{\rmL^{\pcrit}([0,T]\times\bbZ^d_N)},
\end{align*} 
where the estimate from the second to the third line follows from the monotonicity 
\eqref{eq:CC.prop.f} and where we used that 
 $\sigma^N_{i,e} = \delta_i 
(c_{i,k}c_{i,k+e})^{1/2}$ 
is uniformly bounded in $\rmL^{\pcrit}$ by 
$C_{(1)}$ in \eqref{eq:HighInte.pcrit.a}. 

The argument for $j^N_r$ is analogous, however, we have to be aware that we now have to choose 
 $\sigma^N_r = \kappa_r (c^N)^{\gamma^r}$ 
with $\gamma^r = \frac12(\alpha^r{+}\beta^r)$. Thus,  \eqref{eq:HighInte.pcrit.a} and assumption \eqref{ass:Reactions1} only provide a uniform bound for $\sigma^N_r$ in $\rmL^1$. 
However, Lemma \ref{lem:Superlinear} can be employed on the basis of the 
improved higher regularity. We choose $\phi=\C$ and
\[
\psi_d( w)= w \,\big( \log(1{+} w^{1/\pcrit}) \big)^{1/d},
\]
which is increasing and superlinear. Thus, the function 
$\Xi_d= \Xi_{\C, \psi_d}$ is still superlinear and increasing, and the same is true
for its convex hull $\Phi_d= (\Xi_d)^{**} \leq \Xi_d$. 
 With this, fixing and omitting the index $r$, we can estimate 
\begin{align*}
    \frac{1}{N^d}\sum_{k\in\bbZ^d_N}\int_0^T \Phi_d(J^N) \dd t &\leq 
    \frac{1}{N^d}\sum_{k\in\bbZ^d_N}\int_0^T \Xi_d(J^N_k) \dd t \leq 
     \frac{1}{N^d}\sum_{k\in\bbZ^d_N}\int_0^T \big( \CC(J^N_k| \sigma^N_k) + \psi_d(\sigma^N_k) \Big) \dd t
\\
&\leq K^\react_\mathrm{flux} + \frac{1}{N^d}\sum_{k\in\bbZ^d_N}\int_0^T 
 |c^N|^{\pcrit} \big(\log(1{+}|c^N_k|\big)^{1/d} \dd t 
 \leq K^\react_\mathrm{flux} + C_{(2)}
 \end{align*} 
with $C_{(2)}$ from \eqref{eq:HighInte.pcrit.b}.

 For the embedded diffusive fluxes  $f^N=\iota_{N,\diff} F^n$, we recall that \eqref{eq:def.iotaFlux.diff.1D} involves a 
partition of unity. Therefore, it follows 
$\iint_\OmegaT \C(f^N) \dd x \dd t 
\leq \frac{1}{N^d}\sum_{k\in\bbZ^d_N}\int_0^T \C(\frac1N F^N_k) \dd t$ by an application of Jensen's inequality. Similarly, we have for the embedded reactive fluxes the estimate $\iint_\OmegaT\Phi_d(j^N) \dd x \dd t\le \frac{1}{N^d}\sum_{k\in\bbZ^d_N}\int_0^T \Phi_d(J^N_k) \dd t$.

With this, 
the criterion of de la Vall\'e Poussin shows 
that the sequences $(f^N)_N$ and $(j^N)_N$ both are sequentially compact in the weak 
topology of $\rmL^1(\OmegaT)$. Thus, a subsequence (not relabeled) and limits 
$f$ and $j$ exist such that $f^N\rightharpoonup f$ and $j^N \rightharpoonup j$.
Moreover, the convexity of $ \C$ and $\Phi_d$ implies the weak lower semi-continuities
$\int_\OmegaT \C(f_{i,e}) \dd x \dd t \leq \liminf_{N\to \infty} 
\int_\OmegaT \C(f^N_{i,e}) \dd x \dd t \leq C^\diff_\mathrm{flux}$
 and $\int_\OmegaT \Phi_d(j_r) \dd x \dd t \leq \liminf_{N\to \infty} 
\int_\OmegaT \Phi_d( j_r) \dd x \dd t \leq C^\react_\mathrm{flux}$.

With this, the proof of Proposition~\ref{prop:FluxBound} is complete.
\end{proof}

\begin{remark}
The uniform equi-integrability for the diffusive rate $\iint \C(F^n_{i,e}) \dd x \dd t \leq C^\diff_\mathrm{flux}$ was also obtained in  \cite[Lem.\,4.4]{HT23} by a slightly different and more generally applicable argument that only uses that $\sigma^N$ is uniformly bounded in $\rmL^\infty(0,T;\rmL^1(\bbT^d))$.   
\end{remark}

Having established the spatial regularity of $\tilde \rho^N$ in Lemma~\ref{lem:spatial_reg_NEW} as well as boundedness of the fluxes in Proposition~\ref{prop:FluxBound}, our next step is to show time regularity for $\tilde\rho^N$. 

\begin{lemma}[Time regularity]\label{lem:time_regularity}

Assume \eqref{ass:Reactions1} and let $(c^N,F^N,J^N)$ satisfy the a priori estimates \eqref{eq:AprioriEstimBasic}.
Then, we have 
the uniform bound $\sup_{N\in\bbN}\norm{\tilde\rho^N}_{BV(0,T;(W^{1,\infty}(X))^\ast)}<\infty$.
\end{lemma}
\begin{proof}
    We first focus on the more complicated interpolation $\tilde\iota_N$ and recall \eqref{eq:HSDI}, which implies for every $\varphi\in \rmC^1(X)$
    \begin{align*}
        \langle F^N,\overline\nabla\tilde\iota_N^\ast \varphi \rangle 
        &= \frac{1}{N^d}\sum_{k\in\bbZ_N^d}\sum_{e\in \Eu}\sum_{i\in I} F^N_{i,k,e} [(\tilde\iota_N^\ast \varphi)_{i,k+e}-(\tilde\iota_N^\ast \varphi)_{i,k}]\\
        &= \sum_{k\in\bbZ_N^d}\sum_{e\in \Eu}\sum_{i\in I} \sum_{m\in M} F^N_{i,k,e} \int_{Q^N_k} \shift^N_{-k} \fff^N_m \shift^N_{-m}[\shift^N_e\varphi_i-\varphi_i]\dd x\\
        &= \sum_{k\in\bbZ_N^d}\sum_{e\in \Eu}\sum_{i\in I} \sum_{m\in M} \frac{F^N_{i,k,e}}{N} \int_{Q^N_k}\int_0^1 \shift^N_{-k} \fff^N_m(x) \shift^N_{e\theta-m}\partial_{x_e}\varphi_i(x) \dd\theta\dd x\\
        &\le C_M \norm{\nabla\varphi}_{\rmL^\infty(X)} \sum_{i\in I}\sum_{e\in \Eu}\norm[\bigg]{\frac{F_{i,e}^N}{N}}_1.
    \end{align*}
    Similarly, for every $\varphi\in \rmC(X)$ we have
    \begin{align*}
        \langle J^N,\Gamma\tilde\iota_N^\ast \varphi \rangle 
        &= \sum_{k\in\bbZ_N^d} \sum_{r\in R} J^N_{r,k} \sum_{i\in I} \gamma_{r,i}
        \sum_{m\in M} \int_{Q^N_k}\shift^N_{-k} \fff^N_m \shift^N_{-m}\varphi_i \dd x \\
        &\le \norm{\varphi}_{\rmL^\infty(X)}\max_{s\in R}\sum_{i\in I} 
        \abs{\gamma_{s,i}}\sum_{r\in R}\norm{J^N_r}_1.
    \end{align*}
    With this, we consider any partition $(t_m)_{m=0}^M$, $M\in\bbN$ of $[0,T]$. 
    Then, the previous bounds, the discrete continuity equation, and 
    Proposition~\ref{prop:FluxBound} yield, for every $\varphi\in \rmC^1(X)$,
    the estimate 
    \begin{align*}
         \sum_{m=1}^M \langle \tilde\rho^N(t_m)
           -\tilde\rho^N(t_{m-1}), \varphi\rangle
        &= \sum_{m=1}^M \langle c^N(t_m)-c^N(t_{m-1}),
              \tilde\iota_N^\ast\varphi\rangle_N\\
        &= \sum_{m=1}^M\int_{t_{m-1}}^{t_m}
        \langle F^N,\overline\nabla\tilde\iota_N^\ast \varphi \rangle 
        + \langle J^N,\Gamma\tilde\iota_N^\ast \varphi \rangle \dd t
        \le C\norm{\varphi}_{\rmC^1(X)}.
    \end{align*}
    Taking suprema with respect to $\varphi\in\set{\varphi\in \rm\rmC^1(X)$: 
    $\norm{\varphi}_{\mathrm{W}^{1,\infty}(X)}\le 1}$ and the partition, we obtain 
    $\sup_N\norm{\tilde\rho^N}_{\mathrm{BV}(0,T;(\mathrm{W}^{1,\infty})^\ast)}
    <\infty$. 
    This finishes the proof of Lemma~\ref{lem:time_regularity}.
\end{proof}

Combining the spatial regularity of $\tilde u$ from Lemma~\ref{lem:spatial_reg_NEW} and the time regularity of $\tilde\rho$ from Lemma~\ref{lem:time_regularity}, we are now able to apply the Aubin-Lions-Simon lemma to obtain strong $\rmL^1$ compactness for $\tilde\rho^N$. We then show that $\rho^N$ has the same strong limit by comparing it to $\tilde\rho^N$.

\begin{proposition}[Strong compactness]\label{prop:strong_compactness}
    Assume \eqref{ass:Reactions1} and let $(c^N,F^N,J^N)$ satisfy the a priori estimates \eqref{eq:AprioriEstimBasic}.
    
    Then, there exists $\rho\in  \rmL^{\pcrit }([0,T]\times X)$ such that along a (not renamed) subsequence
    both $\tilde\rho^N \to \rho$ strongly in $\rmL^1([0,T]\times X)$ and $\rho^N \to \rho$ strongly in $\rmL^1([0,T]\times X)$.

    Furthermore, it holds $\rho\in\AC([0,T];\calM_+(X)$ and $\rho^N(t)\rightharpoonup^\ast\rho(t)$ weakly-$^\ast$ in $\calM_+(X)$ for all $t\in[0,T]$.
\end{proposition}
\begin{proof}

Lemma~\ref{lem:spatial_reg_NEW} implies that $\tilde\rho^N \in \rmL^1([0,T];Z)$, where $Z\coloneqq \{\omega u^2: u\in \rmH^1(X)\}$. Since $\omega \in \rmL^\infty(X)$, we have the compact embedding $Z \Subset \rmL^1(X)$. Combining this with Lemma~\ref{lem:time_regularity}, we obtain the 
existence of $\rho\in \rmL^1([0,T]\times X)$ such that $\tilde \rho^N\to \rho$ strongly in $\rmL^1([0,T]\times X)$ 
by applying the Aubin-Lions-Simon Lemma, \cite[Theorem 5]{simon1986compact}.

For the convergence of $\rho^N$, we compare it to $\tilde\rho^N$ and recall that $U^N_k = \sqrt{\frac{c^N_k}{w^N_k}}$ to derive 
\begin{align}
    \notag
    \norm{\tilde\rho^N-\rho^N}_{L^1([0,T]\times X)} 
    &= \norm[\bigg]{\omega\bra[\bigg]{\tilde\iota_N\sqrt{\frac{c^N}{w^N}}}^2-\iota_N c^N}_{L^1([0,T]\times X)}\\
    \notag
    &\le \omega^\ast\norm[\bigg]{\bra[\bigg]{\tilde\iota_N\sqrt{\frac{c^N}{w^N}}}^2-\iota_N \frac{c^N}{w^N}}_{L^1([0,T]\times X)}
    + \norm[\bigg]{\omega\iota_N \frac{c^N}{w^N} - \iota_N c^N}_{L^1([0,T]\times X)}\\
    \notag
    &\le \omega^\ast\norm[\big]{(\tilde\iota_N U^N_k)^2-\iota_N(U^N_k)^2}_{L^1([0,T]\times X)}
    + \frac{K^E_{\rmL^1}}{\omega_\ast}\norm[\big]{\omega - \iota_N w^N}_{L^\infty([0,T]\times X)}.
\end{align}
The second summand on the right-hand side vanishes as $N\to\infty$ by \eqref{eq:ConvergenceRefMeas}.
To control the first summand, we first employ \eqref{eq:Props.fff.hhh} to obtain for all $x\in\bbT^d$ the auxiliary inequality
\begin{align*}
    \bra[\Bigg]{\sum_{k\in\bbZ^d_N} (\hhh^N_k \pm \one_{Q^N_k})U^N_k}^2
    &= \sum_{k,l\in\bbZ^d_N}\sum_{m,n\in M}  (\shift^N_{m-k}\fff^N_m(x) \pm \shift^N_{-k}\fff^N_m(x))(\shift^N_{n-l}\fff^N_n(x) \pm \shift^N_{-l}\fff^N_n(x)) U^N_k U^N_l\\
    &= \sum_{k,l\in\bbZ^d_N}\sum_{m,n\in M} \shift^N_{-k}\fff^N_m(x) \shift^N_{-l}\fff^N_n(x) ( U^N_{k+m} \pm U^N_k) ( U^N_{l+n} \pm U^N_l)\\
    &= \sum_{k\in\bbZ^d_N}\sum_{m,n\in M} \shift^N_{-k}\fff^N_m(x) \shift^N_{-k}\fff^N_n(x) ( U^N_{k+m} \pm U^N_k) ( U^N_{k+n} \pm U^N_k)\\
    &\le 2^d \sum_{k\in\bbZ^d_N}\sum_{m,n\in M} \one_{Q^N_k}(x)  ( U^N_{k+m} \pm U^N_k) ( U^N_{k+n} \pm U^N_k)\\
    &\le 2^d \sum_{k\in\bbZ^d_N}\sum_{m\in M}\one_{Q^N_k}(x)\abs{U^N_{k+m} \pm U^N_k}^2.
\end{align*}
We combine this with Hölder's inequality to find
\begingroup
\allowdisplaybreaks
\begin{align*}
    \norm[\big]{(\tilde\iota_N U^N_k)^2-\iota_N(U^N_k)^2}&_{L^1([0,T]\times X)} 
    = \norm[\Bigg]{\bra[\Bigg]{\sum_{k\in\bbZ^d_N}\hhh^N_k U^N_k}^2-\sum_{k\in\bbZ^d_N}\one_{Q^N_k}(U^N_k)^2}_{L^1([0,T]\times X)}\\
    {}&\phantom{_{L^1([0,T]\times X)}}= \norm[\Bigg]{\bra[\Bigg]{\sum_{k\in\bbZ^d_N}\hhh^N_k U^N_k}^2-\bra[\Bigg]{\sum_{k\in\bbZ^d_N}\one_{Q^N_k}U^N_k}^2}_{L^1([0,T]\times X)}\\
    &\le \norm[\Bigg]{\sum_{k\in\bbZ^d_N} (\hhh^N_k + \one_{Q^N_k})U^N_k}_{L^2([0,T]\times X)}^{\frac{1}{2}}\norm[\Bigg]{\sum_{k\in\bbZ^d_N} (\hhh^N_k - \one_{Q^N_k})U^N_k}_{L^2([0,T]\times X)}^{\frac{1}{2}}\\
    &\le 2^d \norm[\Bigg]{\sum_{m\in M} (U^N_{k+m}+U^N_k)}_{L^2([0,T]\times X_N)}^{\frac{1}{2}} \norm[\Bigg]{\sum_{m\in M} (U^N_{k+m}-U^N_k)}_{L^2([0,T]\times X_N)}^{\frac{1}{2}}\\
    &\le C_M\sqrt{\frac{K^E_{\rmL^1} K^s_\diff}{\omega_\ast}\frac{K^S_\diff}{\delta_\ast \omega_\ast}}\frac{1}{N},
\end{align*}
\endgroup
where in the last step we used that each $m\in M$ is a sum of finitely many $d$-dimensional unit vectors, estimated the first factor using \eqref{eq:AprioriEstimBEnergy}, and extracted the power $1/N$ from the second factor by estimating with the uniform bound \eqref{eq:AprioriEstimatesSlopeDiff} and Definition~\ref{def:slope_disc}.

Moreover, by Proposition~\ref{prop:integrability} the curves $\rho^N$ are 
$N$-uniformly bounded in $\rmL^{\pcrit }([0,T]\times X)$, so is the limit $\rho$. This concludes the proof of Proposition~\ref{prop:strong_compactness}.

The weak-$^\ast$ convergence $\rho^N(t)\rightharpoonup^\ast\rho(t)$ for all $t\in[0,T]$ follows from the bounds in Proposition~\ref{prop:FluxBound} by arguing analogously to \cite[Lemma 4.5]{HT23}.

\end{proof}

To later obtain a lower limit inequality for the dissipation functionals, we must ensure that the limit objects from Proposition~\ref{prop:FluxBound} and Proposition~\ref{prop:strong_compactness} satisfy the continuity equation. This fact is established in the following lemma:
\begin{lemma}[Closedness of $\CE$]\label{lem:CE_closed}
    Let $(\rho,j,f)$ be a limit of $(\rho^N,j^N,f^N)_{N\in\bbN}\subset\CE$ in the sense of Propositions~\ref{prop:strong_compactness} and \ref{prop:FluxBound}. Then, it holds $(\rho,j,f)\in\CE$. 
\end{lemma}
\begin{proof}

     By definition the set $\CE$ is closed with respect to the weak-$\rmL^1$ convergence of the 
    time-integrated embedded fluxes shown in Proposition~\ref{prop:FluxBound} and the pointwise-in-time weak-$^\ast$ convergence of the embedded concentrations shown in Proposition~\ref{prop:strong_compactness}.
\end{proof}

To prove the lower limit inequality for the slopes, we will employ a convergence result for the differences of the piecewise constantly embedded concentrations. This is established next.
\begin{proposition}[Convergence of differences]\label{prop:convergence_of_derivatives}
    Assume \eqref{ass:Reactions1} and let $(c^N,F^N,J^N)$ satisfy the a priori estimates \eqref{eq:AprioriEstimBasic}.      
    Let $\rho$ be the limit of $(\iota_N c^N)_N$ from Proposition \ref{prop:strong_compactness}.
    Recalling $U^N_{i,k}= (c_{i,k}/w^N_{i,k})^{1/2}$, we introduce $u^N_i = \iota_N U^N_i$, $u\coloneqq ((\rho_i/\omega_i)^{1/2})_i$, and $\nabla_N u^N \coloneqq \frac1N\sum_{e\in\Eu}\bra{\shift^N_eu^N-u^N}e$.
    
    Then, it holds $u\in \rmL^2(0,T; \rmH^1(\bbT^d))$,
    $u^N\to u$ strongly in $\rmL^2([0,T]\times X)$ 
    and along a (not renamed) subsequence $\nabla_N u^N\rightharpoonup \nabla u$ weakly in $\rmL^2([0,T]\times X)$.
\end{proposition}
\begin{proof}

    Throughout this proof we fix an arbitrary species $i$ and omit the corresponding index. We denote $\omega^N = \iota_N w^N$. The strong $\rmL^2$ convergence $u^N\to u$ immediately follows by integrating the estimate
    \begin{align*}
        \abs{u^N-u}^2 = \abs[\bigg]{\sqrt{\frac{\rho^N}{\omega^N}}-\sqrt{\frac{\rho}{\omega}}}^2 \le \abs[\Big]{\frac{\rho^N}{\omega^N}-\frac{\rho}{\omega}} \le \frac{1}{\omega_\ast}\abs{\rho^N-\rho} +\frac{1}{\omega_\ast^2}\abs{\rho}\abs{\omega^N-\omega},
    \end{align*}
    and using Assumption~\ref{ass:General}.
    
    Next, we consider the differences $\nabla_N u^N = \frac1N\sum_{e\in\Eu}\bra{\shift^N_eu^N-u^N}e$.
    Since $\iota_N$ commutes with multiplication, it holds $\norm{\nabla_N u^N}_{\rmL^2(\OmegaT;\bbR^d)}^2 \le \delta\omega^\ast K^S_\diff$ and
    hence (along a not renamed subsequence) $\nabla_N u^N \rightharpoonup v$ weakly in $\rmL^2(\OmegaT;\R^d)$ for some $v \in \rmL^2(\OmegaT;\R^d)$. 
    This $v$ is the weak gradient of $u$. Indeed, let $\varphi\in \rmC^\infty(\OmegaT)$. Then, for every $e\in \Eu$ (and the above subsequence) we have
	\begin{align*}
		\iint_\OmegaT v(t,x)\cdot e\; \varphi(t,x)\dd x \dd t 
		&= \lim_{N\to\infty} \iint_\OmegaT\bra*{\frac{u^N(t,x+e/N)-u^N(t,x)}{1/N}}\varphi(t,x)\dd x \dd t\\
        &= \lim_{N\to\infty}\iint_\OmegaT u^N(t,x)\bra*{\frac{\varphi(t,x-e/N)-\varphi(t,x)}{1/N}}\dd x \dd t\\
        &=-\iint_\OmegaT u(t,x)\partial_{x_e}\varphi(t,x)\dd x \dd t.
    \end{align*}
    In particular, $u\in \rmL^2([0,T];\rmH^1(\bbT^d))$ and the proof 
    is concluded.
\end{proof}

\subsection{Lower limit of dissipation functionals}
\label{sub:LowLimDiss}

In the previous section we have obtained candidate curves that may be EDB solutions for \eqref{eq:RDSysCont}. Following the strategy of Section~\ref{ss:GeneralConvergenceStrategy}, the next step is to prove rigorous analogs of \eqref{eq:abstract_lower_limit}. More precisely, we will prove lower limit inequalities for the rate and slope terms independently. First, we consider the slopes, employing a Ioffe's liminf theorem.
\begin{proposition}\label{prop:LiminfSlope}
    Assume \eqref{ass:Reactions1} and let $(c^N,F^N,J^N)$ satisfy the a priori estimates \eqref{eq:AprioriEstimBasic}. 
    Let $u_i = \sqrt{\rho_i/\omega_i}$ be the $\rmL^2$-limit of $u_i^N = \iota_N\sqrt{c^N_i\discref^N_i}$ from Proposition~\ref{prop:convergence_of_derivatives}.
    Then, it holds
    \begin{align*}
        \liminf_{N\to\infty}\int_0^T S_N(c^N)\dd t \ge \int_0^T\calS(\rho)\dd t
    \end{align*}
\end{proposition}

\begin{proof}
    By Proposition~\ref{prop:convergence_of_derivatives} we have along a (not renamed) subsequence $\nabla_N u^N_i\rightharpoonup \nabla u_i$ weakly in $\rmL^2(\OmegaT;\bbR^d)$. 
    Thus, an application of Ioffe's liminf theorem, \cite[Thm.\,2.3.1]{Buttazzo1989}, directly yields the lower limit for the diffusive part: 
    \begin{align*}
        \liminf_{N\to\infty}\int_0^T S_{N,\diff}(c^N)\dd t
        &\ge \sum_{i\in I} 2\delta_i \int_0^T\int_{\bbT^d} |\nabla u_i|^{2}\dd\omega_i\dd t = \int_0^T\calS_\diff(\rho)\dd t.
    \end{align*}
    For the reactive part, by definition of $\iota_N$, we have
    \begin{align*}
        S_{N,\react}(c^N) 
        &= \sum_{r\in R} 2\kappa_r \frac{1}{N^d} \sum_{k\in\bbZ_N^d} \bra[\big]{\discref_k^N}^{(\alpha^r+\beta^r)/2}\abs[\Bigg]{\bra[\bigg]{\frac{c_k^N}{\discref_k^N}}^{\alpha^r/2} - \bra[\bigg]{\frac{c_k^N}{\discref_k^N}}^{\beta^r/2}}^2\\
        &= \sum_{r\in R} 2\kappa_r \int_{\bbT^d} \bra[\big]{\omega^N}^{(\alpha^r+\beta^r)/2}\abs[\Big]{\bra[\big]{u^N}^{\alpha^r} - \bra[\big]{u^N}^{\beta^r}}^2.
    \end{align*}
On the other hand, it holds $(u^N)^\lambda(t,x)\to u^\lambda(t,x)$ and $\discref^N(x) \to \omega(x)$ for 
$\scrL$-a.e. $t\in[0,T]$ and $\scrL^d$-a.e. $x\in\bbT^d$, and every multiindex 
$\lambda\in\R^I$.
Thus, Fatou's Lemma with $f_N = (\omega^N){(\alpha^r+\beta^r)/2} 
\abs[\big]{\bra[\big]{\nu^N}^{\alpha^r} - \bra[\big]{\nu^N}^{\beta^r}}^2$ yields
\begin{align*}
        \liminf_{N\to\infty}\int_0^T S_{N,\react}(c^N)\dd t \ge \sum_{r\in R} 2\kappa_r 
        \int_0^T \int_{\bbT^d} \omega^{(\alpha^r+\beta^r)/2}\abs[\big]{u^{\alpha^r} - 
        u^{\beta^r}}^2\dd t = \int_0^T \calS_\react(\rho)\dd t,
\end{align*}
which concludes the proof.
\end{proof}

Next, we focus on the rate parts of the dissipation potentials. Here, the main challenge
is the diffusive rate, where we want to obtain the quadratic dissipation from the cosh-type dissipation. 
The proof is done by dualization following the proof of \cite[Thm.\,6.2 (i)]{HT23}. First, 
we link the cosh-type and quadratic dual dissipation potentials in the following lemma.

\begin{lemma}\label{lem:EstimateDualDissipationPotential}
Let $(c^N)_{N\in\N}$ be any sequence s.t. $\iota_Nc^N\eqqcolon\rho^N\rightharpoonup \rho$ in $\rmL^1(X)$. Moreover, let $\varphi\in \rmC^1(X)$ be given. Then, it holds
    \[
    \limsup_{N\to\infty} R^*_{N,\diff}(c^N,\overline\nabla\iota^*_N\varphi)\leq \cal{R}^*_\diff(\rho,\nabla\varphi).
    \]
    In particular, we have for $\iota_Nc^N\to\rho$ in $\rmL^1([0,T]\times X)$ and any $\varphi\in \rmL^1(0,T;\rmC^1(X))$ that
    \[
    \limsup_{N\to\infty} \int_0^T R^*_{N,\diff}(c^N(t),\overline\nabla\iota^*_N\varphi(t))\dd t\leq \int_0^T \mathcal{R}^*_\diff(\rho(t),\nabla\varphi(t))\dd t.
    \]
\end{lemma}
\begin{proof}
    For $\varphi\in \rmC^1(X)$, we have
    \begin{align*}
    |\iota^*_N\varphi_{i,k+e} - \iota^*_N\varphi_{i,k}| &\leq N^d\int_{\bbT^d} |\varphi_i(x)\bra[\big]{\one_{Q^N_{k+e}}(x)-\one_{Q^N_k}(x)}|\dd x\\
    &= N^d\int_{Q^N_k} |\varphi_i(x) -\varphi_i(x+e/N)|\dd x\\
    &\leq N^d\int_{Q^N_k} |\partial_{x_e}\varphi_i(x)|\cdot |e/N|\dd x = \frac1N\iota^*_N(|\partial_{x_e}\varphi_i|)_k .
\end{align*}
Using that $\C^*(r)=\C^*(-r)=\C^*(|r|)$ and the monotonicity of $[0,\infty)\ni r\mapsto \C^*(r)$, we compute
\begin{align*}
    R^*_{N, \diff} (c^N,\overline\nabla\iota^*_N\varphi) &= \frac{1}{N^d} \sum_{i\in I}\sum_{k\in\bbZ_N^d} \sum_{e\in \Eu} N^2 \delta_i \bra[\big]{c^N_{i,k}c^N_{i,k+e}}^{1/2}    \C^*( \iota^*_N\varphi_{i,k+e} - \iota^*_N\varphi_{i,k})\\
    &= \frac{1}{N^d} \sum_{i\in I}\sum_{k\in\bbZ_N^d} \sum_{e\in \Eu} N^2 \delta_i \bra[\big]{c^N_{i,k}c^N_{i,k+e}}^{1/2}    \C^*(|\iota^*_N\varphi_{i,k+e} - \iota^*_N\varphi_{i,k}|)\\
    &\leq \frac{1}{N^d} \sum_{i\in I}\sum_{k\in\bbZ_N^d} \sum_{e\in \Eu} N^2 \delta_i \bra[\big]{c^N_{i,k}c^N_{i,k+e}}^{1/2}    \C^*\bra[\Big]{\frac1N\iota^*_N(|\partial_{x_e}\varphi_i|)_k}.
\end{align*}
Note that, by the definition of $\C^\ast$, for all $r\in[0,\infty)$ and all $N\in\bbN$ it holds 
\[
    N^2\C^*\bra[\Big]{\frac{r}{N}}\leq \frac{r^2}{2}\cosh\bra[\Big]{\frac{r}{N}}.
\]
Our aim is to apply this with $r=\iota^*_N(|\partial_{x_e}\varphi_i|)_k$, which is why we introduce the scalar
\[
    a_N\coloneqq\max_{k\in\bbZ_N^d,i\in I,e\in\mathrm{E}^d}\Big\{\cosh\Big(\frac1N\iota^*_N(|\partial_{x_e}\varphi_i|)_k\Big)\Big\}\in[1,\infty),
\]
and observe that $a_N\to 1$ as $N\to\infty$. Then, we can conclude that
\begin{align*}
    R^*_{N, \diff} (c^N,\overline\nabla\iota^*_N\varphi) &\leq a_N \frac{1}{N^d} \sum_{i\in I}\sum_{k\in\bbZ_N^d} \sum_{e\in \Eu}  \delta_i \bra[\big]{c^N_{i,k}c^N_{i,k+e}}^{1/2}\frac 1 2 (\iota^*_N(|\partial_{x_e}\varphi_i)_k|))^2\\
    &\le a_N\frac{1}{N^d} \sum_{i\in I}\sum_{k\in\bbZ_N^d} \sum_{e\in \Eu} \delta_i \frac{c^N_{i,k}+c^N_{i,k+e}}{2}\frac 1 2 (\iota^*_N(|\partial_{x_e}\varphi_i|)_k))^2 \\
    &\le a_N \frac{1}{N^d} \sum_{i\in I}\sum_{k\in\bbZ_N^d} \sum_{e\in \Eu} \delta_i \frac{c^N_{i,k}+c^N_{i,k+e}}{2}\frac 1 2 \iota^*_N(|\partial_{x_e}\varphi_i|)_k^2) \\
    &= a_N \sum_{i\in I}\sum_{e\in \Eu}\delta_i  \big\langle \big(c^N_i+\shift^N_e c^N_i\big)/2,\frac12\iota^*_N\big(\abs{\partial_{x_e}\varphi_i}^2\big)\big\rangle_N \\
    &= a_N\sum_{i\in I}\delta_i  \Big\langle \iota_N \big(c^N_i+\shift^N_e c^N_i\big)/2,\sum_{e\in \Eu}\frac12\abs{\partial_{x_e}\varphi_i}^2\Big\rangle \\
    &= a_N\sum_{i\in I}\delta_i  \big\langle \iota_N \big(c^N_i+\shift^N_e c^N_i\big)/2,\abs{\nabla\varphi_i}^2/2\big\rangle \\
    &\overset{N\to \infty}{\longrightarrow} \cR_{\mathrm{diff}}^*(\rho,\nabla\varphi),
\end{align*}
where we have used that for all $\psi \in \rmC(\bbT^d)$ it holds
\begin{align*}
    \langle\iota_N\shift^N_ec_i^N,\psi\rangle = \sum_{k\in\bbZ_N^d} \int_{Q^N_k}c_{i,k+e}\psi(x)\dd x = \sum_{k\in\bbZ_N^d} \int_{Q^N_k}c_{i,k}\psi(x-e/N)\dd x \overset{N\to\infty}{\longrightarrow} \langle\rho_i,\psi\rangle. 
\end{align*}
This proves the first claim. The second claim follows by Fatou's lemma and the pointwise convergence of the integral.
\end{proof}

Having linked the cosh-type and quadratic dual dissipation potentials, we are now in the position to obtain the lower limit for the diffusive rates using duality arguments. For the reactive rate part there is no change in the structure, thus allowing us to again employ Ioffe's liminf theorem.

\begin{proposition}\label{prop:LiminfRate}
    Assume \eqref{ass:Reactions1} and let $(c^N,F^N,J^N)$ satisfy the a priori estimates \eqref{eq:AprioriEstimBasic}. Moreover, let $\rho$ be the limit from $\iota_Nc^N$ from Proposition \ref{prop:strong_compactness}, and let $f$ and $j$ be the limits of $\iota_{N,\diff}F^N$ and $\iota_{N,\react}$ from Proposition \ref{prop:FluxBound}, respectively. Then, we have
    \[
    \liminf_{N\to\infty} \int_0^T R_N(c^N,F^N,J^N)\dd t\geq\int_0^T \mathcal{R}(\rho,f,j) \dd t.
    \]
\end{proposition}
\begin{proof}
    First, we consider the reactive part. 
    Observe that we have \[R_{N,\react}(c,J) = \calR_\react(\iota_N c^N,\iota_N J^N).\]
    Hence, the estimate for the reactive parts follows from Ioffe's liminf theorem, \cite[Thm.\,2.3.1]{Buttazzo1989}, for the convex function $J\mapsto \CC(J|\kappa_r (c^{\alpha^r} c^{\beta^r})^{1/2})$ from the weak-$\rmL^1$ convergence of $J^N$ and the strong-$\rmL^1$ convergence of $(\rho^N)^{\tfrac 1 2 (\alpha^r + \beta^r)}$ by \eqref{ass:Reactions1}.
    
    For the diffusive part, we employ the Lemma \ref{lem:EstimateDualDissipationPotential}. By the duality of $\C$ and $\C^\ast$, we have
    \begin{align*}
        \langle F^N,&\overline\nabla\iota_N^\ast\varphi\rangle_N 
        = \frac{1}{N^d} \sum_{i\in I}\sum_{k\in\bbZ_N^d} \sum_{e\in \Eu} (\overline\nabla\iota_N^\ast\varphi_i)_{k,e} F^N_{i,k,e}\\
        &\le \frac{1}{N^d} \sum_{i\in I}\sum_{k\in\bbZ_N^d} \sum_{e\in \Eu} \pra[\big]{N^2\delta_i \bra[\big]{c^N_{i,k}c^N_{i,k+e}}^{1/2}\C^\ast((\overline\nabla\iota_N^\ast\varphi_i)_{k,e}) + \C\bra{F^N_{i,k,e}|N^2\delta_i \bra[\big]{c^N_{i,k}c^N_{i,k+e}}^{1/2}}}\\
        &= R^*_{N, \diff} (c^N,\overline\nabla \iota^*_N\varphi) + R_{N, \diff} (c^N,F^N).
    \end{align*} 
    Using  $\varphi \in \rmC^1([0,T]\times X)$, $\iota_Nc^N \dd x = \rho^N\rightharpoonup^*\rho$, and $\iota_N F^N \rightharpoonup^* f$ and Lemma \ref{lem:EstimateDualDissipationPotential}, we get
    \begin{align*}
        \int_0^T \langle f,\nabla\varphi\rangle - \cR_{\mathrm{diff}}^*(\rho,\nabla\varphi) \dd t 
        &\le \lim_{N\to\infty} \int_0^T \langle F^N,\overline\nabla\iota_N^\ast\varphi\rangle_N \dd t - \limsup_{N\to\infty} \int_0^T R^*_{N, \diff} (c^N,\overline\nabla \iota^*_N\varphi) \dd t\\
        &\le \liminf_{N\to\infty} \int_0^T \langle F^N,\overline\nabla\iota_N^\ast\varphi\rangle_N - R^*_{N, \diff} (c^N,\overline\nabla\iota^*_N\varphi) \dd t\\
        &\le \liminf_{N\to\infty} \int_0^T R_{N, \diff} (c^N,F^N) \dd t.
    \end{align*} 
    The left-hand side is a quadratic functional in $\nabla \varphi$ and can be continuously extended to $V\coloneqq \overline{\set{\nabla\varphi: \varphi\in \rmC^1([0,T]\times X)}}^{\rmL^2(0,T;\rmL^2_\rho(X^\dom;Y_\diff^\tar))}$. Taking the supremum in $V$, we obtain $\int_0^T\mathcal{R}_\diff(\rho,f)\dd t$ by duality. This finishes the proof.
\end{proof}

Finally, we are in the position to prove Theorem \ref{thm:Convergence}.
\begin{proof}[Proof of Theorem \ref{thm:Convergence}]
The asserted liminf inequality for dissipations follows from Proposition \ref{prop:LiminfSlope} and \ref{prop:LiminfRate} together with the compactness results Proposition \ref{prop:FluxBound}, Proposition \ref{prop:strong_compactness}, Lemma \ref{lem:CE_closed} and Proposition \ref{prop:convergence_of_derivatives}. 

Regarding the liminf inequality for energies, we note that by definition of $\iota_N$ we have $E_N(c^N) = \calE(\rho^N)$. Therefore, as $\calE$ is convex, the liminf inequality follows from the pointwise-in-time weak-$^\ast$ convergence of $\rho^N(t)\rightharpoonup^\ast\rho$ established in Proposition \ref{prop:strong_compactness}.
\end{proof}
\section{Proof of chain rules}\label{s:CRs}

In this section we prove the two chain rules stated above in Lemma~\ref{lem:ChainRule.Discr} and Theorem \ref{thm:CRContinuous} as well 
as the discrete and the continuum versions of the Energy-Dissipation principles of 
Theorem~\ref{thm:DiscEDB.ODE} and Proposition~\ref{prop:CEDB.vs.CRDS}, respectively.

\subsection{Chain rule for discrete reaction-diffusion system}
\label{ss:CRDiscreteRDS}

We provide a full proof of the detailed chain rule in the discrete setting. 
A similar approach is given in \cite{PRST22} for the case of linear reactions, 
i.e., where  $\alpha^r$, $\beta^r$ are Euclidean unit vectors. 

Recall from Section \ref{sec:MainResults} the  modification $B(c,\dot c)$ of the duality product 
$\rmD E_N(c)\cdot \dot c$, given by
\[
B(c,v)= \sum_{i\in I}\sum_{k\in \bbZ^d_N} b\big(\tfrac{c_{i,k}}{w^N_{i,k}},
v_{i,k}\big) \quad  \text{with } b(a,s)= 
    \begin{cases} s \log  a &\text{for } a>0, \\
      0 & \text{for } a=0.     \end{cases}
\]
The special treatment of the singularity of $\log c_{i,k}$ at $c_{i,k}$ leads
to nontrivial implications that can only be handled due to the property
that the underlying (discrete) reaction-diffusion system preserves
non-negativity or even more positivity. For the linear scalar diffusion
equation $u=\Delta u$ in our torus, it is well-known that $u(t_*,x_*)=0$ for
some $t_*>0$ implies $u(t,x)=0$ for all $t>0$ and $x \in \bbT^d$. A similar
statement holds for the discretization on $\bbZ^d$. However, for our
reaction-diffusion system the situation is more complex, since some components
are may vanish (identically) while other are positive. The gradient structure
induced by the detailed-balance condition will provides enough control to
handle the arising degeneracies.

\begin{proposition}[Chain rule for the discrete setting]
\label{prop:ChainRule.Discr}\mbox{} \\
{\upshape a)} Consider $c\in \AC([0,T];X_N)$ such that 
\[
t \mapsto B(c(t),\dot c(t)) \ \text{ lies in } \rmL^1([0,T]).
\]
Then, $t \mapsto E_N(c(t))$ is absolutely continuous and we have the chain rule
formula
\begin{equation}
  \label{eq:CR.Discr}
  \frac{\rmd}{\rmd t} E_N(c(t)) = B(c(t),\dot c(t)) \ \text{ for a.a. } t \in
  [0,T]. 
\end{equation}
{\upshape b)} Consider a fixed vector $(c,F,J)$ such that 
$c_{i,k}=0$ implies $\big(\nablaDisc^*(F,J)\big)_{i,k}=0$, then 
\begin{equation}
  \label{eq:B.estim.RR*}
  \abs{B(c,\nablaDisc^*(F,J))} \leq R_N(c,F,J) + S_N(c)=R_{N,\mathrm{diff}}(c,F){+} 
R_{N,\mathrm{react}}(c,J){+} S_{N,\mathrm{diff}}(c){+} 
S_{N,\mathrm{react}}(c).
\end{equation}
{\upshape c)} We have the equality
\begin{equation}
  \label{eq:B.equal.RR*}
    B(c,\nablaDisc^*(F,J)) = R_N(c,F,J) + S_N(c)
\end{equation}
 if and only if 
\begin{equation}
  \label{eq:DiffReactFlux}
F_{i,k,e} = - \delta_i N^2 \sqrt{w^N_{i,k}w^N_{i,k+e}} \Big(\frac{c_{i,k}}{w^N_{i,k}}
  - \frac{c_{i,k+e}}{w^N_{i,k+e}}\Big) \quad \text{and} \quad 
J_{r,k}= \kappa_r( w^N_{k})^{(\alpha^r+\beta^r)/2}\Big(
\frac{c_k^{\alpha^r}}{(w^N_k)^{\alpha^r}}
-\frac{c_k^{\beta^r}}{(w^N_k)^{\beta^r}} \Big).
\end{equation}
In particular, \eqref{eq:B.equal.RR*} is equivalent to \eqref{eq:RDSysDisc}. 
\end{proposition}
\begin{proof} 
Without loss of generality, we may assume $w^N_{i,k}=1$. Moreover, we
may simplify the notation by only considering reactions, since for fixed $N$
the jumps of $c_i$ from $k$ to $k{+}e$ are simple exchange reactions with
reaction factor  $\delta_i N^2$. (Formally, one can define $I'=I\times
\bbZ^d_N$ and $R'=R\times \bbZ^d_N\cup I\times E\times \bbZ^d_N$.) Thus, in the
rest of the proof, we omit the occurrence of $N$ and $k\in \bbZ^d$, writing
$c(t)\in [0,\infty)^I$  and $\nablaDisc=\Gamma$. 

Part a). It suffices to consider only one species $c_i$ and omit the index $i$,
since $E$ and $B$ are both independent sums over $i$. 

From $c\in \AC([0,T])= \mathrm{W}^{1,1}([0,T])$ we have $c(t)\in [0,L]$ for some $L>0$. 
For $c\geq 0$ and $\eps\in(0,1)$ we define $\beta_\eps(c) = \max\{\log \eps,\log c\}$ with
$\beta_\eps(0)=\log\eps$ and $E_\eps(c) = \int_1^c \beta_\eps(s)\dd s$. With
$\beta(c)=\log c$ for $c>0$ and $\beta(0)=-\infty$ we have 
\[
\forall\, c\geq 0:\quad |\beta_\eps(c)|\leq |\beta(c)| \ \text{ and } \  
0\leq E_\eps(c) \nearrow E(c)  \text{ as } \eps \searrow  0. 
\]
Since $E_\eps$ is locally Lipschitz the chain rule holds: for  $ 0\leq s < t\leq
T$ we have
\begin{equation}
  \label{eq:CR.Ek}
  E_\eps(c(t))-E_\eps(c(s)) = \int_s^t b_\eps(c(r),\dot c(r)) \dd r 
\quad \text{with } b_\eps(c,v)=\beta_\eps(c)v. 
\end{equation}
We claim that $|b_\eps(c(r)),\dot c(r))| \leq |b(c(r)),\dot c(r))|$ a.e.\ in
$[0,T]$, namely on the set where $\dot c(r)$ exists. 
For $c(r)>0$ this follows immediately from $|\beta_\eps(c)|\leq |\log(c)|$. 
If $c(r)=0$ and $\dot c(r)$ exists, then $c(t)\geq 0$ for all $t\in[0,T]$ implies $\dot c(r)=0$;
and hence $b_\eps(c(r),\dot c(r))=0=b(c(r),\dot c(r))$. 

By assumption $r \mapsto |b(c(r),\dot c(r))| $ is an integrable majorant
for the integrand in \eqref{eq:CR.Ek}. Moreover, we have $b_\eps(c(r),\dot c(r))
\to b(c(r),\dot c(r))$ a.e. Hence, we are able to pass to the limit $\eps\to
0$ in \eqref{eq:CR.Ek}, and the chain rule formula \eqref{eq:CR.Discr} follows. \medskip

Part b). Estimate \eqref{eq:B.estim.RR*} follows from the duality of
$R_N$ and $R_N^*$ in the case $c_i\geq \delta>0$, since $\rmD E_N(c)$ is
well defined and $S_N(c)= R_N^*(c,-\Gamma \rmD E_N(c))$. 

For the general case we introduce $c^\eps = (c_i{+}\eps)_{i\in I}$ for which
\eqref{eq:B.estim.RR*} holds. 
We can pass to the limit $\eps \searrow 0$ by noting the
convergences on the right-hand side, as $R_N$ and $S_N$ are continuous in
$c$ (for fixed $(F,J)$), and on the left-hand side as well. For the latter we use
the continuity of $c \mapsto \log c$ if $c_i>0$ and
$(\Gamma^* J)_i=0$ if $c_i=0$.  

Part c). The case $c_i\geq \delta >0$ is trivial since \eqref{eq:B.estim.RR*}
implies
\[
\rmD E_N(c) \cdot \Gamma^*J \leq R_N(c,J) + R_N^*\big(c,-\Gamma \rmD E_N(c)\big).
\]
By strict convexity of $R_N(c,\cdot)$ we have equality if and only if 
$J= - \rmD_\xi R_N^*\big(c,-\Gamma \rmD E_N(c)\big)$. Using
\eqref{eq:CStarandLog} we find the desired relation in \eqref{eq:DiffReactFlux}. 

For the general case, fix a vector $(c,J)$. We decompose the sets $I$ and $R$ into vanishing and
positive parts:
\[
I_\text{v}\coloneqq\{\,i\in I\,|\, c_i=0\,\}, \quad I_\text{p}\coloneqq I\setminus I_\text{v}, \quad 
R_\text{v}\coloneqq  \{\, r\in R\,|\,\exists\,i\in I_\text{v}: \alpha_i^r{+}\beta_i^r>0\,\}, 
\quad R_\text{p}\coloneqq R\setminus R_\text{v}.
\]
The equality \eqref{eq:B.equal.RR*} implies that the right-hand side must be
finite. As $R_N(c,J)$ contains the terms
$\CC\big(J_r|\kappa_r c^{(\alpha^r{+}\beta^r)/2}\big)$ and
$c^{(\alpha^r{+}\beta^r)/2}=0$ for all $r \in R_\text{v}$, we conclude $J_r=0$ for all
$ r \in R_\text{v}$.

Rearranging the index sets $I$ and $R$, we can write the $c=(c_\text{p},c_\text{v})$ and
$J=(J_\text{p},J_\text{v})$ with $c_\text{v}=0$ and $J_\text{v}=0$. Writing $ \rmD_\text{p}
E(c_\text{p})= \big((\log c_i)_{i\in I_\text{p}},(0)_{i\in I_\text{v}}\big)$, the desired equality
\eqref{eq:B.equal.RR*} reduces to 
\begin{equation}
  \label{eq:ReducedFenchel}
\Gamma \rmD_\text{p} E_N(c_\text{p}) \cdot (J_\text{p},0) = R_{N,\text{p}}\big(c_\text{p},(J_\text{p},0)\big)+S_N(c_P,0),  
\end{equation}
where $R_{N,\text{p}}$ is defined as $R_N$ up to reducing the sum to $r\in R_\text{p}$, which implies that $c_\text{v}$ does not appear any more. 
Since $S$ is the sum over all reactions we have 
\[
S_N(c_\text{p},0) \geq S_{N,\text{p}}(c_\text{p})\coloneqq\sum_{r\in R_\text{p}}2\kappa_r\big(c^{\alpha^r/2} 
{-} c^{\beta^r/2}\big)^2
= R^*_{N,\text{p}}\big(c_\text{p}, - \Gamma \rmD_\text{p} E_N(c_\text{p})\big) . 
\]
Replacing $S_N$ by $S_{N,\text{p}}$, the convex duality of $R_N$ and
$R^*_N$ implies 
\[
(J_\text{p},0)= - \rmD_\xi R^*_{N,\text{p}}\big(c_\text{p}, - \Gamma \rmD_\text{p} E_N(c_\text{p}) \big) .
\]
Moreover, the equality in \eqref{eq:ReducedFenchel} holds if and only if 
\[
S_N(c_P,0)- S_{N,\text{p}}(c_\text{p})= \sum_{r\in R_\text{v}} 2\kappa_r\big(c^{\alpha^r/2} 
{-} c^{\beta^r/2}\big)^2 =0 .
\] 
Thus, we find $c^{\alpha^r}=0=  c^{\beta^r}$ for all $r \in R_V$ (since at
least one is $0$ by definition and the difference vanishes by the last relation). 
Therefore, \eqref{eq:DiffReactFlux} holds true also in the general case. 
\end{proof} 

Proposition \ref{prop:ChainRule.Discr} is exactly what we need to show that EDB
solutions in the sense of Definition \ref{def:sol_disc} are ODE solutions of
the discrete system \eqref{eq:RDSysDisc}, which will complete the proof of the 
discrete Energy-Dissipation Principle. 
 
\begin{proof}[Proof of Theorem \ref{thm:DiscEDB.ODE}.] 
\label{ss:Proof.Disc.EDP} 
The direction from \eqref{eq:RDSysDisc} to a EDB solution is 
classical, as the map $t \mapsto c(t)$ is in fact $\rmC^1([0,T];\calM(X_N))$. 

For the opposite direction, we first observe that $L_N(c,F,J)=0$ implies
$D_N(c,F,J)< \infty $. By \eqref{eq:B.estim.RR*} we see that $t\mapsto
 B(c(t), \dot c(t))$ lies in $\rmL^1([0,T])$. Hence, the chain rule
\eqref{eq:CR.Discr} holds. Thus, we have 
\begin{align*}
0&=L_N(c,F,J) = \int_0^T\big( B(c,\dot c)+R_N(c,F,J)+ S_N(c)\big)
\dd t.
\end{align*} 

Using $(c,F,J)\in \lCE_N$ and \eqref{eq:B.estim.RR*}, the integrand is
non-negative, hence we conclude that the integrand has to vanish a.e.\ in
$[0,T]$. Thus, $F$ and $J$ are given by the formulas in
\eqref{eq:DiffReactFlux}. Inserting this into the discrete continuity equation 
$\dot c +\nablaDisc^*(F,J)=0$ gives exactly the desired ODE
\eqref{eq:RDSysDisc}. 
\end{proof}

\subsection{Chain rule for reaction-diffusion system on the torus}
\label{sec:ChainRuleCont}

Before we prove Theorem \ref{thm:CRContinuous}, we first collect and prove 
two lemmas that we need in the following. First, we have the following 
inequality for the perspective function $\CC$. 

\begin{lemma}
\label{lem:PerspFcnLip}
Let $\sigma>0$. Then we have
\begin{equation}
\forall\, j\in\R \ \forall\, a,b\geq\sigma: 
\quad|\CC(j|a)-\CC(j|b)|\leq  { \frac{2|j|}{\sigma}\,|a {-}b|. }
\label{eq:AuxInequForC}
\end{equation}
\end{lemma}

\begin{proof}
    We observe that $\partial_a \CC(j|a)=m(j/a)$ with
    $m(r)=\C(r)-r\C'(r)= 4 - 2\sqrt{4{+}r^2}\leq 0$. Then, using $|m(r)|\leq 2|r|$ and
    $\CC(j|b)- \CC(j|a)= \int_a^b m(j/y) \,\mathrm d y$, the result follows.
\end{proof}

The lemma now helps to bound the difference once we have a bound on  $\rho$ 
and $j\in\rmL^{\C}([0,T]\times Y_\react)$. For this we recall the Hardy-Littlewood maximal function (see e.g.\ \cite{SteinHarmonicAnalysis}) which for a given function $g:\R^{d}\to \R$ is defined by
\[
\rmM g(x)=\sup_{B\ni x}\frac{1}{|B|}\int_{B}|g(y)|\dd y,
\]
where $B\subset\R^{d}$ are balls including $x$. It follows that $\sup_{\eps>0}|g*k_{\eps}(x)|\leq \rmM g(x)$ for any measurable $g$. Regarding integration, there are classical results, showing that for $1<p\leq \infty$ it holds $\rmM g\in\rmL^{p}$ if $g\in\rmL^{p}$. 
In the limiting case $p=1$, one has the weaker statement 
$j\in\L\log\L$ (i.e.\ $\int_\OmegaT \C(|j|) \dd x \dd t< \infty$) if and only if 
$\rmM j\in\rmL^{1}$, see \cite{Stei69NCLL}. 

In the following proposition we will combine Lemma \ref{lem:PerspFcnLip} with the 
estimate through the maximal function. For this, we need the magical estimate 
\eqref{eq:CC.prop.d}, where the assumption $\frac12|\alpha^r{+}\beta^r|_1 \lneqq \pcrit$ 
is crucial to obtain $ \rho^{(\alpha^r+\beta^r)/2} \in \rmL^q(\OmegaT)  $ 
with $q \gneqq 1$. Unfortunately, a superlinear estimate for 
$|\rho|^{\pcrit}$ as obtained in Proposition \ref{prop:integrability} would not be 
enough as is shown by the counterexample in Remark \ref{rem:App.Counterexa}.

The following result can also be seen as a commutator estimate, since it is essential
to estimate $(\rho{*}k_\eps)^{(\alpha^r+\beta^r)/2}$ against $ \big(\rho^{(\alpha^r+\beta^r)/2}\big)*k_\eps$, where $k_\eps$ is a smoothing kernel.

\begin{proposition}[Commutator estimate]
\label{prop:React.Commutator.Estim}
 Assume $\frac12|\alpha^r{+}\beta^r|_1 \lneqq p$. 
 Consider $\rho\in\rmL^p ([0,T]\times\bbT^d)$ and assume $\rho_i\geq\sigma>0$ a.e.\ in $\OmegaT$ for all $i\in I$ and  $\iint_\OmegaT \CC(j_r| 
 \rho^{\left(\alpha^r+\beta^r\right)/2})\dd x\dd t<\infty$. 
Let $k_{\eps}$ be a mollifier approximating the identity, and $j^{\eps}_r\coloneqq j_r*k_{\eps}$,
$\rho^{\eps}\coloneqq\rho*k_{\eps}$. 
Then, we have 
\begin{align}
  \label{eq:LimSup.CC}
    \limsup_{\eps\to0}\iint_\OmegaT \CC(j_r^{\eps}|(\rho^{\eps})^{\left(\alpha^r+\beta^r\right)/2})\dd x\dd t\leq \iint_\OmegaT \CC(j_r|\rho^{\left(\alpha^r+\beta^r\right)/2})\dd x\dd t.
\end{align}    
\end{proposition}
\begin{proof} 
We drop the fixed index $r$ for $j_r, \: \alpha_r$, and $\beta_r$ for the remainder 
of the proof  and use the short-hand notations
\[
\gamma \coloneqq \frac12(\alpha{+}\beta), \quad a\coloneqq\rho^\gamma , \quad 
a_{\eps} \coloneqq (\rho^\eps)^\gamma. 
\]
Using $ |\gamma|_1\lneqq p $ we have $ a \in \rmL^q(\Omega_T)$ with $q= p/|\gamma|_1>1$. 
Thus, we can use the magical property \eqref{eq:CC.prop.d} 
of $\C$ and find 
\[
\iint_{\OmegaT} \C(j) \dd x \dd t \leq  \iint_\OmegaT \Big( \frac{q}{q{-}1} \CC(j|a) 
+\frac4{q{-}1}  a^ q \Big) \dd x \dd t  < \infty. 
\]
His implies $j\in\L\log\L(\OmegaT)$, such that its Hardy-Littlewood maximal 
function (done in the space-time domain $\OmegaT$) is integrable, i.e., 
$\rmM j\in\rmL^{1}([0,T]\times Y_\react)$, see \cite{Stei69NCLL}.
Thus, we find the pointwise estimate 
\begin{equation}
    \label{eq:L1.majorand}
\forall\: \eps>0: \qquad | j_\eps (t,x) |  \leq  \rmM j(t,x)  \quad 
\text{almost everywhere in } \OmegaT,
\end{equation}
this means that the family $ (j_\eps)_\eps$ has a $\rmL^1$ majorant. 

Using the shorthand $x\wedge y= \min\{ x,y\}$ for $x,y\in \R$, the monotonicity of 
$a \mapsto \CC(s|a)$, and the bound \eqref{eq:AuxInequForC} for the derivative in Lemma \ref{lem:PerspFcnLip}, we find for any $M>0$ that 
\begin{align}
\label{eq:CC.estim.M}
  \CC(j_\eps| a_\eps) &\leq \CC(j_\eps| M\wedge a_\eps) 
                       \leq \CC(j_\eps| M\wedge b_\eps) + g_\eps
    \\
    &\text{with } g_\eps = \frac2{\sigma^{|\gamma|_1}} \,|j_\eps| \:\big| (M\wedge b_\eps)
     - (M\wedge a_\eps) \big| \ \text{ and } \ b_{\eps,M} = (a\wedge M)*k_\eps.  \nonumber
\end{align}
Using $b_{\eps,M} \to a\wedge M $ and $a_\eps \wedge M\to a \wedge M$ strongly in $\rmL^q  (\OmegaT) $  and weakly-$^\ast$ in $\rmL^\infty(\OmegaT)$ as $\eps\to 0$ together with 
\eqref{eq:L1.majorand}, Lebesgue's dominated convergence theorem gives 
$G_\eps =\iint_\OmegaT | g_\eps| \dd x \dd t \to 0$. 

Using $j_\eps = j*k_\eps$, $ b_\eps = a*k_\eps$ and the joint convexity of $(j,a)\mapsto 
\CC(j|a)$ allows us to apply Jensen's inequality. Hence, integrating the estimate
\eqref{eq:CC.estim.M} over $ \OmegaT $ we find  
\begin{align*}
\iint_\OmegaT \CC(j_{\eps}|a_\eps)\dd x\dd t  &
\leq \iint_\OmegaT \CC(j_{\eps}|a_\eps\wedge M )\dd x\dd t  
\leq \iint_\OmegaT \big( \CC(j_{\eps}|b_{\eps,M}) + g_\eps \big)\dd x\dd t  
\\
& \overset{\text{Jensen}}\leq \iint_\OmegaT \CC(j|a \wedge M) \dd x\dd t  + G_\eps.   
\end{align*}
Keeping $M$ fixed and taking the upper limit $\eps\to 0$ we find 
\[
\forall \, M\geq 1: \quad \limsup_{\eps\to 0} \iint_\OmegaT 
\CC(j_{\eps}|a_\eps)\dd x\dd t  \leq  \iint_\OmegaT \CC(j|a \wedge M) \dd x\dd t .
\]
To perform the limit $M\to \infty$, we use $ \CC(j|a \wedge M) \leq \CC(j| a\wedge 1) 
\leq \max\big\{ \CC(j | 1), \CC( j| \rho^\gamma)\big\} \in \rmL^1(\OmegaT)$ 
due to the assumption and $\CC(j|1)=\C(j) \in \rmL^1(\OmegaT)$. Hence, by dominated convergence the limit $M\to \infty$ provides the desired estimate \eqref{eq:LimSup.CC}.
\end{proof}

\begin{remark}[Convexity instead of commutator estimate]
The above commutator estimate can be avoided if the function $(\rho,s) \mapsto \CC(s| \rho^\gamma)$ is jointly convex. Then , the result follows simply by applying 
Jensen's inequality for convolutions, i.e., $\iint \CC( J{*}k_\eps|(\rho{*}k_\eps)^\gamma) \dd x \dd t \leq \iint \CC( j|\rho^\gamma) \dd x \dd t$. 
This argument is usually used for linear reactions, see e.g.\ \cite{Step21EDPLRDS,PRST22,HT23}. 

Indeed, the joint convexity holds if and only if $|\gamma|_1\leq 1$. Since 
$\CC(s|g(\rho))$ is the Legendre-Fenchel transorm of $g(\rho)\C^*(\zeta)$, we have 
joint convexity if and only if $\rho \mapsto g(\rho)$ is concave. For 
$g(\rho)=\rho^\gamma$ the second derivative $\rmD^2 g$ has the explicit form 
\[
\rmD^2 g(\rho) = - \rho^\gamma\, \diag(1/\rho_i)_I \,A(\gamma)\, \diag(1/\rho_i)_I 
\quad \text{with } A(\gamma) = \diag(\gamma) - \gamma {\otimes}\gamma. 
\]
Hence, $g$ is concave if and only if $A(\gamma)$ is positive semi-definite. However, we have 
\[
b\cdot A(\gamma) b = \sum_{i\in I}\gamma_i b_i^2 -\Big(\sum_{i\in I} \gamma_ib_i\Big)^2 \geq 
 \sum_{i\in I}\gamma_i b_i^2 -\Big(\sum_{\bar\imath\in I} \gamma_{\bar\imath}\Big)\,\Big(\sum_{i\in I} \gamma_i b_i\Big)
 = \Big( 1 -\sum_{\bar\imath\in I} \gamma_{\bar\imath}\Big) \sum_{i\in I}\gamma_i b_i^2 .
\]
Hence, $\sum_{\bar\imath\in I}\gamma_{\bar\imath}\leq 1$ implies the desired concavity. However, considering 
the function $t\to g(tc)= t^\lambda c^\gamma$ gives $\lambda =\sum_{\bar\imath\in I}\gamma_{\bar\imath} $, and concavity implies $\lambda\leq 1$.
\end{remark}

Putting the above results together, we can now prove Theorem \ref{thm:CRContinuous}.

\begin{proof}[Proof of Theorem \ref{thm:CRContinuous}]
The proof is performed in several steps. First, 
we regularize and shift the density by a positive constant and show the chain rule 
for that situation. Then follows 
the harder part of estimating the limits. For this we rely on 
Proposition~\ref{prop:React.Commutator.Estim}.\\
{\bf 1. Step (Regularization): } We note that from the bound on the energy and 
dissipation, the curve  $t\mapsto\rho(t)$ is absolutely continuous with values 
in $(\mathrm{W}^{1,\infty}(X))^*$ and it has a Lebesgue density $\rho\dd x$ for 
almost all $t\in[0,T]$. Furthermore, we have  $\rho\in \rmL^{\pcrit}([0,T]\times X)$,
$f\in\rmL^1(0,T,Y_\diff)$, $j\in\rmL^1([0,T]\times Y_\react)$ by Proposition~\ref{prop:strong_compactness} and Proposition~\ref{prop:FluxBound}.  Given $\sigma>0$ and a mollifier 
$(k_\eps)_{\eps>0}$, we define the component-wise shifted and regularized trajectory 
\[
\rho^{\eps,\sigma}\coloneqq(\rho+\sigma)\ast k_{\eps},
\]
and correspondingly the regularized fluxes $f^{\eps}\coloneqq f\ast k_{\eps}$ and
$j^{\eps} \coloneqq j\ast k_{\eps}$. Clearly, $(\rho^{\eps, \sigma}, f^{\eps},
j^{\eps})\in\CE$, where we have used that for the reactions the stoichiometric 
matrix $\Gamma^*$ commutes with the mollification. Moreover, we have
\[
\rho^{\eps,\sigma}\to\rho\text{ in }\rmL^{1}([0,T]\times X),\quad 
j^{\eps}\to j\text{ in }\rmL^{1}([0,T]\times Y_\diff),\quad 
f^{\eps}\to f\text{ in }\rmL^{1}([0,T]\times Y_\react).
\]
{\bf 2. Step (Chain rule for regularized curve): } Now, we show that for fixed 
$\eps,\sigma>0$ the trajectory $t\mapsto{\cal E}(\rho^{\eps,\sigma}(t))$
is absolutely continuous and satisfies the upper chain rule. For this,
we first note that there is a constant $M_{\eps}>0$ such that 
$\|\rho^{\eps,\sigma}\|_{\rmL^{\infty}([0,T]\times\bbT^{d})}\leq M_{\eps}$
and we have $\rho^{\eps,\sigma}\geq\sigma>0$. To show absolute continuity,
we fix $s,t\in[0,T]$, and since on $[\sigma,M_{\eps}+\sigma]$ the Boltzmann
function $[\sigma,M_{\eps}+\sigma]\ni r\mapsto\LB(r)\in[0,\infty)$
is Lipschitz continuous, i.e., there is $L_{\sigma,\eps}>0$ such that
$\forall\, r_{1},r_{2}\in[\sigma,M_{\eps}+\sigma]:\quad|\LB'(r_{1}) 
-\LB^{'}(r_{2})|\leq L_{\sigma,\eps}|r_{1}-r_{2}|,$ we compute
\begin{align*}
{\cal E}(\rho^{\eps,\sigma}(t))-{\cal E}(\rho^{\eps,\sigma}(s)) 
& \leq\sum_{i\in I}\int_{\bbT^{d}}|\LB(\rho_i^{\eps,\sigma}(t))- 
 \LB(\rho_i^{\eps,\sigma}(s)|\dd x
 \\
& \leq L_{\sigma,\eps}\sum_{i\in I}\int_{\bbT^{d}}|\rho_i^{\eps,\sigma}(t)-
 \rho_i^{\eps,\sigma}(s)|\dd x =L_{\sigma,\eps}\sum_{i\in I}\int_{\bbT^{d}} 
  |\rho_i^{\eps}(t)-\rho_i^{\eps}(s)|\dd x.
\end{align*}
The mollifier $k_{\eps}$ is a test function in $\rmC^\infty_c$
with a (possibly bad) Lipschitz constant $C_{\eps}$, which implies 
\[
{\cal E}(\rho^{\eps,\sigma}(t))-{\cal E}(\rho^{\eps,\sigma}(s))\leq 
 C_{\eps}L_{\sigma,\eps}\|\rho(t)-\rho(s)\|_{\left(W^{1,\infty}\right)^*}.
\]
Hence, $t\mapsto{\cal E}(\rho^{\eps,\sigma})$ is absolutely continuous,
and we obtain by the differentiability of $r\mapsto\LB(r)$
on $[\sigma,M_{\eps}+\sigma]$ that
\begin{align*}
\frac{\rmd}{\rmd t}{\cal E}(\rho^{\eps,\sigma}(t)) 
& =\sum_{i\in I}\int_{\bbT^{d}} 
   \log(\rho^{\eps,\sigma}_i(t)/\omega_i)\partial_t{\rho}_i^{\eps,\sigma}(t)\dd x
  =\langle \log(\rho^{\eps,\sigma}(t)/\omega),
    \left(-\mathrm{div}f^{\eps}(t)+\Gamma^{*}j^{\eps}(t)\right)\rangle
\\
 & =\langle \nabla\log(\rho^{\eps,\sigma}(t)/\omega), 
     f^{\eps}(t)\rangle+  \langle \Gamma\log(\rho^{\eps,\sigma}(t)/w), j^{\eps}(t)\rangle ,
\end{align*}
which by integrating in time leads to
\[
{\cal E}(\rho^{\eps,\sigma}(t))-{\cal E}(\rho^{\eps,\sigma}(s))=\int_{s}^{t} 
\langle \nabla\log(\rho^{\eps,\sigma}(r)/\omega), 
     f^{\eps}(r)\rangle+  \langle \Gamma\log(\rho^{\eps,\sigma}(r)/w), j^{\eps}(r)\rangle \dd r.
\]
Here, we have used the chain rule with the nice test function $\nabla\log(\rho^{\eps,\sigma}(r)/w)$.
In particular, by Legendre duality we obtain
\begin{align*}
{\cal E}(\rho^{\eps,\sigma}(t))-{\cal E}(\rho^{\eps,\sigma}(s)) & =-\int_{s}^{t}\langle \nabla\log(\rho^{\eps,\sigma}(r)/\omega), 
     -f^{\eps}(r)\rangle+  \langle \Gamma\log(\rho^{\eps,\sigma}(r)/w), -j^{\eps}(r)\rangle \dd r\\
 & \geq-\int_{s}^{t}{\cal R}_\diff(\rho^{\eps,\sigma},f^{\eps})+{\cal S}_\diff(\rho^{\eps,\sigma})+{\cal R}_\react(\rho^{\eps,\sigma},j^{\eps})+{\cal S}_\react(\rho^{\eps,\sigma})\dd r.
\end{align*}
Hence, it follows the chain rule inequality for the regularized curve that $\mathcal{L}^{[s,t]}(\rho^{\eps,\sigma},f^\eps,j^\eps)\geq 0$.\\
{\bf 3. Step (Limit $\sigma\to0$ and $\eps\to0$): } First, we observe that convergence of the energies is clear due to the convexity. Hence, it suffices to show that 
\begin{align*}
\limsup_{\sigma\to0}\limsup_{\eps\to0}\int_{s}^{t}{\cal R}_\diff(\rho^{\eps,\sigma},f^{\eps})\dd\tau & \leq\int_{s}^{t}{\cal R}_\diff(\rho,f)\dd\tau,\\
\limsup_{\sigma\to0}\limsup_{\eps\to0}\int_{s}^{t}S_\diff(\rho^{\eps,\sigma})\dd\tau & \leq\int_{s}^{t}S_\diff(\rho)\dd\tau,\\
\limsup_{\sigma\to0}\limsup_{\eps\to0}\int_{s}^{t}{\cal R}_\react(\rho^{\eps,\sigma},j^{\eps})\dd\tau & \leq\int_{s}^{t}{\cal R}_\react(\rho,j)\dd\tau,\\
\limsup_{\sigma\to0}\limsup_{\eps\to0}\int_{s}^{t}S_\react(\rho^{\eps,\sigma})\dd\tau & \leq\int_{s}^{t}S_\react(\rho)\dd\tau.
\end{align*}
We will treat all four estimates and also the convergences $\sigma\to0$ and $\eps\to0$ separately. In each term we will consider the limit $\eps\to0$ first, sending $\sigma\to0$ afterwards.\\
{\bf 3a (Diffusive terms): } The rate term $\calR_\diff$ as well as the slope term $\calS_\diff$ are convex functionals. Hence, the upper limit bound for $\eps\to0$ follows by Jensen's inequality (see e.g.\ \cite[Lem.\,8.1.10]{AGS}) together with $\nabla\rho_i^{\eps}\rightharpoonup\nabla\rho_i$, $j^{\eps}_r\rightharpoonup j_r$ in $\rmL^{1}([0,T])$ and $\rho_i^{\eps}\rightharpoonup\rho_i$. For the limit $\sigma\to 0$, we simply observe that $\rho_i+\sigma\geq\rho_i$
which implies that $\frac{|f_i|^{2}}{\rho_i+\sigma}\leq\frac{|f_i|^{2}}{\rho_i}$. Moreover, we have $\nabla \sqrt{\rho_i+\sigma} = \frac{\sqrt{\rho_i}}{\sqrt{\rho_i+\delta}}\nabla\sqrt{\rho_i}\leq \nabla\sqrt{\rho}$.
This proves the desired estimate for the diffusive terms, both the slope and the rate term.\\
{\bf 3b (Reactive rate term): } The limit $\eps\to 0$ was shown in 
Proposition~\ref{prop:React.Commutator.Estim}, where we now rely on 
Assumption \eqref{ass:Reactions2*} with $p=\pcrit$. 
For the limit $\sigma\to 0$, we again use the 
monotonicity of the perspective function, 
to get the pointwise bound in the integrand
\[
\CC\left(j_r\Big|\sqrt{\left(\rho_i+\sigma\right)^{\alpha}\left(\rho+\sigma\right)_i^{\beta}}\right)\leq\CC\left(j_r\Big|\sqrt{\rho_i^{\alpha}\rho_i^{\beta}}\right).
\]
{\bf 3c (Reactive slope term):  } 
We use the general fact that a continuous function $\Phi:\R^I \to \R$ satisfying the growth estimate $|\Phi(u)| \leq C(1{+}|u|)^r$ defines via $u\mapsto \Phi\circ u$ a (strongly) continuous Nemitskii operator from $\rmL^{rq}(\Omega)$ into $\rmL^q(\Omega)$ for all $q\geq 1$. 

The reactive slope is the sum $\calS_\react= \sum_{r\in R}\calS_r$ with 
\[
\calS_r(\rho) = 2\kappa_r \iint_\OmegaT \!\!\Big( \omega^{(\beta^r-\alpha^r)/2}\rho^{\alpha^r} - 2 \rho^{(\alpha^r+\beta^r)/2}  + \omega^{(\alpha^r-\beta^r)/2}\rho^{\beta^r} \Big)\dd x \dd t\eqqcolon \calS_r^1(\rho)+ \calS^2_r(\rho)+ \calS^3_r(\rho). 
\]
By Assumption \ref{ass:Reactions2*} all three terms define strongly continuous mappings from $\rmL^{\pcrit}$ into $\rmL^1$, which implies that $\calS_\react=\sum_{r \in R}\calS_r$ is strongly continuous from $\rmL^{\pcrit}$ into $[0,\infty)$. Thus, 
using $\rho^{\eps,\sigma} \to \rho$ in $\rmL^{\pcrit} $ we can pass to the limit 
$\eps,\sigma\to 0$, and the result follows.
\end{proof}

It remains to prove Proposition \ref{prop:CEDB.vs.CRDS}, relating the notions of continuum EDB solutions and weak solutions with each other.

\begin{proof}[Proof of Proposition \ref{prop:CEDB.vs.CRDS}]
Under the condition $\rho_i \in [\sigma , 1/\sigma]$ it is standard to show that 
weak solutions are continuum EDB solutions. Indeed, for $i\in I$, we start from  the definition of weak solutions with $\varphi \in \rmL^2(0,T;\rmH^1(X))$ in the form 
\[
0= \int_0^T\! \big\langle \partial_t \rho_i,\varphi\big\rangle \dd t
+ \int_\OmegaT \delta_i \rho_i \nabla \log\Big(\frac{\rho_i}{\omega_i}\Big) \cdot 
\nabla \varphi + \sum_{r\in R}\gamma^r_i \kappa_r \omega^{(\alpha^r+\beta^r)/2}  \Big( \frac{\rho^{\alpha^r}}
{\omega^{\alpha^r} }  - \frac{\rho^{\beta^r}}{\omega^{\beta^r}} \Big)\varphi \dd x \dd t.
\]
Using $\rho_i \in [\sigma,1/\sigma]$ we are allowed to choose the test function $\varphi (t)= \log(\rho_i/\omega_i)$ for $t \in [t_1,t_2]$ and $0$ otherwise. 

Summing over $i\in I$ and using the classical chain rule for $\calE$ (now evaluated 
only on the interval $[\sigma/\omega^*, 1/(\omega_*\sigma)] \subset (0,\infty)$\,), we obtain 
\[
\calE(\rho(t_1)) - \calE(\rho(t_2)) + \int_{t_1}^{t_2}\!\!\int_{\bbT^d} 
\nablaCont \rmD\calE(\rho) \bullet (f,j) \dd x \dd t = 0 
\]
where $f=(f_{i})_i$ and $ j= (j_r)_r$ are given as in \eqref{eq:CR_cont}. By the definitions of $\calS$ and $\calR$ in the Definitions \ref{def:slope_cont} 
and \ref{def:R_cont} with $ \calS(\rho)= \calR^*(\rho, - \nablaCont \rmD \calE(\rho))$ 
(as $\rho_i\geq \sigma$), we have the identity 
$ \nablaCont \rmD\calE(\rho) \bullet (f,j) = \calR(\rho,f,j)+ \calS(\rho)$, 
which implies that $ (\rho,f,j)$ is a continuum EDB solution. 

For the opposite direction, we start from a continuum EDB solution $(\rho,f,j)$ such that we have $(\rho,f,j)\in \CE$, and $\calD(\rho,f,j)< \infty$, which, under the assumption $\rho_i \in 
[\sigma, 1/\sigma] $, imply the regularity
\[
\rho\in \rmL^2(0,T;\rmH^1(\bbT^d)), \quad 
f\in \rmL^2(\OmegaT), \quad j \in \rmL^1(\OmegaT) , \quad 
\partial_t \rho  \in \rmL^2(0,T;\rmH^{-1}(\bbT^d)) + \rmL^1(\OmegaT).  
\]
Moreover, the derivative $\rmD \calE(\rho) = \big( \log(\rho_i/\omega_i)\big)_{i\in I}$ is well-defined in $ \rmL^2(0,T;\rmH^1(\bbT^d))$ such 
that $ \calS(\rho)=\calR^*\big(\rho, - \nablaCont \rmD\calE(\rho)\big)$. Together, this is enough to establish the chain rule 
\begin{align*}
\frac\rmd{\rmd t} \calE(\rho(t)) & 
 = \langle \partial_t \rho, \rmD\calE(\rho) \rangle
 =  \langle \nablaCont \rmD\calE(\rho),(f,j)\rangle 
\\
&=\int_{\bbT^d} \sum_{i\in I} \nabla \log\Big(\frac{\rho_i}{\omega_i}\Big)
 \cdot f_i + \log\Big(\frac{\rho}{\omega}\Big) \bullet 
   \nablaCReact^* j \dd x.
\end{align*}
Inserting this into the relation $\calL(\rho,f,j)=0$ for continuum EDB solutions and  using that $ \calS(\rho)=\calR^*\big(\rho, - \nablaCont \rmD\calE(\rho)\big)$, we obtain
\[
\langle \nablaCont \rmD\calE(\rho),(f,j)\rangle  = \calR(\rho,f,j) + 
\calR^*\big(\rho, - \nablaCont \rmD\calE(\rho)\big) \ \text{ for a.a.\ } t\in [0,T].
\]
Since $\rho_i\geq \sigma>0$, we conclude $(f,j) = \rmD_{(\xi,\zeta)} \calR^*\big(\rho, - \nablaCont \rmD\calE(\rho)\big)$ which provides the desired flux relations \eqref{eq:CR_cont} a.e.\ in $[0,T]\times \bbT^d$.  The fact that $\rho$ is a weak solution follows now from the fact that $(\rho,f,j)$ satisfies the continuity equation $\CE$ in the sense of distributions, i.e., 
$\partial_t \rho =\nablaCont^* (f,j)$. 
\end{proof}

\appendix 
\section{Proof of the magical estimate 
\texorpdfstring{\eqref{eq:CC.prop.d}}{(3.6c)}}
\label{app:CC.estim}

Throughout, we consider $p>1$. With $\LB(r)=r \log r - r +1$ and 
$U_p(w)=\frac1{p(p{-}1)}(w^p-pw+p-1)$ we have the identity  
\begin{align}
    \label{eq:LB.p.LB}
w\LB\big(\frac cw\big) = \frac{p{-}1}p \LB(c) - (p{-}1) U_p(w) + \frac1p w^p
\LB\big(\frac c{w^p}\big) \geq \frac{p{-}1}p \LB(c) - (p{-}1) U_p(w),
\end{align}
see \cite[Eqn.\,(2.7)]{FHKM22GEAE} for an earlier occurrence. Our function $\C$ 
is the convex conjugate of $\C^*$, which is the sum of two exponentials. Hence, 
$\C$ can be written as an infimal convolution, namely 
\begin{equation}
\label{eq:CCC.convol}
\C(s) = \min\big\{2\LB(a_1)+2\LB({-}a_2) \,\big| \,  a_1+ a_2= s \big\}. 
\end{equation}

Combining this representation with \eqref{eq:LB.p.LB} we obtain a 
lower estimate on $\CC(a|w)=w\C(a/w)$ that corresponds to \eqref{eq:LB.p.LB}. 

\begin{proposition}[Magical estimate for $\CC$]
\label{prop:MagEst} 
For all $s\in \R$, $w>0$, and $p>1$ we have 
\[
 \CC(s|w) \geq \begin{cases}  \C(s) & \text{for } w\in  [0,1], \\ 
\frac{p{-}1} p\, \C(s) - 4(p{-}1)U_p(w)&\text{for }w\geq 1 . 
\end{cases}.
\]
\end{proposition}
\begin{proof} 
The estimate for $w\in [0,1]$ follows directly from the monotonicity 
\eqref{eq:CC.prop.b}. 

For $w\geq 1$ we exploit the infimal convolution \eqref{eq:CCC.convol} and 
\eqref{eq:LB.p.LB} to obtain the following chain of estimates:
\begin{align*}
 \CC(s|w)& \overset{\text{\eqref{eq:CCC.convol}}}= \  
  2\,\min_{a_1+a_2=s} \big\{w\LB(a_1/w) + w\LB({-}a_2/w) \}\\
& \geq 2 \,\min_{a_1+a_2=s} \big\{
  \tfrac{p{-}1}p \LB(a_1) - (p{-}1)U_p(w) + \tfrac{p{-}1}p \LB({-}a_2) -
  (p{-}1)U_p(w) \big\} \\
&  \overset{\text{\eqref{eq:CCC.convol}}}= \ \frac{p{-}1}p \,\C(s) - 4(p{-}1)U_p(w). 
\end{align*}
This is the desired estimate for $w\geq 1$. 
\end{proof}

The desired magical estimate \eqref{eq:CC.prop.d} now follows from $U_p(w) 
\leq w^p/(p(p{-}1))$ for $w\geq 1$ and Proposition \ref{prop:MagEst} 
by rearranging the estimate. 

The main usage of the estimate is in the integrated form namely 
\[
\int_\Omega \C(s) \dd \mu \leq \frac p{p{-}1} \int_\Omega \CC(s| w) \dd \mu + \frac4{p{-}1} \int_\Omega w^p \dd \mu. 
\]
The following example shows that estimating the integral on the left-hand 
side by the two integrals on the right-hand side is not possible for the case 
$p=1$, i.e.\ $p>1$ is essential.  We give an example with 
$\int_\Omega \CC(s|w) \dd \mu + \int_\Omega \LB(w) \dd \mu <\infty$ but 
$\int_\Omega \C(s)\dd \mu=\infty$. 

\begin{remark}[Counterexample]\label{rem:App.Counterexa} 
We let $\Omega={]0,1/2[}$, take $\mu=\mathfrak \rmL^1$ and choose 
\[
s(x)=\frac1{x \big(\log (1/x)\big)^\gamma} \quad \text{and} \quad 
w(x)=\frac1{x \big(\log (1/x)\big)^\omega} \quad \text{with }
1<\gamma<2<\omega . 
\]
This gives $s \in \rmL^1(\Omega)$,
$\int_\Omega \LB(w) \dd x<\infty$, and $\int_\Omega \C(s) \dd x =\infty$.
With $s(x)/w(x)=\big(\log(1/x)\big)^{\omega-\gamma}$ and $\C(r)\approx r\log
(1{+}r)$ for $r\gg 1$ we find $\int_\Omega \CC(s|w) \dd x <\infty.$ 
\end{remark}

\section{Superlinear functions: Proof of Lemma \ref{lem:Superlinear}}
\label{app:Superlinear}

Lemma \ref{lem:Superlinear} involves the superlinear functions $\phi$ and $\psi$ 
and constructs another superlinear function $\Xi=\Xi_{\phi,\psi}$. It is a 
generalization of \eqref{eq:CC.prop.d} which corresponds to $\phi=\C$ and 
$\psi(w)=cw^p$ with $p>1$. Then $\psi_\C$ can be estimated below by $c_p\sfC$. 

It is easy to see that $\Xi$ is even and increasing on $[0,\infty)$ as 
$s\mapsto w \phi(s/w)$ is so for each $w>0$. As $\psi$ is increasing and 
$w\mapsto w\phi(s/w)$ is decreasing (as $s\phi'(s) \geq \phi(s)$)  
we have the lower estimate
\[
\Xi (s) \geq \min\{ w_* \phi(s/w_*), \psi(w_*)\} \quad \text{for all } w_*>0, 
\]
and it remains to choose $w_*$ appropriate for each $s$. 

The superlinearity of $\psi$ provides for each $M>1$ a $w_M\geq 1$ such that 
$\psi(w_M)\geq M w_M$.  For $s_M=M^{1/2} w_M$ and using $\CC(s|w)=w\C(s/w) $ 
we thus obtain
\[
\frac{\Xi(s_M)}{s_M} 
\geq \min\Big\{\frac{w_M \phi(s_M/w_M)}{s_M}, \frac{\psi(w_M)}{s_M} \Big\} 
= \min\Big\{\frac{\phi(M^{1/2})}{M^{1/2}}, M^{1/2} \Big\} \to \infty  
\]
for $M\to \infty$, which implies $s_M\to \infty$. As $\Xi$ is increasing on 
$[0,\infty)$, the desired superlinearity of $\Xi$ and 
Lemma \ref{lem:Superlinear} are established.

\section{Gagliardo-Nirenberg}
\label{app:GagliaNirenb}

To prove Proposition \ref{prop:integrability}, we will use a variant of the Gagliardo-Nirenberg estimate 
handling spatial and temporal integrability according to the 
a priori estimates from the $\rmL^\infty$ bound for the energy and the $\rmL^2$
bound for the dissipation. We will use the classical 
dimension-dependent Gagliardo-Nirenberg estimate 
\[
\|u\|_{\rmL^q(\bbT^d)} \leq C_{q,d} \| u\|_{\rmL^2(\bbT^d)}^{1-\theta_q} 
  \|u\|_{\rmH^1(\bbT^d)}^{\theta_q} \quad
\text{ with } \theta_q = \frac d2- \frac dq,
\]
where $q\in [2,\infty)$ and $(d{-}2)q\leq 2d$. With this, 
we obtain for
$\alpha>0$ and $r\geq 1$ with $\alpha r\geq 2$ and $(d{-}2)\alpha r\leq 2d$
the estimate   
\begin{equation}
    \label{eq:SpecialGN}
\begin{aligned}
\iint_\OmegaT \! u^\alpha v \dd x \dd t &\leq \int_0^T\!\! \|
u(t)\|_{\rmL^{\alpha r}(\bbT^d)}^\alpha \| v(t)\|_{\rmL^{r'}(\bbT^d)} \dd t 
  \\ &
\leq C^\alpha_{\alpha r,d} \| v\|_{\rmL^\infty([0,T];\rmL^{r'}(\bbT^d))} \| 
u\|^{\alpha(1-\theta_{\alpha r})}_{\rmL^\infty([0,T];\rmL^2(\bbT^d))} \int_0^T\! \|u(t)\|^{\alpha
    \theta_{\alpha r}}_{\rmH^1(\bbT^d)}  \dd t,
\end{aligned}
\end{equation}
where $r' = r/(r{-}1)$ is the dual exponent of $r$.

\bibliographystyle{alpha_AMs}
\bibliography{bibliography}

\end{document}